\documentclass[12pt]{amsart} 
\usepackage{amsthm,amsbsy,amsfonts,amssymb,amscd}

\usepackage[mathscr]{euscript} 

\usepackage{enumerate} 
\usepackage{tikz} 
\usetikzlibrary{calc,intersections,through, backgrounds,arrows,decorations.pathmorphing,backgrounds,positioning,fit,petri} 
\usetikzlibrary{calc}

\usepackage{tikz-cd} 
\usepackage[all]{xy}
\xyoption{knot}
\xyoption{poly}

\usepackage{hyperref}

\usepackage[margin=1in]{geometry}

\newcommand{\rad}{{\rm rad}} 

\newcommand{\cF}{{\mathcal{F}}}

\newcommand{\TL}{{\mathrm{TL}}}

\newcommand{\Span}{{\rm{Span}}}

\newcommand{\cM}{\mathcal{M}}

\newcommand{\tto}{\twoheadrightarrow}

\newcommand{\cB}{\mathcal B}
\newcommand{\cO}{\mathcal O}

\newcommand{\cP}{\mathcal{P}}

\newcommand{\cR}{\mathcal{R}}

\newcommand{\im}{{\rm{im}}}
\newcommand{\id}{\mathrm{id}}
\newcommand{\Id}{\mathrm{Id}}

\newcommand{\Top}{\mathrm{Top}}

\newcommand{\Hom}{\mathrm{Hom}} 
\newcommand{\Ext}{\mathrm{Ext}}

\newcommand{\End}{\mathrm{End}}

\newcommand{\dfs}{{/\kern-2pt/}}

\newcommand{\JJJ}{\mathscr{J}}
\newcommand{\op}{{\rm op}}
\newcommand{\mk}{\Bbbk}
\newcommand{\cC}{\mathcal{C}}
\newcommand{\cL}{\mathcal{L}}
\newcommand{\cD}{\mathcal{D}}
\newcommand{\cS}{\mathcal{S}}
\newcommand{\cA}{\mathcal{A}}
\newcommand{\Ob}{{\rm{Ob}}}

\newcommand{\FI}{{\rm{FI}}}
\newcommand{\charr}{{\rm{char}}}

\newcommand{\FFF}{\overline{\mathcal{F}}}

\newcommand{\fn}{\mathfrak{n}}
\newcommand{\fg}{\mathfrak{g}}
\newcommand{\fh}{\mathfrak{h}}
\newcommand{\fb}{\mathfrak{b}}
\newcommand{\mN}{\mathbb{N}}
\newcommand{\mC}{\mathbb{C}}
\newcommand{\mZ}{\mathbb{Z}}
\newcommand{\mS}{\mathbb{S}}

\numberwithin{equation}{section}

\newcommand{\cQ}{\mathcal{Q}}

\newcommand{\cJ}{\mathcal{J}}

\newtheoremstyle{notes} {} {} {} {} {\bfseries} {.} {.5em} {}

\theoremstyle{plain}

\newtheorem{prop}[subsubsection]{Proposition}

\newtheorem{lemma}[subsubsection]{Lemma}

\newtheorem{cor}[subsubsection]{Corollary}

\newtheorem{thm}[subsubsection]{Theorem}

\newtheorem{thmA}{Theorem}

\newtheorem{CorA}[thmA]{Corollary}
\newtheorem{QueA}[thmA]{Question}

\theoremstyle{remark}

\newtheorem{rem}[subsubsection]{Remark} 

\newtheorem{ddef}[subsubsection]{Definition} 

\swapnumbers

\usepackage{etoolbox}

\makeatletter
\pretocmd{\appendix}{\addtocontents{toc}{\protect\addvspace{10\p@}}}{}{}
\makeatother

\theoremstyle{remark}

\newtheorem{ex}[subsubsection]{Example}

\newtheoremstyle{construction} {} {} {} {} {\bfseries} { } {0pt} {}

\theoremstyle{construction}

\title[Borelic pairs]{Borelic pairs for stratified algebras}

\author{Kevin~Coulembier}

\author{Ruibin~Zhang}

\newcommand{\ind}{{\rm Ind}}
\newcommand{\res}{{\rm Res}}


\begin{document} 
\date{} 
\begin{abstract}
We determine all values of the parameters for which the cell modules form a standard system, for a class of cellular diagram algebras including partition, Brauer, walled Brauer, Temperley-Lieb and Jones algebras. 
For this, we develop and apply a general theory of finite dimensional algebras with Borelic pairs. The theory is also applied to give new uniform proofs of the cellular and quasi-hereditary properties of the diagram algebras and to construct quasi-hereditary 1-covers, in the sense of Rouquier, with exact Borel subalgebras, in the sense of K\"onig.
Another application of the theory leads to a proof
 that Auslander-Dlab-Ringel algebras admit exact Borel subalgebras.

	\end{abstract}

\maketitle 

\tableofcontents

\section{Introduction}

The theory of quasi-hereditary algebras was initiated by Cline, Parshall and Scott in~\cite{CPS, Scott} and provided a unified framework for studying modular representation theory of semisimple algebraic groups and the BGG category~$\cO$. A central role in this theory is played by the {\em standard modules} of quasi-hereditary algebras. 
Cellular algebras were introduced by Graham and Lehrer in~\cite{CellAlg} and include many {\em diagram algebras} such as Brauer, Iwahori-Hecke and BMW algebras. These algebras are of central importance in representation theory and low dimensional topology. Cellular algebras have a class of natural modules, known as the {\em cell modules}. For many values of their parameters, Brauer and BMW algebras are also quasi-hereditary and then the standard modules coincide with the cell modules. This remains true for a range of cellular algebras.

In \cite{Nakano1, Nakano2}, it is proved that, in most cases, the cell modules of Iwahori-Hecke algebras of type $A$ behave as the standard modules of some quasi-hereditary algebra, even though the Iwahori-Hecke algebra is itself not quasi-hereditary. This is formulated into the statement that ``the cell modules form a {\em standard system}'', see~\cite{DR}, and is equivalent to the condition that the algebra admits a {\em cover-Schur algebra}, in the sense of \cite{HHKP, Rou}. The result in \cite{Nakano1} extends to Brauer algebras, by~\cite{Paget}, to partition and BMW algebras, by~\cite{HHKP}, and to Iwahori-Hecke algebras of type $B$, by \cite{Rou}. An important consequence is that multiplicities in cell filtrations are well-defined.

In the present paper, we determine when the cell modules form a standard system for a variety of diagram algebras: partition, Brauer, walled Brauer, Temperley-Lieb and Jones algebras. In the cases of the Brauer and partition algebras, this reproduces (and completes) the corresponding results of~\cite{Paget, HHKP}. We will achieve this by developing a construction which is in part inspired by the theory of exact Borel subalgebras.

In \cite{Koenig}, K\"onig introduced the notion of an {\em exact Borel subalgebra} of a quasi-hereditary algebra, inspired by the Borel subalgebra of a semisimple Lie algebra. When an exact Borel subalgebra exists, the standard modules are given explicitly as the modules induced from the simple modules of that subalgebra. Not every quasi-hereditary algebra admits an exact Borel subalgebra. However, it was proved in~\cite{KKO} that every quasi-hereditary algebra over an algebraically closed field is Morita equivalent to a (quasi-hereditary) algebra admitting an exact Borel subalgebra. Although the proof is constructive, it would not be trivial to apply it to construct the Morita equivalent algebra and its exact Borel subalgebra for a given quasi-hereditary algebra.
In fact, few explicit examples of Morita equivalent algebras with exact Borel subalgebras for prevalent quasi-hereditary algebras are known, see {\it e.g.} \cite{Koenig, PSW}.

We will define a generalisation of the concept of exact Borel subalgebras, {\it viz.} {\em Borelic pairs $(B,H)$ of arbitrary algebras} $A$, and develop the theory of algebras with such pairs. We prove that key structural properties of~$A$ are determined by those of~$H$. 
For instance, if~$H$ is semisimple then $A$ is quasi-hereditary, and, roughly speaking, cellularity of~$H$ implies cellularity of~$A$. Furthermore, in the latter case, the cell modules of~$A$ form a standard system if and only if this is the case for~$H$. These and further applications are more precisely discussed in Section~\ref{IntroBP}.

We use this general theory to study the diagram algebras mentioned above, in particular for (i) determining when the cell modules form standard systems and~(ii) constructing Morita equivalent algebras with exact Borel subalgebras. For any of the diagram algebras $A$, we construct an algebra~$C$, satisfying a double centraliser property with~$A$. The algebra~$C$ also admits a Borelic pair, which allows to prove that it is cellular, determine when the cell modules form a standard system, and find out when it is quasi-hereditary. In the latter case,~$C$ admits an exact Borel subalgebra. In most cases, $C$ is Morita equivalent to~$A$ and in the remaining cases, it is a $1$-faithful cover in the sense of~\cite{Rou}. In some cases, $C$ constitutes a previously unknown `Schur algebra'. To the best of our knowledge, this is the first purely diagrammatic description of a Schur algebra. Aside from the new results on the algebras $A$, our constructions also yields a new proof of their cellularity and quasi-heredity.

\subsection{Standardly based and base stratified algebras}
In \cite{JieDu}, Du and Rui introduced a generalisation of cellular algebras, called {\em standardly based algebras}. We will prove that any algebra over an algebraically closed field can be given at least one standardly based structure. Standardly based algebras also come with cell modules. When the standardly based algebra is cellular, the two types of cell modules coincide. 

In \cite{HHKP}, Hartman {\it et al.} defined the concept of {\em cellularly stratified algebras}, giving a powerful tool to determine when cell modules form a standard system. We introduce the corresponding weaker notion of {\em base stratified algebras}, which is a significant simplification that is better compatible with the notion of Borelic pairs, and actually suffices for studying when the cell modules of a cellular algebra form a standard system. This generalisation will also allow the periplectic Brauer algebra, which is a non-cellular diagram algebra, to be studied using our techniques, see \cite{peri}.

\vspace{-1mm}
\begin{figure}[!h]\label{tabb}
\caption{Links between structures}
\[
\xymatrix{
&& *+[F]\txt{stratified}  \\
 *+[F]\txt{standardly\\ based} &*+[F]\txt{exactly stratified}\ar[ur]&&*+[F]\txt{standardly stratified}\ar[ul]\\
&&*+[F]\txt{ exactly\\ standardly stratified}\ar[ul]\ar[ur]&*+[F]\txt{strongly\\ standardly stratified}\ar[u]\\
*+[F]\txt{base\\stratified}\ar[uu]\ar[urr]&&*+[F]\txt{properly stratified}\ar[u]\ar[ur]\\
&&*+[F]\txt{quasi-hereditary}\ar[u]\ar[ull]
}
\]
\end{figure}
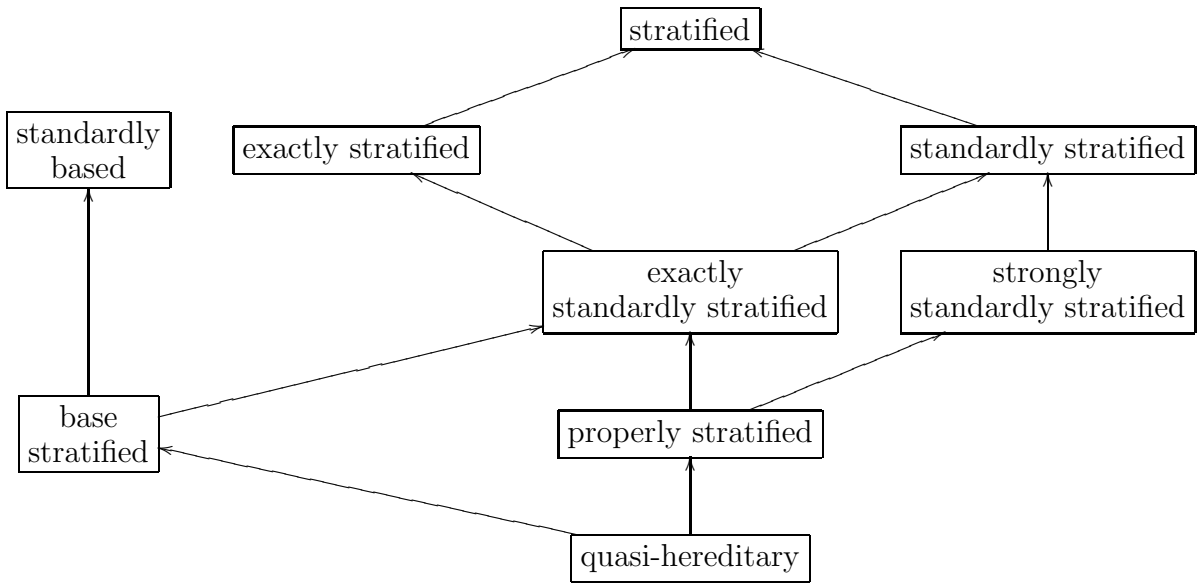
\vspace{-1mm}

\subsection{Overview of the structures on (finite dimensional) algebras}
The different structures on algebras we study are summarised in Figure 1. {\em Properly stratified algebras} were introduced in~\cite{Dlab, Kluc}.
The notion of {\em (standardly) stratified algebras} were introduced in~\cite[Chaper~2]{CPSbook}. The notion of {\em strongly standardly stratified algebras} was studied in~\cite{ADL}, although the term``standardly stratified'' was used, leading to inconsistency with \cite{CPSbook}. The term ``strongly standardly stratified algebras'' was coined in \cite{Frisk}, which also introduced {\em exactly standardly stratified algebras}, as ``weakly properly stratified algebras''.

An arrow means that one structure implies the other one.
For arrows between different versions of standardly stratified algebras, the (proper) standard modules remain the same. The arrow from quasi-hereditary to base stratified only holds when the field is algebraically closed and in this case the standard modules are mapped to the cell modules.
A general base stratified algebra has standard modules, cell modules and proper standard modules. The first type of modules are different from the other two, unless the algebra is quasi-hereditary.

\subsection{Borelic pairs}\label{IntroBP}
 We generalise the notion in \cite{Koenig} of exact Borel subalgebras from quasi-hereditary algebras to exactly standardly stratified algebras, where the generalisation to properly stratified algebras was already introduced by Klucznik and Mazorchuk in~\cite{Kluc}.

We further generalise this theory by introducing the notion of {\em Borelic pairs} $(B,H)$ of an arbitrary finite dimensional algebra~$A$. We prove that, when~$A$ admits a Borelic pair~$(B,H)$, it is standardly stratified. Moreover, if~$H$ is quasi-local, $A$ is strongly standardly stratified. We also consider a stronger notion of {\em exact Borelic pairs}, leading to exactly standardly stratified algebras. Then we find that, if~$H$ is quasi-local (resp. semisimple) $A$ is properly stratified (resp. quasi-hereditary) and in each case~$B$ is an exact Borel subalgebra. If $H$ is standardly based, $A$ is base stratified, so in particular standardly based. Furthermore, the cell modules of~$H$ form a standard system if and only the standard modules of~$A$ form a standard system.

\subsection{Overview of main applications} Since we will almost exclusively deal with {\em finite dimensional, unital, associative} algebras, we don't mention these characteristics in the following statements.

\begin{thmA}\label{ThmA}
Over a perfect field $\mk$, the Auslander-Dlab-Ringel algebra of a $\mk$-algebra~$R$ is quasi-hereditary with exact Borel subalgebra.
\end{thmA}
This will be proved in Theorem~\ref{ThmAus}. The weak assumption that~$\mk$ is perfect is not required for the quasi-heredity of ADR algebras, see \cite{DRAus}, but our construction of the exact Borel subalgebra fails without it. This exact Borel subalgebra and the ones in Theorem~\ref{ThmC} seem unrelated to the construction in \cite{KKO}, see e.g.~Remark~\ref{NotKKO}.
The quasi-heredity of ADR algebras has the following consequence, see Theorem~\ref{AllBased}.
\begin{CorA}\label{CorB}
Any algebra over an algebraically closed field has a standardly based structure.
\end{CorA}

The following theorem, which uses a total quasi-order $\preccurlyeq$ and a partial order $\le$ on the set of simple modules, will be proved in Theorems~\ref{ThmDiagram1}, \ref{ThmMor} and~\ref{ThmCov} and Corollary~\ref{CorCell}. 
\begin{thmA}\label{ThmC}
Consider an arbitrary field $\mk$ and a fixed $\delta\in\mk$. Let $A$ be the partition algebra~$\cP_n(\delta)$, the Brauer algebra~$\cB_n(\delta)$, the walled Brauer algebra~$\cB_{r,s}(\delta)$, the Jones algebra~$J_n(\delta)$ or the Temperley-Lieb algebra~$\TL_n(\delta)$, with~$n,r,s\in\mN$, $n>1$, $r\ge 1$ and~$s\ge 1$. Then
\begin{enumerate}
\item $A$ is cellular if~$\mathscr{C}^3_A$;
\item $(A,\preccurlyeq)$ is exactly standardly stratified(*) if~$\mathscr{C}^1_A$;
\item $(A,\le)$ is quasi-hereditary if~$\mathscr{C}^1_A$ and~$\mathscr{C}^2_A$;
\item $A$ admits a cover $C$ which is:
\begin{itemize}
\item Morita equivalent to~$A$ if~$\mathscr{C}^1_A$;
\item quasi-hereditary with exact Borel subalgebra if~$\mathscr{C}^2_A$;
\item always exactly standardly stratified(*) with exact Borel subalgebra;
\item cellular and base stratified if~$\mathscr{C}^3_A$.
\end{itemize}
\end{enumerate}

\vspace{-4mm}

\begin{center}
\begin{tabular}{ | l | l | l |l |}
\multicolumn{4}{c }{}\\
\hline
algebra~$A$& $\mathscr{C}^1_A$& $\mathscr{C}^2_A$ & $\mathscr{C}^3_A$ \\ \hline\hline
$\cP_n(\delta)$ & $\delta\not=0$  &  $\charr(\mk)\not\in [2,n]$&$\emptyset$\\ \hline
$\cB_n(\delta)$ & $\delta\not=0$ or~$n$ odd  &  $\charr(\mk)\not\in [2,n]$&$\emptyset$\\ \hline
$J_n(\delta)$ & $\delta\not=0$ or~$n$ odd  & $n$ odd and~$\charr(\mk)\not\in [3,n]$ or & $x^i-1$ splits over $\mk$,\\ 
 &  &$n$ even and~$\charr(\mk)\not\in \{2\}\cup[3,n/2]$& for~$i\in\{n,n-2,\ldots\}$\\\hline
$\TL_n(\delta)$ & $\delta\not=0$  or~$n$ odd  &  $\quad\emptyset$&$\emptyset$\\ \hline
$\cB_{r,s}(\delta)$ & $\delta\not=0$ or~$r\not=s$ &  $\charr(\mk)\not\in [2,\max(r,s)]$&$\emptyset$\\ \hline
\end{tabular}
\end{center}
When $\mathscr{C}^1_A$ is satisfied, but not $\mathscr{C}^2_A$, the algebra~$A$ is not quasi-hereditary, for any order.

(*) In case~$A=J_n(\delta)$, they are furthermore properly stratified.
\end{thmA}
Here and throughout the paper, $\emptyset$ represents the `empty' condition. So condition $\emptyset$ is always satisfied.

When $\cC^1_A$ is not satisfied, it follows from \cite{CellQua} that~$A$ is not quasi-hereditary, for any order. Hence, we obtained a {\em Morita equivalent algebra with exact Borel subalgebra for all cases when~$A$ is quasi-hereditary}. 
The quasi-heredity and cellularity of~$A$ in Theorem~\ref{ThmC} were previously obtained by a variety of methods in~\cite{AST1, CDDM, CellAlg, CellQua, partition, Xi}. Here we find a new unified proof, which also constructs the exact Borel subalgebras. The quasi-heredity of~$J_n(\delta)$ has previously only been stated, in~\cite[Proposition~4.2]{CellQua}, without the explicit conditions.
Next we determine when the cell modules of~$A$ form a standard system.
\begin{thmA}\label{ThmD}
If the field $\mk$ is algebraically closed, all algebras $A$ in Theorem~\ref{ThmC} are cellular. The cell modules form a standard system if and only if the following condition is satisfied:
\vspace{-4mm}
\begin{center}
\begin{tabular}{ | l | l |  }
\multicolumn{2}{c }{}\\
\hline
algebra~$A$& condition \\ \hline\hline
$\cP_n(\delta)$ & $\delta\not=0$ if~$n=2$; and~$\begin{cases}\charr(\mk)\not\in\{2,3\}\qquad\mbox{or}\\\charr(\mk)=3\mbox{ and }n=2\end{cases}$  \\ \hline
$\cB_n(\delta)$ & $\delta\not=0$ if~$n\in \{2,4\}$; and~$\begin{cases}\charr(\mk)\not\in\{2,3\}\qquad\mbox{or}\\\charr(\mk)=3\mbox{ and }n=2\end{cases}$\\ \hline
$J_n(\delta)$ & $\delta\not=0$ if~$n\in \{2,4\}$; and~$\begin{cases} \charr(\mk)\not\in [3,n] & \mbox{if~$n$ is odd} \\ 
  \charr(\mk)\not\in\{2\}\cup[3,n/2]&\mbox{if~$n$ is even}\end{cases}$\\\hline
  $\TL_n(\delta)$ & $\delta\not=0$ if~$ n\in \{2,4\}$ \\ \hline
$\cB_{r,s}(\delta)$ &$\delta\not=0$ if~$(r,s)\in\{(1,1),(2,2)\}$; and~$\begin{cases}\charr(\mk)\not\in\{2,3\}\qquad\mbox{or}\\\charr(\mk)=3\mbox{ and }\max(r,s)\le 2\end{cases}$  \\ 
\hline
\end{tabular}
\end{center}

\end{thmA}
 This follows from Theorems~\ref{ThmDiagram2}, \ref{ThmMor} and~\ref{ThmCellStan}. The results for~$\cB_n(\delta)$ were previously obtained in~\cite{Paget} and most of the claim for~$\cP_n(\delta)$ was proved in~\cite{HHKP}. This theorem, together with \cite[Theorem~2]{DR} and Lemma~\ref{LemSchur}, yields the following consequence.

\begin{CorA}\label{CorE}
The cellular algebra~$A$, under the condition in Theorem~\ref{ThmD}, admits a Schur algebra in the sense of \cite[Definition~12.1]{HHKP} or Definition~\ref{DefSchur}. Moreover, the cell multiplicities in modules which admit cell filtrations are independent of the specific filtration.
\end{CorA}

When the condition~$\cC^1_A$ in Theorem~\ref{ThmC} is not satisfied and the cover $C$ is hence not Morita equivalent to~$A$, in most cases it is still a $1$-faithful cover. This yields a quasi-hereditary~$1$-cover in the sense of~\cite{Rou} when~$C$ is quasi-hereditary ({\em i.e.} when~$\cC^2_A$ is satisfied).

\begin{thmA} \label{ThmF}
Consider an arbitrary field $\mk$ and the cover $C$ of~$A$ in Theorem~\ref{ThmC}, then $C$ is a 1-faithful quasi-hereditary cover if~$\mathscr{C}^2_A$ and the condition in the table is satisfied.
\vspace{-4mm}
\begin{center}
\begin{tabular}{ | l | l |  }
\multicolumn{2}{c }{}\\
\hline
algebra~$A$& condition (along with~$\cC_A^2$) \\ \hline\hline
$\cP_n(\delta)$ & $n\not=2$ or~$\delta\not=0$  \\ \hline
$\cB_n(\delta)$ & $n\not\in \{2,4\}$ or~$\delta\not=0$ \\ \hline
$J_n(\delta)$ & $n\not\in \{2,4\}$ or~$\delta\not=0$  \\ \hline

  $\TL_n(\delta)$ & $n\not\in \{2,4\}$ or~$\delta\not=0$  \\ \hline
$\cB_{r,s}(\delta)$ &  $(r,s)\not\in\{(1,1),(2,2)\}$ or~$\delta\not=0$\\ 
\hline
\end{tabular}
\end{center}
Under these conditions, the algebra~$C$ is the Schur algebra predicted in Corollary~\ref{CorE}.
\end{thmA}
This is proved in Theorems \ref{Thm0Cover} and \ref{ThmMor}. 
If, for~$A=\cB_n(\delta)$, we take the stronger condition `$n$ even or~$\delta\not=0$', but relax condition~$\cC^2_A$ to `$\charr(\mk)\not\in\{2,3\}$', the Schur algebra of Corollary~\ref{CorE} is constructed in a more combinatorial way in~\cite{Henke}, see also \cite{Bowman} for the walled Brauer algebra. The combination of our methods with the ones in \cite{Henke} can be used to construct the Schur algebras of Corollary~\ref{CorE} in full generality, see Remark~\ref{SchurGen}. The result in Theorem~\ref{ThmF} also leads to the following question. The answer is likely to grow with $n,r,s$.
\begin{QueA}
What is the maximal~$k\in\mN$ for which the quasi-hereditary covers in Theorem~\ref{ThmF} are $k$-faithful?
\end{QueA}

Finally, we remark that our treatment of Theorems \ref{ThmA}, \ref{ThmC} and \ref{ThmD} provide alternative proofs for several results in~\cite{AST1, CDDM, DRAus, CellAlg, HHKP, Paget, CellQua, partition, Xi}. Our approach does not rely on any of the results in these papers.

\subsection{Structure of the paper} In Section~\ref{SecPrel}, we recall some useful results and introduce some notation. Sections \ref{SecPre} to \ref{SecBase} constitute Part~\ref{GenThe}. Here, the general theory of borelic pairs is developed. In Part~\ref{Examp} we apply these results to a variety of examples. Concretely, in Section~\ref{SecADR} we consider Auslander-Dlab-Ringel algebras, in Section~\ref{SecBGG} a thick version of the BGG category~$\cO$ and in Section~\ref{SecDia} the class of diagram algebras.
In Part~\ref{Faith} (Section~\ref{SecCov}) we determine the precise homological connection between the diagram algebras and their covers constructed in Section~\ref{SecDia}. In Appendix~\ref{TheApp}, we recall some known technical properties of stratified algebras, which are not easy to find with proof in the literature.

\section{Preliminaries}\label{SecPrel}

We work over an arbitrary ground field $\mk$. Unless specified otherwise, we make no assumptions on its characteristic and do not require it to be algebraically closed. A field is called {\em perfect} if every algebraic field extension of~$\mk$ is separable. Examples are algebraically closed fields, fields of characteristic zero and finite fields. We set $\mN=\{0,1,2,\cdots\}$.

\subsection{Algebras and modules}
By an algebra~$A$ over $\mk$, we will always mean an {\em associative, unital, finite dimensional} algebra. All subalgebras of an algebra~$A$ are assumed to contain the identity $1_A$.  Modules over $A$ are supposed to be unital, in the sense that~$1_A$ acts as the identity. Furthermore, they are always assumed to be left modules and finitely generated, unless explicitly specified otherwise. The corresponding module category is denoted by~$A$-mod. The Jordan-H\"older multiplicity of a simple module~$L$ in $M\in A$-mod will be denoted by $[M:L]$.

Two algebras $A$ and~$B$ are {\em Morita equivalent} if there is an equivalence of categories 
$$A\mbox{-mod }\cong\; B\mbox{-mod}.$$
 We write this as $A\;\stackrel{M}{=}\; B$. A {\em local} algebra is an algebra with unique maximal left ideal.  We say that an algebra is {\em quasi-local} if it is Morita equivalent to the direct sum of local algebras, so if there are no extensions between non-isomorphic simple modules.

\subsection{Jacobson radical} The Jacobson radical~$\rad A$ of an algebra~$A$ is the intersection of all maximal left ideals, see \cite[Chapter~4]{Lam}. An algebra~$A$ is called {\em semisimple} if its Jacobson radical is zero, which is equivalent to the condition that~$A$ is the direct sum of simple algebras by \cite[Theorem 4.14]{Lam}.
All finite dimensional unital algebras are {\em semiprimary}, meaning that~$A/\rad A$ is semisimple and~$\rad A$ is nilpotent, see \cite[Theorem 4.15]{Lam}.

A $\mk$-algebra~$A$ is called {\em separable} if for any field extension~$\mathbb{K}$ of~$\mk$, $A\otimes_{\mk}\mathbb{K}$ is a semisimple $\mathbb{K}$-algebra. Every simple algebra over a perfect field is separable. 
If $A$ has a semisimple subalgebra~$S$, such that~$A=S\,\oplus \,\rad A$ as vector spaces, we say that~$A$ is a {\em Wedderburn algebra.} Wedderburn's principal theorem states that any algebra such that~$A/\rad A$ is separable is Wedderburn, see \cite[Exercise~9.3.1]{Weibel}.

For a subspace $I\subset A$ (not containing $1_A$) and an $A$-module $M$, we denote by 
\begin{equation}\label{eqInvariants}M^I=\{v\in M\,|\, xv=0,\quad\mbox{ for all $x\in I$}\},\end{equation}
the $I$-invariants of~$M$. In case~$I$ is a two-sided ideal, the space $M^I$ comes with a natural~$A/I$-module structure.

\subsection{Idempotents}\label{SecIdem}
An {\em idempotent}~$e\in A$ is {\em primitive} if the left $A$-module~$Ae$ is indecomposable. Two idempotents~$e$ and~$f$ in $A$ are {\em equivalent} if and only if~$Ae\cong Af$ as left $A$-modules, or equivalently if there are $a,b\in A$ such that~$e=ab$ and~$f=ba$.

For an algebra~$A$, there is a one-to-one correspondence between the isomorphism classes of simple unital modules, the equivalence classes of primitive idempotents and the isomorphism classes of projective covers of the simple modules. So there is a (finite) set $\Lambda=\Lambda_A$ with corresponding simple module~$L(\lambda)$, primitive idempotent~$e_\lambda$ and projective module~$P(\lambda):=A e_\lambda$ for each~$\lambda\in \Lambda_A$, exhausting the classes irredundantly, such that
$$e_{\lambda}L(\lambda')\not=0\quad\Leftrightarrow\quad\lambda'=\lambda\quad\Leftrightarrow\quad\Hom_A(P(\lambda), L(\lambda'))\not=0.$$

\subsection{Centraliser algebras}\label{SecCover}
For an algebra~$C$ with idempotent~$e$, the centraliser algebra of~$e$ is $C_0=eCe$. We consider the
pair of adjoint functors $(F,G)$: 
\begin{equation}\label{eq630}
\xymatrix{ 
C\text{-}\mathrm{mod}\ar@/^/[rrrrrr]^{F=eC\otimes_C-\,\cong\, e-\,\cong\,\mathrm{Hom}_C(Ce,-)}&&&&&& 
C_0\text{-}\mathrm{mod}.\ar@/^/[llllll]^{G=\Hom_{C_0}(eC,-)}
}
\end{equation}
We have $F\circ G\cong \Id$ on $C_0$-mod, and~$F$ is exact while $G$ is left exact.

\subsubsection{}\label{SecSecCover} We say that~$C$ is a {\em cover} of~$C_0$ if the restriction of~$F$ to the category of projective $C$-modules is fully faithful. In other words, the canonical morphism 
$$C\,\stackrel{\sim}{\to}\, \End_{C}(C)^{\op}\,\stackrel{F}{\to}\, \End_{C_0}(e C)^{\op}; \qquad c\mapsto\alpha_c,\quad\mbox{with}\quad \alpha_c(x)=xc\;\;\mbox{ for all $x\in eC$,}$$ is an isomorphism if and only if $C$ is a cover.

\subsubsection{}\label{SecMor} Assume that~$C=CeC$, then every primitive idempotent of~$C$ is equivalent to one contained in~$eCe$. Hence $Ce$ is a projective generator for~$C$-mod, so it follows that~$F$ is faithful. Since we already had $F\circ G\cong \Id$, it follows that~$F$ is an equivalence, so $C$ and~$C_0$ are Morita equivalent.

\subsection{Orders and partitions}

A {\em partial quasi-order} (also known as pre-order) $\preccurlyeq$ is a binary relation which is reflexive and transitive. When a partial quasi-order is also anti-symmetric, it is a {\em partial order}, and will usually be denoted by $\le$. When $\preccurlyeq$ satisfies the condition that, for all elements $s,t$, at least one of~$s\preccurlyeq t$ or~$t\preccurlyeq s$ holds, it is a {\em total quasi-order}, or simply a {\em quasi-order}. 

For a partial quasi-order $\preccurlyeq$, we use the notation~$s\prec t$ when~$s\preccurlyeq t$ but $t\not\preccurlyeq s$. We also write $s\sim t$ when~$s\preccurlyeq t$ and~$t\preccurlyeq s$. An {\em extension} $\preccurlyeq_e$ of a partial quasi-order $\preccurlyeq$ is a partial quasi-order such that~$s\preccurlyeq t$ implies $s\preccurlyeq_et$, and such that~$s\sim t$ if and only if~$s\sim_e t$. In particular, an extension of a partial order remains a partial order.

An {\em $n$-decomposition~$\cQ$ of a set} $S$ is an ordered disjoint union~$S=\sqcup_{i=0}^n S_i$ into~$n+1$ subsets. In case the set $S$ is finite, decompositions are in natural bijection with total quasi-orders. For the decomposition~$\cQ$ the corresponding total quasi-order $\preccurlyeq_{\cQ}$ on $S$ is given by~$s\preccurlyeq_{\cQ} t$ if and only if~$s\in S_i$ and~$t\in S_j$ with~$i\ge j$. To each decomposition~$\cQ$, we also associate a partial order $\le_{\cQ}$, defined by~$s<_{\cQ} t$ if and only if~$s\in S_i$ and~$t\in S_j$ with~$i> j$. Note that we cannot obtain every partial order in this way.

We will use the term {\em partition} in two different situations. A {\em partition of~$n\in\mN$} is a (non-strictly) decreasing sequence of natural numbers adding up to~$n$. When~$\lambda$ is a partition of~$n$ we write ~$\lambda\vdash n$. For~$p>0$, we say that~$\lambda\vdash n$ is $p$-regular if the sequence of the partition does not contain $p$ times the same non-zero number. For a field~$\mk$, we say that~$\lambda$ is $\mk$-regular if either $\charr(\mk)=0$ or~$\lambda$ is $\charr(\mk)$-regular, and write this as ~$\lambda\vdash_{\mk} n$.  Secondly, we will use the notion of a {\em partition of a set}, which is a grouping of the elements into non-empty subsets (or equivalence classes).

\subsection{Stratifying ideals}\label{SecStratId}
An idempotent ideal~$J$ in an algebra~$A$ is a two-sided ideal which satisfies $J^2=J$. This means $J=AeA$ for some idempotent~$e$, see {\it e.g.}~\cite[p673]{APT}. 
This idempotent~$e$ is not uniquely determined by~$J$. However, it follows easily, see {\it e.g.}~\cite[p673]{APT}, that~$AeA=A\tilde{e}A$, for two idempotents $e$ and~$\tilde e$, implies
\begin{equation}\label{MoreAe0}eAe\;\stackrel{M}{=}\; \tilde{e}A\tilde{e}.\end{equation}

\subsubsection{}\label{SecStratId2} We introduce some specific types of stratifying ideals, using in particular terminology from \cite{CPS, CPSbook, Frisk, Kluc}.  An idempotent ideal~$J=AeA$ in $A$ is
\begin{enumerate}
\item[(0)]  an {\em exactly stratifying ideal} if:
\begin{itemize}
\item the right $A$-module~$J_A$ is projective.
\end{itemize}
\item  a {\em standardly stratifying ideal} if:
\begin{itemize}
\item the $A$-module~${}_AJ$ is projective.
\end{itemize}
\item  a {\em exactly standardly stratifying ideal} if:
\begin{itemize}
\item the left $A$-module~${}_AJ$ and right $A$-module~$J_A$ are projective.
\end{itemize}
\item  a {\em strongly standardly stratifying ideal} if:
\begin{itemize}
\item the $A$-module~${}_AJ$ is projective and
\item the algebra~$eAe$ is quasi-local.\end{itemize}
\item  a {\em properly stratifying ideal} if:
\begin{itemize}
\item the left $A$-module~${}_AJ$ and right $A$-module~$J_A$ are projective and
\item the algebra~$eAe$ is quasi-local.\end{itemize}
\item  a {\em heredity ideal} if:
\begin{itemize}
\item the $A$-module~${}_AJ$ is projective and
\item the algebra~$eAe$ is semisimple.
\end{itemize}
\end{enumerate}
All these ideals are stratifying in the sense of~\cite[Definition~2.1.1]{CPSbook}, by~\cite[Remark~2.1.2(b)]{CPSbook}. A heredity ideal is properly stratifying by \cite[Corollary 5.3]{APT}, the other relations in Figure 1 are by definition. 

\begin{rem}\label{MoreAe}
By equation~\eqref{MoreAe0}, the definitions depend only on $J$ and not on $e$.
\end{rem}

\subsubsection{}\label{SecChain}We say that a chain of idempotent ideals
\begin{equation}
\label{stratchain}
0=J_0\subsetneq J_1\subsetneq \cdots \subsetneq J_{m-1}\subsetneq J_m =A,
\end{equation}
has length $m-1$. With this convention, the trivial chain $0=J_0\subset J_1=A$ has length~$0$. Chains of length $k$ are in one-to-one correspondence with~$k$-decompositions of~$\Lambda$. The decomposition corresponding to a chain of idempotent ideals \eqref{stratchain} is defined by setting~$\lambda\in\Lambda_i$ for the minimal~$i$ for which $J_{i+1} L(\lambda)\not=0$. Consequently, the chains of idempotent ideals of~$A$ are also in one-to-one correspondence with the total quasi-orders of~$\Lambda$.

\subsubsection{}If for a chain \eqref{stratchain}, each ideal~$J_{i}/J_{i-1}$ is a standardly, exactly standardly, strongly standardly or properly stratifying ideal in $A/J_{i-1}$, the chain is also called standardly, exactly standardly, strongly standardly or properly stratifying. If each ideal~$J_{i}/J_{i-1}$ is a heredity ideal, the chain is called heredity.

We will take the convention of denoting the image of an idempotent~$e\in A$, in the quotient~$A/J$, for some idempotent ideal~$J$ again by~$e$.

\begin{rem}\label{DefAi}
For a chain \eqref{stratchain}, we can choose idempotents $f_i$ for~$i\in\{1,\cdots,m\}$ such that~$J_i=Af_iA$, with~$f_if_j=f_i=f_jf_i$ if~$i\le j$ and~$f_m=1$.
 We define the algebras $$A^{(j)}=f_{j+1}(A/J_{j})f_{j+1},\quad\mbox{ for }\quad j\in\{0,1,\ldots, m-1\},$$ which are only uniquely associated to the chain up to Morita equivalence, by \eqref{MoreAe0}. 
\end{rem}

\subsection{Standardly stratified algebras}\label{SecStrat}
An algebra~$A$ with some partial quasi-order $\preccurlyeq$ on $\Lambda$ will be denoted as $(A,\preccurlyeq)$. 

\begin{ddef}\label{DefA1}${}$
Consider an algebra~$(A,\preccurlyeq)$, for~$\preccurlyeq$ a total quasi-order, and the chain \eqref{stratchain} of idempotent ideals corresponding to~$\preccurlyeq$. We say that~$(A,\preccurlyeq)$ is 
\begin{enumerate}
\item {\em standardly stratified}, if the chain \eqref{stratchain} is standardly stratifying;
\item {\em exactly standardly stratified}, if the chain \eqref{stratchain} is exactly standardly stratifying;
\item {\em strongly standardly stratified}, if~$\preccurlyeq$ is a total order and the chain \eqref{stratchain} is strongly standardly stratifying;
\item {\em properly stratified}, if~$\preccurlyeq$ is a total order and the chain \eqref{stratchain} is properly stratifying;
\item {\em quasi-hereditary}, if~$\preccurlyeq$ is a total order and the chain \eqref{stratchain} is heredity.
\end{enumerate}
\end{ddef}
\begin{rem}\label{RemTriv}
The trivial chain $0=J_0\subset J_1=A$ is an exactly standardly stratifying chain for any algebra~$A$. It is thus essential to specify the quasi-order $\preccurlyeq$ when speaking about types of standardly stratified algebras.
\end{rem}

\begin{rem}\label{RemSSS}
If $\preccurlyeq$ is a total order and the chain \eqref{stratchain} corresponding to~$\preccurlyeq$ is standardly (resp. exactly standardly) stratifying, it is automatically strongly standardly (resp. properly) stratifying, as the algebra~$A^{(i)}$ will only have one simple module up to isomorphism.
\end{rem}

There is also a module theoretic approach to standardly stratified algebras, where we no longer demand the orders to be total. The equivalence between both definitions (when the order is total) is well-known by e.g. \cite{CPS, CPSbook, Dlab, Frisk}, see also Appendix~\ref{AppEq}.
\begin{ddef}\label{DefA2}
Consider an algebra~$(A,\preccurlyeq)$ for some partial quasi-order $\preccurlyeq$ on $\Lambda$. For any $\lambda\in\Lambda$, let $L(\lambda)$ denote the corresponding simple $A$-module and~$P(\lambda)$ its projective cover.
\begin{enumerate}
\item If there is a set of~$A$-modules $\{S(\lambda),\lambda\in \Lambda\}$ such that for any~$\lambda,\mu\in\Lambda$:
\begin{itemize}
\item we have $[S(\lambda):L(\mu)]=0$ unless $\mu\preccurlyeq \lambda$, and
\item there is a surjection~$P(\lambda)\tto S(\lambda)$ such that the kernel has a filtration where the section are isomorphic to modules $S(\nu)$ for~$\lambda \prec\nu$,
\end{itemize}
we say that~$(A,\preccurlyeq)$ is {\em standardly stratified.}
\item If $(A,\preccurlyeq)$ is standardly stratified and there are also modules $\{\overline{S}(\lambda),\lambda\in \Lambda\}$ with~$[\overline{S}(\lambda):L(\lambda)]=1$ and~$[\overline{S}(\lambda):L(\mu)]=0$ unless $\mu=\lambda$ or~$\mu\prec\lambda$, such that each module~$S(\lambda)$ has a filtration where the sections are isomorphic to modules $\overline{S}(\mu)$ with~$\mu\sim\lambda$, we say that~$(A,\preccurlyeq)$ is {\em exactly standardly stratified.}
\item If $\preccurlyeq$ is a partial order and~$(A,\preccurlyeq)$ is standardly stratified as in (1), we say that~$(A,\preccurlyeq)$ is {\em strongly standardly stratified.}
\item If $\preccurlyeq$ is a partial order and~$(A,\preccurlyeq)$ is exactly standardly stratified as in (2), we say that~$(A,\preccurlyeq)$ is {\em properly stratified.}
\item If $\preccurlyeq$ is a partial order, $(A,\preccurlyeq)$ is standardly stratified as in (1) and in addition~$[S(\lambda):L(\lambda)]$=1 for all~$\lambda\in \Lambda$, we say that~$(A,\preccurlyeq)$ is {\em quasi-hereditary.}
\end{enumerate}
\end{ddef}
Observe that if~$(A,\preccurlyeq)$ is standardly stratified, $(A,\preccurlyeq')$  is also standardly stratified, for any extension~$\preccurlyeq'$ of~$\preccurlyeq$. 

\subsubsection{}\label{SAj}The modules $S(\lambda)$ in Definition~\ref{DefA2} have simple top $L(\lambda)$ and are known as {\em standard modules}. They satisfy~$S(\lambda)\cong Ae_\lambda/J_{j-1}e_\lambda$ with~$j$ the largest such that~$J_{j-1}L(\lambda)=0$. The modules $\overline{S}(\lambda)$ are {\em proper standard modules} and are given by~$S(\lambda)/(Ae_\lambda\rad S(\lambda))$. 

\subsection{Faithful covers and standard systems}\label{IntroFaithCov}
In \cite[Definition~4.37]{Rou}, Rouquier introduced notions of {\em faithfulness} of {\em quasi-hereditary covers}. We extend this to standardly stratified algebras.
Consider a cover $C$ as in~\ref{SecSecCover}, which is standardly stratified.
Such a cover is called {\em $j$-faithful} if the functor~$F=\Hom_C(Ce,-)$ induces isomorphisms
\begin{equation}\Ext^i_C(M,N)\;\,\tilde\to\;\, \Ext^i_{C_0}(FM,FN),\qquad\forall \;0\le i\le j,\label{ifaith}\end{equation}
for all modules $M,N$ admitting a filtration with sections given by proper standard modules. Important examples of quasi-hereditary~$1$-covers are the $q$-Schur algebras of the Hecke algebra, by \cite{Nakano2} and the analogue of \cite{qSchAndrew, qSchJie} in type $B$, by \cite[Theorem~6.6]{Rou}.

Consider an abelian category~$\cC$ and a partially ordered set $(S,\le)$. 
As in~\cite[Section~3]{DR} or \cite[Definition~10.1]{HHKP}, 
a {\em standard system in $\cC$ for~$S$} (also known as an exceptional sequence) is a set of objects $\{\Theta(p)\,|\, p\in S\}$ in $\cC$, such that for all $p,q\in S$:
\begin{enumerate}
\item $\End_{\cC}(\Theta(p))$ is a division ring;
\item $\Hom_{\cC}(\Theta(p),\Theta(q))=0$ unless $p\le q$;
\item  $\Ext^1_{\cC}(\Theta(p),\Theta(q))=0$ unless $p< q$.
\end{enumerate}
For a quasi-hereditary algebra~$(A,\le)$, the standard modules $\{S(\lambda),\,\lambda\in\Lambda\}$ form a standard system in $A$-mod for~$(\Lambda,\le)$, see {\it e.g.} \cite[Lemmata 1.2, 1.3 and 1.6]{DR}.

\subsection{Standardly based algebras and Schur algebras}\label{SecSBalg}
There is a close connection, and large overlap, between quasi-hereditary algebras and {\em cellular algebras}, see \cite[Remark~3.10]{CellAlg}, \cite[Proposition~4.1 and Corollary~4.2]{StructureCell}, \cite[Theorem 1.1]{CellQua} and \cite[Theorem 1.1]{Cao}. Moreover, an interesting concept of ``cellularly stratified'' algebras was recently introduced in~\cite{HHKP}, which combines properties of cellular and stratified algebras.
Comparing or combining properties of cellular and stratified algebras is often complicated by the involutive anti-automorphism $\imath$ in the definition of a cell datum in~\cite[Definition 1.1]{CellAlg}. Omitting~$\imath$ in the definition leads to the concept of standardly based algebras.

We use a reformulation of~\cite[Definition~1.2.1]{JieDu}. A {\em standardly based structure} of an algebra~$A$ is a poset $L$, with two-sided ideals $A^{\ge p}$ for each $p\in L$, such that for all $p,q\in L$
\begin{itemize}
\item $A^{\ge p}\supseteq A^{\ge q}$ if~$p<q$, hence $A^{>p}:=\cup_{q>p}A^{\ge q}$ in $A^{\ge p}$ is an ideal in $A^{\ge p}$
\item we can take complements $A^{(p)}$ of~$A^{>p}$ in $A^{\ge p}$, such that~$\bigoplus_{p\in L}A^{(p)}=A$;
\item $A^{\ge p}/A^{>p}\cong W(p)\otimes_{\mk} W'(p)$ as $A$-bimodules for a left module~$W(p)$ and right module~$W'(q)$.
\end{itemize}
The modules $W(p)$ will be referred to as the {\em cell modules} of~$A$.

\begin{rem}\label{remMorSB}
A standardly based structure of an algebra is preserved under Morita equivalences, see \cite[Section~3]{Yang}, meaning that the ideals are naturally linked. In particular, the cell modules are mapped to the corresponding cell modules.
\end{rem}

By \cite[Theorem~2.4.1]{JieDu}, $\Lambda$ can be naturally identified with a subset of~$L$. We consider $\Lambda$ then as a poset for the inherited partial order from~$L$. For~$\lambda\in\Lambda\subset L$, we have furthermore that~$\Top W(\lambda)=L(\lambda)$ and
$$[W(\lambda):L(\lambda)]=1\qquad\mbox{and}\qquad [W(\lambda):L(\mu)]=0\quad\mbox{unless } \mu\le \lambda.$$
By \cite[Proposition 2.4.4]{JieDu}, every indecomposable projective module~$P(\lambda)$ with~$\lambda\in\Lambda$ has a filtration with sections given by modules $W(p)$ such that
$$(P(\lambda):W(\lambda))=1\qquad\mbox{and}\qquad (P(\lambda):W(p))=0\quad\mbox{unless } p\ge \lambda.$$ 
Hence if~$L=\Lambda$, the standardly based algebra is quasi-hereditary with standard modules~$W(\lambda)$. Conversely, over an algebraically closed field,
every quasi-hereditary algebra is standardly based for~$L=\Lambda$ (as posets) and~$W(\lambda)=S(\lambda)$, by \cite[Theorem~4.2.3]{JieDu}.

\begin{ex}\label{ExS1}
Consider $A:=\mk[x]/(x^{t}-1)$ for a field $\mk$
such that~$x^t-1$ splits as
$$x^t-1=\prod_{i=1}^t(x-\omega_i),$$
for not necessarily distinct $\omega_i\in\mk$. Then consider $L:=\{1,2,\cdots,t\}$ with usual order and
$$a_k=\prod_{i=1}^{k-1}(x-\omega_i)\;\;\mbox{ for }\; 2\le k\le t\quad\mbox{ and } a_1=1.$$
The ideals $A^{\ge i}=\Span\{a_i, a_{i+1},\ldots,a_t\}$ give a standardly based structure of~$A$. $A$ is even cellular for involution~$\imath=\id_A$ the identity, by \cite[Lemma~1.2.4]{JieDu}, or \cite[Example 1.3]{CellAlg}.
\end{ex}

\begin{ex}\label{ExS2}
For any field $\mk$, the group algebra~$\mk \mS_t$ of the symmetric group $\mS_t$ on $t$ symbols is cellular and hence standardly based for
$L=\{\lambda\vdash t\},$
equipped with the partial order obtained by reversing the dominance order on partitions, see \cite[Example~1.2]{CellAlg}.
The ideals are obtained from the Murphy basis and the cell modules are the Specht modules of \cite[Section~4]{James}. The algebra~$\mk\mS_t$ is then cellular for the involution~$\imath$, which is the linearisation of the inversion on the group $\mS_t$.
\end{ex}

The following definition is essentially \cite[Definition~12.1]{HHKP}, see also \cite{Henke, Rou}.
\begin{ddef}\label{DefSchur}A {\em (cover-)Schur algebra} of a standardly based algebra~$A$ is a quasi-hereditary~$1$-cover $(\cS,\le)$, such that~$(\Lambda_{\cS},\le)=(L_A,\le)$ and~$F(S(p))\cong W(p)$ for all $p\in L_A$. \end{ddef}
By \cite[Corollary~4.46]{Rou}, a cover-Schur algebra of a standardly based algebra is unique, up to Morita equivalence, if it exists. We use the specification cover-Schur algebra since the generalised Schur algebra representing the orthogonal or symplectic group in the double centraliser property for the Brauer algebra does not act as a Schur algebra in the above sense.


\part{General theory}
\label{GenThe}

 \section{Pre-Borelic pairs}\label{SecPre}

In this section we introduce the notion of pre-Borelic pairs for arbitrary algebras. These will lead to Borelic pairs in the next section, where we will also prove that these contain K\"onig's notion of exact Borel subalgebras for quasi-hereditary algebras as a special case.

 \subsection{Definition and properties}
Fix an algebra~$A$.

\begin{ddef}\label{defB}{\rm
A pair~$(B,H)$ of subalgebras $H\subset B\subset A$ forms a {\em pre-Borelic pair of} $A$ if there exists a two-sided ideal~$B_+$ in $B$, with~$B=H\oplus B_+$, such that
\begin{enumerate}[(I)]
\item $A$ is projective as a right $B$-module and~$B$ is projective as a left $H$-module;
\item taking~$B_+$-invariants, $(\res^A_B-)^{B_+}$, as in \eqref{eqInvariants}, yields an equivalence of categories between simple $A$-modules and simple~$H$-modules;
\item the ideal~$B_+$ is contained in the Jacobson radical~$\rad B$ of~$B$.
\end{enumerate}
}
\end{ddef}

\subsubsection{}\label{ee0} Label the set of isomorphism classes of simple $A$-modules by~$\Lambda$ as in Section~\ref{SecIdem}. By (II), we can use the same set for~$H$. We use the notation~$L(\lambda)$ (resp. $L^0(\lambda)$) for the corresponding simple modules over $A$ (resp.~$H$), and hence 
\begin{equation}
\label{LL0}
L^0(\lambda)\;\cong \;L(\lambda)^{B_+},\qquad\mbox{ for all $\lambda\in\Lambda$}.
\end{equation}  By assumption (III), we also have a one-to-one correspondence between simple $B$-modules and simple~$H$-modules, so
\begin{equation}\label{eqLABH}\Lambda=\Lambda_A=\Lambda_B=\Lambda_H.\end{equation} We use the same notation~$L^0(\lambda)$, for the simple $B$-module with trivial~$B_+$-action defined as the inflation of the $H$-module~$L^0(\lambda)$. 
For each~$\lambda\in\Lambda$, there is an idempotent~$e^0_\lambda\in H\subset B\subset A$, primitive in $H$ (but generally not in $A$) such that~$e^0_\lambda L^0(\lambda)\not=0$.

\subsubsection{}\label{DefDelta} 
For any~$H$-module~$N$, interpreted as a $B$-module with trivial~$B_+$-action, we define
$$\Delta_N\;:=\;\ind^A_B N=A\otimes_BN.$$
By adjunction, we have
\begin{equation}\label{FroDel}\Hom_A(\Delta_N,M)\;\cong \;\Hom_H(N,M^{B_+}),\qquad\mbox{for all $M\in A$-mod.}
\end{equation}

\begin{lemma}\label{TopStan}
Consider a pre-Borelic pair~$(B,H)$ of~$A$. For any $N\in H${\rm-mod} with simple top $L^0(\lambda)$, for some~$\lambda\in \Lambda$, we have
$$\Top \Delta_N\;\cong\; L(\lambda).$$
\end{lemma}
\begin{proof}
Equations \eqref{FroDel} and \eqref{LL0} imply that, for~$\lambda,\lambda'\in\Lambda$, we have
$$\dim\Hom_A(\Delta_N,L(\lambda'))=\dim\Hom_H(N,L^0(\lambda'))=\delta_{\lambda,\lambda'},$$
proving the claim.
\end{proof}

\begin{lemma}\label{PPlambda}
Given a pre-Borelic pair~$(B,H)$, define $P_\lambda=  Ae_\lambda^0$, for any~$\lambda\in \Lambda$. Then
$$P_\lambda\;\cong\; \bigoplus_{\mu\in\Lambda}P(\mu)^{\oplus c_\mu^\lambda}\quad\mbox{with}\quad c_\mu^\lambda:= [\res^A_BL(\mu):L^0(\lambda)].$$
\end{lemma}
\begin{proof}
This follows from $A\otimes_B Be^0_\lambda\cong Ae^0_\lambda$ and adjunction.
\end{proof}

\subsubsection{}For a fixed pre-Borelic pair~$(B,H)$, we introduce the modules
\begin{equation}\label{Standards}
\Delta(\lambda)\;=\;\Delta_{He^0_\lambda}\;=\; A\otimes_B He^0_\lambda,\qquad\mbox{and}\qquad \overline{\Delta}(\lambda)\;=\;\Delta_{L^0(\lambda)}\;=\; A\otimes_B L^0(\lambda).
\end{equation}
We say that an $A$-module~$N$ {\em has a $\Delta$-flag~$\cM$ of length $d\in\mN$} if there are submodules 
\begin{equation}\label{eqFlag}N= M_0\supset M_1\supset M_2\supset\cdots\supset M_{d-1}\supset M_d=0 ,\end{equation}
such that for each $0\le i< d$, we have $M_i/M_{i+1}\cong \Delta(\lambda)$, for some $\lambda\in\Delta$.
For such a $\Delta$-flag~$\cM$, we introduce
\begin{equation}\label{DefMult}(N:\Delta(\lambda))_{\cM}\,:=\; \sum_{i=0}^{d-1}\dim \Hom_A\left(M_i/M_{i+1},L(\lambda)\right).\end{equation}
By definition and Lemma~\ref{TopStan}, we then have
\begin{equation}\label{flagGro}[N]=\sum_{\mu\in\Lambda}\,(N:\Delta(\mu))_{\cM}\;[\Delta(\mu)],\end{equation}
in the Grothendieck group $G_0(A)$ of~$A$-mod.

By left exactness of $\Hom_A(-,L(\lambda))$, equation~\eqref{DefMult} implies the following inequality.
\begin{lemma}\label{BoundMult}
If and~$A$-module~$M$ has a $\Delta$-flag~$\cM$, then
$$\dim\Hom_A(M,L(\lambda))\;\le\;(M:\Delta(\lambda))_{\cM},\qquad\mbox{for all $\lambda\in\Lambda$.}$$
\end{lemma}

\subsection{A special case: idempotent graded algebras}\label{SecPreSpec}

For an algebra~$A$, we choose a decomposition
\begin{equation}\label{decomp1}1_A\;=\; e^\ast_0+e^\ast_1+\cdots+e^\ast_n,
\end{equation}
where~$e^\ast_i$ are mutually orthogonal, but not necessarily primitive, idempotents. For notational convenience, we allow some of these idempotents to be zero.
We have a $\mZ$-grading on $A$:
\begin{equation}\label{grBr}A_j=\bigoplus_{i=\max(0,-j)}^{\min(n, n-j)} e^\ast_i A e^\ast_{i+j},\qquad\mbox{for all $j\in\mZ$}.\end{equation} Indeed, we have $A=\bigoplus_j A_j$ as vector spaces and furthermore
$$A_jA_k=\bigoplus_i e^\ast_i A e^\ast_{i+j} Ae^\ast_{i+j+k}\subset A_{j+k},\qquad\mbox{for all $j,k\in\mZ$}.$$
We call this the {\em idempotent grading} associated to the (ordered) choice of~$e_i^\ast$ in \eqref{decomp1} and set
$$A_+:=\bigoplus_{j>0}A_j\;\mbox{ and }\;A_{-}:=\bigoplus_{j<0}A_j.$$
\begin{ddef}\label{BHgrad}
An {\em $\mN$-graded subalgebra} of an idempotent graded algebra~$A$ is a graded subalgebra~$B$ such that~$B_{j}=0$ for~$j<0$.
For such $B$, set 
$$H\,:=\,B_0\,=\,\bigoplus_{i\in\mN}e^\ast_i B e^\ast_i\quad\mbox{ and}\qquad B_+\,:=\,\bigoplus_{j>0}B_j.$$ An $\mN$-graded subalgebra~$B$ is {\em complete} if 
\begin{equation}\label{eqComplete}A=B\oplus A_{-}B.\end{equation}
\end{ddef}
Note that any $\mZ$-graded subalgebra~$B$ satisfying \eqref{eqComplete} is automatically $\mN$-graded. Furthermore, it is easy to check that equation~\eqref{eqComplete} is equivalent to
\begin{equation}\label{eqComplete2}A\;=\; B\,\oplus\, A_-A.\end{equation}

We will use freely the fact that, as follows from the definitions, {\em any module~$M$, over $A$, $B$ or~$H$, is automatically a $\mZ$-graded module}, by setting
\begin{equation}\label{eqDefZGrad}M_{-j}:=e^\ast_{j}M,\quad\mbox{ for }\;\quad j\in\{0,1,\cdots,n\}.\end{equation}

\begin{lemma}\label{LemGradBor}
Consider a complete $\mN$-graded subalgebra~$B$ of an idempotent graded algebra~$A$. If $A_B$ and~${}_HB$ are projective, then~$(B,H)$ is a pre-Borelic pair.
\end{lemma}
\begin{proof}
Condition (I) in Definition~\ref{defB} is given.

Condition (II) is the requirement that~$\{L(\lambda)^{B_+},\lambda\in \Lambda\}$ be a complete set of non-isomorphic simple~$H$-modules. Consider the simple $A$-module~$L(\lambda)$, which we equip with $\mZ$-grading as in~\eqref{eqDefZGrad}. Let $k_0$ be the maximal degree for which $L(\lambda)_{k_0}$ is non-zero. We clearly have $L(\lambda)_{k_0}\subset L(\lambda)^{A_+}$. Now take some $w\in L(\lambda)_i$, for~$i<k_0$. Since $L(\lambda)$ is simple, we have $L(\lambda)_{k_0}\subset A_+ w$. In particular $w\not\in L(\lambda)^{A_+}$. It then follows quickly that~$L(\lambda)_{k_0}=L(\lambda)^{A_+}$.
Similarly, simplicity of the $A$-module~$L(\lambda)$ implies that each $L(\lambda)_j$, so in particular $L(\lambda)^{A_+}=L(\lambda)_{k_0}$, is simple as an $A_0$-module. Equation~\eqref{eqComplete} implies that~$A_0=H\oplus (A_{-}A_+)_0$. Consequently, we have $A_0v=Hv$ for any $v\in L(\lambda)^{A_+}$ and the restriction of~$L(\lambda)^{A_+}$ to an $H$-module remains simple.
Equation~\eqref{eqComplete} implies
$A_+\subset AB_+, $
from which it follows that~$L(\lambda)^{B_+}=L(\lambda)^{A_+}$ as $H$-modules. Hence,
\begin{equation}\label{simpleComp}L^0(\lambda):=L(\lambda)^{B_+}\end{equation} is a simple~$H$-module.

Now we consider an arbitrary simple~$H$-module~$L^0$, which we can interpret as a simple $B$-module contained in one degree, say $-i$ for~$i\in\mN$, so $L^0=e_i^\ast L^0$. The $A$-module~$M:=\ind^A_B L^0$ satisfies $e_j^\ast M=0$ unless $j\ge i$, with~$e_i^\ast M\cong L^0$, by equation~\eqref{eqComplete}. It thus follows that~$M\not=0$, and that any proper submodule~$S$ of~$M$ satisfies $e_i^\ast S=0$. Hence, taking the union of all proper submodules yields the
 unique maximal submodule. We conclude that~$\ind^A_B L^0$ has simple top. It follows from adjunction that this top is~$L(\lambda)$ if and only if~$L^0\cong L^0(\lambda)$.
This implies that the simple modules in \eqref{simpleComp} exhaust all simple~$H$-modules and that~$L^0(\lambda)\cong L^0(\mu)$ implies $L(\lambda)\cong L(\mu)$, which proves condition~(II).

Condition (III) is immediate from the $\mN$-grading on $B$. This concludes the proof.
\end{proof}

The lemma motivates the following definition.
\begin{ddef}\label{GrPreBor}
Let $A$ be an idempotent graded algebra with complete $\mN$-graded subalgebra~$B$.
If $A_B$ and~${}_HB$ are projective with~$H:=B_0$, we call $B$ a {\em graded pre-Borelic algebra}.
\end{ddef}
The reason we just work with the algebra~$B$ for graded pre-Borelic subalgebras, instead of a pair, is that~$H:=B_0$ is defined through the grading on $B$.

\begin{rem}\label{RemJchain}
An idempotent grading on an algebra~$A$ implies a chain of idempotent ideals. Set $f_k:=\sum_{j<k}e^\ast_j$, for~$1\le k\le n+1$ and~$J_k:=Af_kA$, leading to the chain
\begin{equation}\label{chainProp}0=J_0\subset J_1\subset \cdots\subset J_{n}\subset J_{n+1}=A.\end{equation}
Using the chain \eqref{chainProp}, we can rewrite equation~\eqref{eqComplete2} as
\begin{equation}\label{eqABJ}e^\ast_j A\;=\; e^\ast_j B\,\oplus\, e^\ast_jJ_{j},\qquad\mbox{for all }\,0\le j\le n.\end{equation}
In particular, $e_j^\ast Ae_j^\ast=e_j^\ast He_j^\ast\oplus e_j^\ast J_je_j^\ast$, and for the algebras in Remark \ref{RemSSS} we find
\begin{equation}\label{eqAiH}A^{(i)}=f_{i+1} (A/J_i)f_{i+1}=e_i^\ast (A/J_i)e_i^\ast\cong e_i^\ast He_i^\ast=He_i^\ast=e_i^\ast H.\end{equation}
\end{rem}

\begin{lemma}\label{LemHB}
Let $A$ be an idempotent graded algebra with complete $\mN$-graded subalgebra~$B$. Then ${}_HB$ is projective if and only if the left $A^{(j)}$-module~$f_{j+1} A/J_j$ is projective for~$0\le j\le n$.
\end{lemma}
\begin{proof}
By \eqref{eqAiH}, we have $ He_j^\ast\cong A^{(j)}$ and by \eqref{eqABJ}, $e_j^\ast B\cong f_{j+1} (A/J_j)$, which implies the statement.
\end{proof}

\begin{lemma}\label{LemABAM}
Consider an idempotent graded algebra~$A$ with graded pre-Borelic algebra~$B$, and an $He_i^\ast\cong A^{(i)}$-module~$M$. We have
$$\Delta_M=A\otimes_B M\;\cong\; (A/J_i)f_{i+1}\otimes_{A^{(i)}}M.$$
\end{lemma}
\begin{proof}
We will prove that the right adjoints of the functors $A\otimes_B-$ and~$(A/J_i)f_{i+1}\otimes_{A^{(i)}}-$ are isomorphic. These
right adjoints are given respectively by
$$e_i^\ast(\res^A_B-)^{B_+}\qquad\mbox{and}\qquad \Hom_A(Af_{i+1}/J_if_{i+1},-).$$
When applied to~$N\in A$-mod, the first corresponds to taking all elements in $e_i^\ast N$ which are annihilated by $B_+$ and thus by $A_+$. The second corresponds to taking all elements in $f_{i+1}N$ which are annihilated by elements of $J_if_{i+1}$. These two clearly coincide.
\end{proof}


\section{Borelic pairs}
\subsection{Definition}
Consider an algebra~$A$ with pre-Borelic pair~$(B,H)$.

\subsubsection{}
We define an equivalence relation~$\sim_H$ on $\Lambda$, where~$\lambda\sim_H\mu$ if and only if~$L^0(\lambda)$ and~$L^0(\mu)$ are in the same indecomposable block of~$H$-mod. This equivalence relation is hence trivial if and only if $H$ is quasi-local. In particular, the equivalence relation is trivial when~$B_+=\rad B$. A partial quasi-order $\preccurlyeq$ on $\Lambda$ is said to be {\em~$H$-compatible} if
$$\lambda\sim_H\mu\; \Rightarrow\; \lambda\sim\mu,$$
where we recall that~$\lambda\sim\mu$ stands for~$\lambda\preccurlyeq\mu$ and~$\mu\preccurlyeq \lambda$.

\begin{ddef}\label{DefBorPair}
Consider an~$H$-compatible partial quasi-order $\preccurlyeq$ on $\Lambda$.
We say that~$(B,H)$ is a {\em Borelic pair} of~$(A,\preccurlyeq)$ if, for all~$\lambda,\mu\in \Lambda$, we have
\begin{enumerate}
\item $[B_+e^0_\lambda:L^0(\mu)]=0\mbox{ unless }\mu\succ \lambda; $
\item $[\Delta(\lambda):L(\mu)]=0\mbox{ unless }\; \mu \preccurlyeq\lambda;$
\end{enumerate}
A Borelic pair~$(B,H)$ is called {\em exact} if~$[\overline{\Delta}(\lambda):L(\lambda)]=1$, for all~$\lambda\in\Lambda$.
\end{ddef}

\subsection{Properties}

Now we start exploring the properties of Borelic pairs. Whenever a partial quasi-order $\preccurlyeq$ on $\Lambda$ is considered in relation to a pre-Borelic pair~$(B,H)$ of~$A$, we assume it to be~$H$-compatible.

\begin{lemma}
\label{PlambdaFilt}
Consider an algebra~$(A,\preccurlyeq)$ with Borelic pair~$(B,H)$, then the module~$P_\lambda=A\otimes_B Be^0_\lambda$ admits a $\Delta$-flag. In particular, the module~$K$ defined through the short exact sequence
\begin{equation}\label{sesPD}0\to K\to P_\lambda\to \Delta(\lambda)\to 0,\end{equation} has a $\Delta$-flag~$\cM$ with
$$(K: \Delta(\mu))_{\cM}=0\quad \mbox{unless} \;\;{\mu}\succ{\lambda}.$$
\end{lemma}
\begin{proof}
We start from the short exact sequence of~$B$-modules
$$0\to B_+ e^0_\lambda\to Be^0_\lambda\to He_\lambda^0\to 0.$$
We set $N:=B_+ e^0_\lambda$, and apply the exact functor~$A\otimes_B-$, see Definition \ref{defB}(I), to the above short. This yields the short exact sequence
$$0\to A\otimes_BN\to P_\lambda\to \Delta(\lambda)\to 0.$$

We claim that the $B$-module~$N=B_+ e^0_\lambda$
has a filtration with sections given by the $B$-modules $He_\mu^0=Be_\mu^0/B_+e_\mu^0$ with~$\mu\succ \lambda$. 
Firstly, by Definition~\ref{defB}(I), ${}_HBe^0_\lambda$ and~$He_\lambda^0$ are projective, so also $N$ is projective as an $H$-module. Since $\sim$ is $H$-compatible, we have a decomposition
$${}_HN\;=\;\bigoplus_{[\mu]}N_{[\mu]},$$
where $[\mu]$ runs over the equivalence classes of~$\sim$ and~$[N_{[\mu]}:L^0(\nu)]=0$ unless $\mu\sim\nu$. Furthermore, $N_{[\mu]}$ is projective as an $H$-module, so a direct sum of modules $He^0_\kappa$ with $\kappa\in [\mu]$. By Definition~\ref{DefBorPair}(1),
 $N_{[\mu]}=0$ unless $\mu\succ\lambda$.
Take $\mu\in\Lambda$ such that~$N_{[\mu']}=0$ if~$\mu'\succ\mu$. This actually constitutes a $B$-submodule with trivial~$B_+$ action, by Definition~\ref{DefBorPair}(1). We can then proceed iteratively with the module~$N/N_{[\mu]}$. This yields the desired filtration.

As $\ind^A_B-$ is exact by Definition~\ref{defB}(I), this filtration of~$N$ induces the desired filtration of~$K:=A\otimes_BN$.
\end{proof}

\begin{cor}\label{CorPP}
Consider an algebra~$(A,\preccurlyeq)$ with Borelic pair~$(B,H)$, the projective module~$P_\lambda$ is the direct sum of $P(\lambda)$ and certain $P(\mu)$ with $\mu\succ\lambda$. 
\end{cor}
\begin{proof}
It follows from Lemmata~\ref{PlambdaFilt} and~\ref{BoundMult} that 
$$\dim\Hom_A(P_\lambda,L(\lambda))=1,\qquad\mbox{and}\qquad\Hom_A(P_\lambda,L(\mu))=0,\;\mbox{ for~$\mu\not\succ\lambda$.}$$
This proves the requested decomposition of $P_\lambda$ into projective covers. 
\end{proof}

\begin{lemma}\label{LemBasis}
Consider an algebra~$(A,\preccurlyeq)$ with Borelic pair~$(B,H)$. The modules $\Delta(\lambda)$ induce a $\mk$-basis $\{[\Delta(\lambda)]\,|\, \lambda\in\Lambda\}$ of~$G_0(A)$.
\end{lemma}
\begin{proof}
The indecomposable projective $A$-modules induce a basis $\{[P(\lambda)]\,|\, \lambda\in\Lambda\}$. Corollary~\ref{CorPP} implies that another basis is given by~$\{[P_\lambda]\,|\, \lambda\in\Lambda\}$. Lemma~\ref{PlambdaFilt} and equation~\eqref{flagGro} then imply that~$\{[\Delta(\lambda)]\,|\, \lambda\in\Lambda\}$ also forms a basis.
\end{proof}

By equation~\eqref{flagGro} and Lemma~\ref{LemBasis}, the multiplicities $(N:\Delta(\mu))_{\cM}$ coincide for all possible $\Delta$-flags $\cM$, if~$(B,H)$ is a 
Borelic pair. Hence, we leave out the reference to~$\cM$.

The following lemma states that, in case~$\preceq$ is actually a partial order, the modules $\{\Delta(\lambda)\,|\,\lambda\in\Lambda\}$ form a standard system.
\begin{lemma}\label{LemDmor}\label{LemExt1}
For an algebra~$(A,\preccurlyeq)$ with Borelic pair~$(B,H)$ and~$\lambda,\mu\in\Lambda$, we have
\begin{enumerate}
\item $\Hom_A(\Delta(\lambda),\Delta(\mu))=0\qquad\mbox{unless }\; \lambda\preccurlyeq \mu;$
\item $\Ext^1_A(\Delta(\lambda),\Delta(\mu))=0\quad\mbox{unless}\quad \lambda\prec\mu;$
\item $\Ext^1_A(\Delta(\lambda),M)=0$ for any $M\in A\mbox{\rm{-mod}}$ for which $[M:L(\nu)]=0$ if $\nu\succ\lambda$.
\end{enumerate}
\end{lemma}
\begin{proof}
Part (1) is a consequence of Definition~\ref{DefBorPair}(2) and Lemma~\ref{TopStan}. Part (2) is a special case of part (3), by Definition~\ref{DefBorPair}(2).

Now we consider $M$ as in part (3). Consider $K$ as in Lemma~\ref{PlambdaFilt}. The contravariant left exact functor~$\Hom_A(-,M)$ applied to \eqref{sesPD} yields a surjection
$$\Hom_A(K,M)\tto \Ext^1_A(\Delta(\lambda),M).$$
By Lemmata~\ref{PlambdaFilt} and~\ref{TopStan}, the left-hand side (and therefore the right-hand side) is zero.\end{proof}

\begin{lemma}\label{DomDelta}
Consider an algebra~$(A,\preccurlyeq)$ with Borelic pair~$(B,H)$. If for an $A$-module~$M$ with~$\Delta$-flag and~$\lambda\in\Lambda$, we have 
$$(M:\Delta(\lambda))\not=0\qquad\mbox{and}\qquad(M:\Delta(\mu))=0\quad\mbox{if $\mu\succ\lambda$,}$$
then there exists a module $M'$ with $\Delta$-flag for which we have a short exact sequence
$$0\to \Delta(\lambda)\to M\to M'\to 0.$$
\end{lemma}
\begin{proof}
This follows immediately from Lemma~\ref{LemExt1}.
\end{proof}

\begin{lemma}\label{LemSurjN}
Consider an algebra~$(A,\preccurlyeq)$ with Borelic pair~$(B,H)$. If we have a surjection $M_1\tto M_2$ where $M_1,M_2$ have $\Delta$-flags, the kernel $K$ also has a $\Delta$-flag.

\end{lemma}
\begin{proof}
We prove this by induction on the length of the $\Delta$-flag of~$M_1$. Since all $\Delta(\lambda)$ have different simple top $L(\lambda)$, see Lemma~\ref{TopStan}, it follows that there are no epimorphisms $\Delta(\lambda)\tto\Delta(\mu)$, different from identities. It follows easily that the claim in the lemma is thus true for flags of length 1, meaning if $M_1\cong \Delta(\lambda)$, for some $\lambda\in\Lambda$.

First assume that there exists $\mu\in\Lambda$ such that~$(M_1:\Delta(\mu))\not=0$, but $(M_2:\Delta(\nu))=0$ for all $\nu  \succcurlyeq\mu$. Without restriction, we can take such $\mu$ such that~$(M_1:\Delta(\kappa))=0$, for all $\kappa\succ\mu$. By Lemma~\ref{DomDelta}, we have $\Delta(\mu)\hookrightarrow M_1$ with cokernel with $\Delta$-flag and such that the composition with $M_1\tto M_2$ is zero. Thus there exists a morphism $\iota$ making
$$\xymatrix{
K\ar@{^{(}->}[rr]&&M_1\ar@{->>}[rr]&&M_2\\
&&\Delta(\mu)\ar@{^{(}->}[u]\ar@{.>}[llu]^{\iota}
}$$ 
commute. Clearly $\iota$ is a monomorphism and
the kernel $K$ of~$M_1\tto M_2$ will thus have a $\Delta$-flag if the kernel $K/\Delta(\mu)$ of~$M_1/\Delta(\mu)\tto M_2$ has one. 

Now we assume that there exists no $\mu$ as in the previous paragraph.
We take $\lambda\in\Lambda$ for which we can apply Lemma~\ref{DomDelta} to obtain a monomorphism $\Delta(\lambda)\hookrightarrow M_2$. If we would have $[K:L(\nu)]\not=0$, for some $\nu\succ\lambda$, then by Definition~\ref{DefBorPair}(2), we must have $(M_2:\Delta(\mu))\not=0$, for some $\mu\succcurlyeq\nu\succ \lambda$. This contradicts the assumption in the beginning of this paragraph. By Lemma~\ref{LemExt1}(3), we thus have $\Ext^1_{A}(\Delta(\lambda),K)=0$.
This means we get a morphism $\iota$, making
$$\xymatrix{
K\ar@{^{(}->}[rr]&&M_1\ar@{->>}[rr]&&M_2\\
&&\Delta(\lambda)\ar@{^{(}->}[rru]\ar@{.>}[u]^{\iota}
}$$
commute. Since $\iota$ is a monomorphism, $K$ is the kernel of~$M_1/\Delta(\lambda)\tto M_2/\Delta(\lambda)$.

In both cases, we have thus reduced the problem to a case where the length of the $\Delta$-flag of~$M_1$ is strictly lower.
\end{proof}

\begin{cor}\label{corPFlag}
Consider an algebra~$(A,\preccurlyeq)$ with Borelic pair~$(B,H)$. Then the kernel $K$ of~$P(\lambda)\tto \Delta(\lambda)$ has a $\Delta$-flag with 
$$(K: \Delta(\nu))=0\quad \mbox{unless} \;\;{\nu}\succ{\lambda}.$$
\end{cor}
\begin{proof}
We prove by induction that~$P(\lambda)$ has a $\Delta$-flag. If there are no $\nu\succ\lambda$, then $P_\lambda=P(\lambda)$ by Corollary~\ref{CorPP}. Hence it has a $\Delta$-flag by Lemma~\ref{PlambdaFilt}. 

If we already established that~$P(\mu)$ has a $\Delta$-flag for all $\mu\succ\lambda$, then it follows from the inclusion $P(\lambda)\hookrightarrow P_\lambda$ of the direct summand~$P(\lambda)$ into~$P_\lambda$ and Lemma~\ref{LemSurjN} that~$P(\lambda)$ has a $\Delta$-flag.

Lemma~\ref{LemSurjN} then further implies that the kernel of~$P(\lambda)\tto \Delta(\lambda)$ has a $\Delta$-flag and the restriction on multiplicities is inherited from Lemma~\ref{PlambdaFilt}.
\end{proof}

 \subsection{Exact Borel subalgebras for exactly standardly stratified algebras}

 \label{IntroQH}

Before we present the properties of algebras with Borelic pairs we need to review the notion of exact Borel subalgebras.

\subsubsection{}\label{DefKoe}An {\em exact Borel subalgebra of a quasi-hereditary algebra}~$(A,\le)$, as introduced by K\"onig in~\cite{Koenig}, is a subalgebra~$B$ satisfying the following three conditions from \cite[Definition~2.2]{KKO}.
\begin{enumerate}[(i)]
\item The simple modules of~$B$ are also labelled by~$\Lambda_A$, so $\Lambda_A=\Lambda_B=\Lambda$ and we denote the simple $B$-modules by $\{L^0(\lambda)\,|\, \lambda\in\Lambda\}$. The only simple subquotients in the radical of the projective cover of~$L^0(\lambda)$ are $L^0(\mu)$ with~$\mu>\lambda$.
\item The functor~$A\otimes_B-$ is exact.
\item There is an isomorphism $A\otimes_B L^0(\lambda)\cong S(\lambda)$ for each~$\lambda\in\Lambda$.
\end{enumerate}

For an algebra~$B$ with simple modules $L^0(\lambda)$ for~$\lambda\in\Lambda$, choose a decomposition~$1=\sum_{\lambda\in\Lambda}\overline{e}_\lambda$, where $\{\overline{e}_\lambda\}$ is a set of mutually orthogonal idempotents, satisfying~$\overline{e}_\mu L^0(\lambda)=0$ unless $\mu=\lambda$.
 
 \begin{lemma}\label{BorelWed}
 An exact Borel subalgebra~$B$ of a quasi-hereditary algebra~$(A,\le)$ is Wedderburn. We have
 $$B\;=\; H\,\oplus\, \rad B\qquad\mbox{ with }\; H:=\bigoplus_{\lambda\in\Lambda}\overline{e}_\lambda\, B\, \overline{e}_\lambda.$$
 \end{lemma}
 \begin{proof}
 By condition (i) above, we have $\overline{e}_\mu \,B \,\overline{e}_\lambda=0$ unless $\mu \ge \lambda$, showing that 
 $$\oplus_{\mu\not=\lambda} \overline{e}_\lambda \,B\, \overline{e}_\mu=\oplus_{\mu>\lambda} \overline{e}_\lambda\, B \,\overline{e}_\mu\;\subset\; \rad B.$$
 That $\overline{e}_\lambda \,B \,\overline{e}_\lambda$ is semisimple follows immediately from the fact that~$[B\overline{e}_\lambda:L^0(\lambda)]$ is equal to the number of direct summands in $B\overline{e}_\lambda$ by condition (i). Hence, $H$ is semisimple and the displayed inclusion an equality.
\end{proof}

  In \cite[Definition~1]{Kluc}, the notion of exact Borel subalgebras was generalised to properly stratified algebras and it can be further generalised to exactly standardly stratified algebras.

\begin{ddef}\label{DefEB}
An {\em exact Borel subalgebra of an exactly standardly stratified algebra}~$(A,\preccurlyeq)$ is a subalgebra~$B$ such that the following four conditions are satisfied.
\begin{enumerate}
\item The simple modules of~$B$ are also labelled by~$\Lambda_A$, so $\Lambda_A=\Lambda_B=\Lambda$ and we denote the simple $B$-modules by $\{L^0(\lambda)\,|\, \lambda\in\Lambda\}$. We have $\overline{e}_\mu \, B\, \overline{e}_\lambda=0$ unless~$\lambda\preccurlyeq \mu$.
\item The functor~$A\otimes_B-$ is exact.
\item The algebra~$$H:=\bigoplus_{\lambda\sim\mu} \overline{e}_\lambda\, B\,\overline{e}_\mu\;\cong\; B/\left(\bigoplus_{\mu \prec\lambda}\overline{e}_\lambda \, B\, \overline{e}_\mu\right)$$
is such that the left $H$-module~${}_HB$ is projective. The interpretation of~$H$ as a quotient algebra of~$B$ allows to interpret $H$-modules as $B$-modules by inflation.
\item There are isomorphisms $A\otimes_B L^0(\lambda)\cong \overline{S}(\lambda)$ and~$A\otimes_B P^0(\lambda)\cong {S}(\lambda)$, with~$P^0(\lambda)$ the $H$-projective cover of~$L^0(\lambda)$, for each~$\lambda\in\Lambda$.
\end{enumerate}
\end{ddef}

\begin{rem}\label{CompareEB}
There is no conflict between Definition~\ref{DefEB} and the special case of quasi-hereditary algebras in \ref{DefKoe}:
\begin{itemize}
\item An exact Borel subalgebra~$B$ of a quasi-hereditary algebra~$(A,\le)$ is an exact Borel subalgebra of~$(A,\le)$, interpreted as an exactly standardly stratified algebra. This follows from Lemma~\ref{BorelWed}.
\item Consider an exact Borel subalgebra~$B$ of an exactly standardly stratified algebra~$(A,\preccurlyeq)$ which happens to be quasi-hereditary. As $S(\lambda)\cong \overline{S}(\lambda)$, Definition~\ref{DefEB}(4) implies that~$H$ is semisimple. It then follows that the multiplicity of~$L^0(\lambda)$ in its projective cover over $B$ is 1. Hence $B$ is an exact Borel subalgebra of~$(A,\preccurlyeq)$ as a quasi-hereditary algebra.
\end{itemize}
\end{rem}

\begin{rem}
If $(A,\preccurlyeq)$ is properly stratified, Definition~\ref{DefEB} is almost identical to \cite[Definition~1]{Kluc}. The latter additionally requires projectivity of $B_H$.
\end{rem}

\subsection{Algebras with (exact) Borelic pairs}\label{Sec4}

\begin{thm}\label{ThminfQH}
If an algebra~$A$ admits a Borelic pair~$(B,H)$ for an~$H$-compatible partial quasi-order $\preccurlyeq$, then~$(A,\preccurlyeq)$ is standardly stratified with $S(\lambda):=\Delta(\lambda)$. If moreover $\preccurlyeq$ is a partial order (implying that~$H$ is quasi-local) then~$(A,\preccurlyeq)$ is strongly standardly stratified.
\end{thm}
\begin{proof}
Consider $(A,\preccurlyeq)$ with Borelic pair~$(B,H)$. That the modules $S(\lambda):=\Delta(\lambda)$
satisfy Definition~\ref{DefA2}(1) follows from Definition~\ref{DefBorPair}(2) and Corollary~\ref{corPFlag}.
The same reasoning for the case where $\preccurlyeq$ is a partial order implies that~$(A,\preccurlyeq)$ is strongly standardly stratified as in Definition~\ref{DefA2}(3). 
\end{proof}

\begin{thm}\label{PropKo}
Consider an algebra~$(A,\preccurlyeq)$ for some partial quasi-order $\preccurlyeq$.
\begin{enumerate}
\item Assume that~$\preccurlyeq$ is a partial order. The following statements are equivalent, for any subalgebra~$B$ of~$A$:
\begin{itemize} 
\item $B$ is Wedderburn and~$(B,B/\rad B)$ is an exact Borelic pair.
\item $(A,\preccurlyeq)$ is quasi-hereditary and~$B$ is an exact Borel subalgebra.
\end{itemize}
\item Assume that~$\preccurlyeq$ is a partial order. The following statements are equivalent, for any subalgebras $H\subset B$ of~$A$:
\begin{itemize} 
\item $(B,H)$ is an exact Borelic pair.
\item $(A,\preccurlyeq)$ is properly stratified and~$B$ is an exact Borel subalgebra.
\end{itemize}
\item The following statements are equivalent, for any subalgebras $H\subset B$ of~$A$:
\begin{itemize} 
\item $(B,H)$ is an exact Borelic pair.
\item $(A,\preccurlyeq)$ is exactly standardly stratified and~$B$ is an exact Borel subalgebra.
\end{itemize}
\end{enumerate}
\end{thm}
We start the proof of this theorem with the following lemma.

\begin{lemma}\label{LempreBor}
If $A$ is exactly standardly stratified with an exact Borel subalgebra~$B$ with subalgebra~$H$ as in Definition~\ref{DefEB}, then~$(B, H)$ is a pre-Borelic pair.
\end{lemma}
\begin{proof}
We check the conditions in Definition~\ref{defB}. Condition~\ref{defB}(I) follows from Definition~\ref{DefEB}(2) and~(3). We define the two-sided ideal~$B_+:=\oplus_{\lambda\not\sim\mu}\overline{e}_\lambda B\overline{e}_\mu=\oplus_{\lambda\succ\mu}\overline{e}_\lambda B\overline{e}_\mu$, for which we have $B=H\oplus B_+$.
Condition~\ref{defB}(III) is clearly satisfied.

As the modules ${S}(\lambda)$ of \ref{DefA2}(1) have simple top $L(\lambda)$, Definition~\ref{DefEB}(4) leads, using adjunction, to
$$\delta_{\lambda,\mu}\;=\;\dim\Hom_B(P^0(\lambda),L(\mu))\;=\;\dim\Hom_H(P^0(\lambda),L(\mu)^{B_+})\;=\; [L(\mu)^{B_+}: L^0(\lambda)].$$
Condition~\ref{defB}(II) follows from this equation. \end{proof}

\begin{proof}[Proof of Theorem~\ref{PropKo}]
It suffices to prove part (3). Part (2) is a special case and so is part~(1) by remark~\ref{CompareEB}.

Firstly, we assume that~$(B,H)$ is an exact Borelic pair. By Theorem~\ref{ThminfQH}, $A$ is standardly stratified. Furthermore, the standard module~$S(\lambda)$ has a filtration where the sections are given by~$\overline{S}(\mu):=\overline{\Delta}(\mu)$ by Definition~\ref{defB}(I), with~$\mu\sim_H\lambda$. By Definition~\ref{DefBorPair} we have $[\overline{S}(\mu):L(\mu)]=1$, so $(A,\preccurlyeq)$ is exactly standardly stratified.

Now we continue by proving that~$(B,\preccurlyeq)$ is an exact Borel subalgebra, as in Definition~\ref{DefEB}. Condition~\ref{DefEB}(1) is satisfied by equation~\eqref{eqLABH} and Definition~\ref{DefBorPair}(1). Condition~\ref{DefEB}(2) follows from \ref{defB}(I). To prove the two remaining conditions \ref{DefEB}(3) and~(4), we just need to show that the subalgebra of~$B$ defined in \ref{DefEB}(3), which we denote by~$H'$, is the same as the subalgebra~$H$ in the pair~$(B,H)$. Definition~\ref{DefBorPair}(1) implies that $H'\subset H$. The other inclusion follows from the fact that~$\sim$ is assumed to be $H$-compatible.

Secondly, we assume that~$(A,\preccurlyeq)$ is exactly standardly stratified and~$B$ is an exact Borel subalgebra. By Lemma~\ref{LempreBor}, $(B,H)$ is a pre-Borelic pair, so it remains to check the conditions in Definition~\ref{DefBorPair}. Condition~\ref{DefBorPair}(1) follows from \ref{DefEB}(1) and~\ref{DefEB}(3). Condition~\ref{DefBorPair}(2) follows from Definitions \ref{DefA2}(1) and~\ref{DefEB}(4). Hence $(B,H)$ is a Borelic pair.
The extra condition for exact Borelic pairs follows from~\ref{DefEB}(4).\end{proof}

\subsection{A special case: graded Borelic subalgebras}\label{ssspecial}
Fix an idempotent graded algebra~$A$ as in Section~\ref{SecPreSpec}. We have the corresponding chain \eqref{chainProp} of idempotent ideals and hence an $n$-decomposition~$\cQ$ of~$\Lambda=\sqcup_{i=0}^n\Lambda_i$, as in Section~\ref{SecChain}, with corresponding quasi-order and partial-order.  We now introduce these directly.

\begin{ddef}\label{lepleq}
Introduce the function
$F_A:\Lambda\to \{0,1,\cdots,n\}$, where $F_A(\lambda)$ is the minimal~$i\in\mN$ such that there is an idempotent~$f\in e_i^\ast Ae_i^\ast$, equivalent to~$e_\lambda$. Then we have~$\Lambda_i:=\{\lambda\in\Lambda\,|\, F_A(\lambda)=i\}$.
\begin{enumerate}
\item The {\em idempotent quasi-order} is the total quasi-order $\preccurlyeq_{\cQ}$ on $\Lambda$ such that~$\mu\preccurlyeq_{\cQ} \lambda$ if and only if~$F_A(\mu)\ge F_A(\lambda)$.
\item The {\em idempotent partial order} is the partial order $\le_{\cQ}$ on $\Lambda$ such that~$\mu<_{\cQ} \lambda$ if and only if~$F_A(\mu)> F_A(\lambda)$
\end{enumerate}
\end{ddef}

\begin{lemma}\label{LemHcomp}
For any~$\lambda\in\Lambda$ we have~$e_\lambda^0\in He^\ast_i$ with~$i=F_A(\lambda)$. In particular, the partial quasi-order $\preccurlyeq_{\cQ}$ is~$H$-compatible.
\end{lemma}
\begin{proof}
It follows from equation~\eqref{simpleComp} that~$L^0(\lambda)$ is a vector space contained in the highest degree of the graded vector space $L(\lambda)$. This implies that~$e^\ast_iL^0(\lambda)\not=0$ for the lowest $i$ for which there is an idempotent~$e$ equivalent to~$e_\lambda$ such that~$ee^\ast_i=e$.

Hence, for that minimal~$i$, which is $F_A(\lambda)$, we have~$e_\lambda^0e_i^\ast=e^0_\lambda$.
\end{proof}

\begin{prop}\label{grBrBo}
Consider an algebra~$A$ with an idempotent grading. Let $B$ be a graded pre-Borelic subalgebra with~$H:=B_0$.
\begin{enumerate}
\item The pair~$(B,H)$ is an exact Borelic pair for the idempotent quasi-order $\preccurlyeq_{\cQ}$.
\item If~$H$ is quasi-local, $(B,H)$ is an exact Borelic pair for the idempotent partial order~$\le_{\cQ}$. 
\end{enumerate}
\end{prop}
Combining this with Theorem~\ref{PropKo} yields the following corollary.
\begin{cor}\label{CorgrBrBo}
Consider an algebra~$A$ with an idempotent grading and graded pre-Borelic subalgebra~$B$ with~$H:=B_0$.
\begin{enumerate}
\item Then $(A,\preccurlyeq_{\cQ})$ is exactly standardly stratified with exact Borel subalgebra~$B$.
\item If $H$ is quasi-local, $(A,\le_{\cQ})$ is properly stratified with exact Borel subalgebra~$B$.
\item If $H$ is semisimple, $(A,\le_{\cQ})$ is quasi-hereditary with exact Borel subalgebra~$B$.
\end{enumerate}
\end{cor}

\begin{rem}
The corollary indicates that our methods might be well-suited to extend to the setting of \cite{AffineK}. With notation introduced in Appendix~\ref{AppKl}, it can be paraphrased as: If an idempotent graded algebra~$A$ has graded pre-Borelic subalgebra~$B$ with~$H\in\cB$, then $A$ is $\cB$-properly stratified, for~$\cB$ one of the classes of algebras $\cD$, $\cL$ or~$\cS$.
\end{rem}

\begin{proof}[Proof of Proposition \ref{grBrBo}]
First observe that~$\overline{\Delta}(\lambda)$, with grading as in \eqref{eqDefZGrad}, lives in degrees $\{-j,-j-1,\ldots\}$ with~$j=F_A(\lambda)$, by Lemma~\ref{LemHcomp} and equation~\eqref{eqComplete}. Moreover, the degree $-j$-component is precisely $L^0(\lambda)$. Each simple module~$L(\mu)$ lives similarly in degree $\{-k,-k-1,\ldots\}$ with~$k=F_A(\mu)$.
Therefore we find that
\begin{equation}
\label{deltabar}
[\overline{\Delta}(\lambda):L(\lambda)]=1\quad\mbox{ and }\quad [\overline{\Delta}(\lambda):L(\mu)]=0\mbox{ unless $\mu\le_{\cQ}\lambda$}.
\end{equation}

Now we focus on the proof of part (1). By \eqref{deltabar} we only need to check conditions (1) and~(2) in Definition~\ref{DefBorPair}. We have
$$[B_+e^0_\lambda:L^0(\mu)]=\sum_{j>0}\dim e_\mu^0 B_j e_\lambda^0.$$
By Lemma~\ref{LemHcomp} we find that~$e^0_\mu B_j e^0_\lambda\not=0$ for~$j>0$ implies that~$\mu\succ_{\cQ}\lambda$, proving condition~\ref{DefBorPair}(1). 
The module~$\Delta(\lambda)=A\otimes_{B}He_\lambda^0$ has a filtration with sections $\overline{\Delta}(\nu)=A\otimes_{B}L^0(\nu)$ with~$\nu\sim_H \lambda$, by Definition~\ref{defB}(I). Hence, by equation~\eqref{deltabar}, $[\Delta(\lambda):L(\mu)]\not=0$ implies there is a $\nu\in\Lambda$ with
$$\mu\,\le_{\cQ}\,\nu\,\sim_H \lambda\qquad\mbox{and hence}\qquad \mu\preccurlyeq_{\cQ} \lambda.$$
So condition~\ref{DefBorPair}(2) for~$\preccurlyeq_{\cQ}$ is satisfied.

For the proof of part (2) we only need to observe that $<_{\cQ}$ and~$\prec_{\cQ}$ are identical and that, since $H$ is quasi-local, now $\Delta(\lambda)$ has a filtration where each section is isomorphic to~$\overline{\Delta}(\lambda)$.\end{proof}

\begin{rem}
By the proofs in Sections \ref{SecPreSpec} and \ref{ssspecial}, we find that Corollary~\ref{CorgrBrBo} remains true for a $\mN$-graded subalgebra~$B$, such that~$A_B$ and~${}_HB$ are projective, and for which
$$
A_+\subset AB_+\qquad\mbox{and}\qquad
A_0=H\oplus (A B_+)_0.
$$
However, in all examples of quasi-hereditary algebras with exact Borel subalgebras we will encounter in Part II, the stronger condition \eqref{eqComplete} is satisfied.
\end{rem}


\section{Base stratified algebras}
\label{SecBase}
\subsection{Definition}
The following definition is inspired by \cite[Propositions 5.1 and 5.2]{HHKP}.
\begin{ddef}\label{BasedId}
A {\em base stratifying ideal} in an algebra~$A$ is an idempotent ideal~$J=AeA$ which is exactly standardly stratifying, {\it i.e.} ${}_AJ$ and~$J_A$ are projective, see \ref{SecStratId2}, and such that~$eAe$ is standardly based, see Section~\ref{SecSBalg}.
\end{ddef}

\begin{ex}\label{ExQHBS}
By Corollary~\ref{CorB}, {\em any} exactly standardly stratifying ideal can be given the structure of a base stratifying ideal, when working over an algebraically closed field.
\end{ex}

\begin{rem}
By equation~\eqref{MoreAe0} and Remark~\ref{remMorSB}, the definition above does not intrinsically depend on the choice of the idempotent~$e$.
\end{rem}

We use Definition~\ref{BasedId} to introduce a generalisation of the concept of ``cellularly stratified algebra'' of~\cite[Definition~2.1]{HHKP}, see also \cite[Section~5]{HHKP}.
\begin{ddef}\label{DefBS}
Consider an algebra~$A$ with an $n$-decomposition~$\cQ$: $\Lambda=\sqcup_{i=0}^n\Lambda_i$. We consider the corresponding chain of idempotent ideals $\{J_i\,|\, 0\le i\le n+1\}$ of~\ref{SecChain}. Then $(A,\cQ)$ is {\em base stratified} if each $J_i/J_{i-1}$ is a base stratifying ideal in $A/J_{i-1}$.
\end{ddef}

\begin{rem}
It follows easily from Remark~\ref{remMorSB} and the equivalence of Definition~\ref{DefA1}(2) and~\ref{DefA2}(2) that a base stratification is a property of the Morita class of an algebra.
\end{rem}

\begin{rem}\label{RemParTriv}
A trivial example of a base stratified algebra is any standardly based algebra with the $0$-decomposition of~$\Lambda$. It is hence essential to keep track of the decomposition~$\cQ$ of a base stratified algebra~$(A,\cQ)$ in order to make non-trivial statements.
\end{rem}

\subsubsection{}\label{DefL} For a base stratified algebra~$(A,\cQ)$, we make a choice of idempotents $\{f_l\,|\, 1\le l\le n+1\}$ as in Remark \ref{DefAi} and consider the corresponding algebras $A^{(l)}$. By definition $A^{(l)}$ is standardly based for some poset $L_l$.
We consider the set $L:=\sqcup_{i=0}^n L_i=\{(i,p)\,|\, 1\le i\le n\,,\; p\in L_i\}$. We make $L$ into a poset by setting~$(i,p)\le (j,q)$ if~$i>j$, or~$i=j$ and~$p\le q$, so we take the lexicographic ordering, up to reversal of the order on natural numbers.

\begin{thm}\label{ThmBaSt}
A base stratified algebra~$(A,\cQ)$ is standardly based for the poset~$L$. The cell modules of~$A$ satisfy
$$W(i,p)\;\cong\; A f_{i+1}\otimes _{f_{i+1}Af_{i+1}} W^0_i(p),$$
with~$W^0_i(p)$ the cell modules of~$A^{(i)}$, regarded as $f_{i+1} Af_{i+1}$-modules with trivial~$f_{i+1}J_{i}f_{i+1}$-action.

\end{thm}

\begin{proof}
We use induction on $n$, where $\cQ$ is an $n$-decomposition. If $n=0$, there is nothing to prove. Assume the statement is true for~$n-1$ and consider the case~$n$. Clearly the algebra~$\tilde{A}:=A/J_1$ is base stratified and hence standardly based for the poset $L'=\sqcup_{i=1}^n L_i$. We fix corresponding ideals $\tilde{A}^{\ge (l,p)}$ with~$1\le l\le n$ and~$p\in L_l$.

We set $K:=A^{(0)}=f_1Af_1$. By Definition~\ref{BasedId} and Lemma \ref{LemStu} we have
$$J_1\;\cong\; Af_1\otimes_{f_1Af_1}f_1A\;\cong\; Af_1\otimes_K K\otimes_K f_1A.$$
By assumption, $K$ is standardly based and we set
$$A^{\ge (0,p)}\;:=\;Af_1\otimes_K K^{\ge p}\otimes_K f_1A,$$
for all $p\in L_0$. The fact that the quotients $A^{\ge (0,p)}/A^{> (0,p)}$ satisfy the desired properties then follows from the exactness of the functors $Af_1\otimes_K-$ and~$-\otimes_K f_1A$ in Lemma \ref{LemStu}. In particular, we have
$$W(0,p)\;\cong\; A f_{1}\otimes _{K} W^0_0(p),$$

The set of ideals of~$A$ is then completed by taking~$A^{\ge(l,p)}$ to be an arbitrary ideal in $A$ containing~$J_1$ with~$A^{\ge(l,p)}/J_1=\tilde{A}^{\ge (l,p)}$ for~$l>0$ and~$p\in L_l$.
\end{proof}

\begin{rem}
It follows from the considerations in Remark \ref{MoreAe} that the cell module~$W(i,p)$ of~$A$ does not depend on the actual choice of idempotents $\{f_i\}$.
\end{rem}

Our setup gives a very simple proof of a generalisation of~\cite[Theorem 10.2]{HHKP}.
\begin{prop}\label{corHKKP}
Maintain the notation and assumptions of Theorem~\ref{ThmBaSt}. 
\begin{enumerate}
\item For~$0\le i\le j\le n$, $p\in L_i$ and~$q\in L_j$, we have
$$\Ext^k_A(W(i,p), W(j,q))\;\cong\; \delta_{i,j}\,\Ext^k_{A^{(i)}}(W^0_i(p),W_i^0(q)),\qquad\mbox{for all $k\in\mN$}.$$
\item The cell modules of~$A$ form a standard system for poset $L$ if and only if, for each $0\le i\le n$, the cell modules of~$A^{(i)}$ form a standard system for poset $L_i$. 
\end{enumerate}
\end{prop}
\begin{proof}
Part (1) follows from Theorem~\ref{ThmBaSt} and Lemma~\ref{NewLemStr}, part (2) from part~(1).\end{proof}

As, by definition, a base stratified algebra is exactly standardly stratified, it has (proper) standard modules $S(\lambda)$ and~$\overline{S}(\lambda)$, for~$\lambda\in\Lambda$, see Section~\ref{SecStrat}. For easy comparison with the cell modules we actually write $S(i,\lambda)$, instead of $S(\lambda)$, for~$\lambda\in\Lambda_i\subset\Lambda$, with same convention for proper standard modules. 

\begin{cor}
Consider a base stratified algebra~$(A,\cQ)$. For any fixed $i$, $\lambda\in\Lambda_i$ and~$p\in L_i$, we have
\begin{itemize}
\item $S(i,\lambda)$ has a filtration where the sections are modules $W(i,q)$, $q\in L_i$ with~$q\ge \lambda$;
\item $W(i,p)$ has a filtration where the sections are modules $\overline{S}(i,\mu)$,~$\mu\in\Lambda_i$.
\end{itemize}
\end{cor}
\begin{proof}
The filtration of~$S(i,\lambda)$ is induced from the filtration of the projective left $A^{(i)}$-module using its cell modules. The filtration of~$W(i,p)$ is induced from the Jordan-H\"older decomposition series of the $A^{(i)}$-module~$W^0_i(\lambda)$.
\end{proof}

\subsection{Base-Borelic pairs}\label{BBP}
Given an algebra~$A$, we fix a decomposition of unity as in~\eqref{decomp1}. Assume that we have a graded pre-Borelic algebra~$B$ as in Definition~\ref{GrPreBor}. We have the corresponding decomposition $\cQ$ as in Section~\ref{SecPreSpec}.

\begin{thm}\label{ThmIdGrB}
Consider an algebra~$A$ with a graded pre-Borelic algebra~$B$ such that~$H=B_0$ is standardly based. Then $(A,\cQ)$ is base stratified.
\end{thm}
\begin{proof}
Corollary~\ref{CorgrBrBo}(1) and the equivalence of Definition~\ref{DefA1} and~\ref{DefA2} implies that the chain \eqref{chainProp} is exactly standardly stratifying.
Equation \eqref{eqAiH} then implies the chain is base stratifying.\end{proof}

\begin{cor}\label{CorStB}
Consider an algebra~$A$ with a graded pre-Borelic algebra~$B$ such that~$H$ is standardly based. For each $0\le i\le n$ consider the poset $L_i$ (containing~$\Lambda_i$) of the standardly based algebra~$ H e^\ast_i$. Then $A$ is standardly based for the poset $L$ defined in \ref{DefL}.

Moreover, the cell modules over $A$ form a standard system for~$L$, if and only if the cell modules for~$H e^\ast_i$ form a standard system for~$L_i$, for each $0\le i\le n$.
\end{cor}
\begin{proof}
The first statement follows from the combination of Theorems~\ref{ThmIdGrB} and~\ref{ThmBaSt}.
The claim about standard systems follows from Proposition~\ref{corHKKP}(2) and Theorem \ref{chainProp}(1).
\end{proof}

\begin{prop}\label{CellB}
Maintain the notation and assumptions of Corollary~\ref{CorStB}. The cell modules of the algebra~$A$ satisfy
$$W(i,p)\cong\Delta_{W^0_i(p)}=A\otimes_B W^0_i(p),$$
with~$W^0_i(p)$ the cell module for~$He^\ast_i$, with~$p\in L_i$.
\end{prop}
\begin{proof}
This follows from Lemma \ref{LemABAM} and Theorem~\ref{ThmBaSt}.\end{proof}

\subsection{Schur algebras}
The following observation is implicit in \cite[Sections~11-13]{HHKP}.
\begin{lemma}\label{LemSchur}
Consider a standardly based algebra~$A$. The cell modules $\{W(p), \,p\in L\}$, form a standard system for~$(L_A,\le)$ in $A${\rm -mod} if and only if A admits a cover-Schur algebra as in Definition \ref{DefSchur}.\end{lemma}
\begin{proof}
The existence of a Schur algebra clearly implies that the cell modules form a standard system. Now consider a standardly based algebra such that~$\{W(p)\}$ forms a standard system.

Set $W=\oplus_p W(p)$. We denote by $\cF(W)$ the exact category of $A$-modules with filtrations where each section is a direct sum of summands of $W$. By \cite[Section~3]{DR}, $\cF(W)$ contains unique indecomposable objects $Y(p)$, with~$p\in L$, such that 
$\Ext^1_A(Y(p),W)=0$ and there is a surjection~$Y(p)\tto W(p)$ with kernel in $\cF(W)$. Since $P(\lambda)\in \cF(W)$, by Section~\ref{SecSBalg},
we have in particular $Y(\lambda)\cong P(\lambda)$ for~$\lambda\in\Lambda_A\subset L_A$. Hence, we can take $m_p\in \mN$ such that
$$A\cong \End_A(\bigoplus_{\lambda\in\Lambda_A}P(\lambda)^{\oplus m_\lambda})^{\op},\quad\mbox{and we define} \quad \cS:= \End_A(\bigoplus_{p\in L_A}Y(p)^{\oplus m_p})^{\op}.$$
Consider the idempotent $f\in\cS$ corresponding to the projection of $\bigoplus_{p\in L_A}Y(p)^{\oplus m_p}$ onto the summand~$ \bigoplus_{\lambda\in\Lambda_A}P(\lambda)^{\oplus m_\lambda}$. Then we have $f\cS f\cong A$ and we can interpret $\bigoplus_{p\in L_A}Y(p)^{\oplus m_p}$ as a $(A,\cS)$-bimodule of the form $f\cS$.

By~\cite[Section~3]{DR}, $(\cS,\le)$ is quasi-hereditary and the functor
$$G=\Hom_A(\bigoplus_{p\in L_A}Y(p)^{\oplus m_p},-)\;:\; A\mbox{-mod}\to \cS\mbox{-mod}.$$
is exact and fully faithful on $\cF(W)$ and maps the cell modules of~$A$ to the standard modules of~$\cS$.
Since the functor 
$$F=f-\;:\; \cS\mbox{-mod}\to A\mbox{-mod}.$$
satisfies $F\circ G=\Id$, it follows immediately that~$F$ induces isomorphisms of the appropriate homomorphism spaces and first extension groups, which implies $\cS$ is a $1$-faithful quasi-hereditary cover.
\end{proof}


\part{Examples}\label{Examp}

\section{Auslander-Dlab-Ringel algebras}
\label{SecADR}
In \cite{Auslander}, Auslander proved that each finite dimensional unital algebra~$R$ has a cover~$A(R)$ with finite global dimension. In \cite{DRAus}, Dlab and Ringel proved that~$A(R)$ is quasi-hereditary. In this section, we prove that~$A(R)$ admits an exact Borel subalgebra (while giving an independent proof of the quasi-heredity), under the weak assumption that~$R$ is Wedderburn. 

\subsection{The Auslander construction}\label{introAus}
There are several Morita equivalent versions of the algebra~$A(R)$ which appear in the literature. Let $R$ be a $\mk$-algebra.
Consider the {\em right} regular $R$-module and its radical, which is the Jacobson radical~$\cJ$ of~$R$. For each $i\ge 1$, we let $M_R^i$ be the largest direct summand of the {\em right} $R$-module~$R/\cJ^i$ such that no direct summand of~$M_R^i$ appears as a direct summand of some $M_R^j$ for~$j<i$. We set $n$ to be the largest integer for which $M_R^n\not=0$. Alternatively, $n$ is defined as the smallest positive integer for which $\cJ^{n}=0$.

The {\em ADR-algebra} $A(R)$ is the cover of $R$ defined as
$$A(R):=\End_R(\oplus_{i=1}^nM_R^i).$$
We label the simple {\em right} $R$-modules by~$\Lambda^0:=\Lambda_{R^{\op}}$. We have idempotents $\{c^\ast_\lambda\,|\,\lambda\in\Lambda^0\}$ in $R$ such that the top of the right $R$-module~$c_\lambda^\ast R$ is of type $\lambda$ and such that~$1_R=\sum_\lambda c^\ast_\lambda$. 

For each $\lambda\in \Lambda^0$, let $\ell(\lambda)$ be the Loewy length of~$c_\lambda^\ast R$. For~$l\in\mN$ we set $$c_{\ge l}^\ast=\sum_{\lambda\in\Lambda^0\,|\, \ell(\lambda)\ge l}c_\lambda^\ast,$$
so in particular $c_{\ge 1}^\ast=1_R$.
Then we have
$$M_R^i\;\cong\;c_{\ge i}^\ast R/\cJ^i,\quad\mbox{ for~$1\le i\le n$}.$$
In particular, every direct summand of~$M_R^i$ has Loewy length $i$.

\subsection{Construction of the Borelic subalgebra} 
From now on, we assume that~$R$ is Wedderburn, meaning~$R=S\oplus\cJ$, for a semisimple subalgebra~$S$. Then we have a direct sum of simple algebras
$$S\;=\;\bigoplus_{\lambda\in\Lambda^0} S c_\lambda^\ast,$$
where the idempotent $c_\lambda^\ast$ is now central in $S$. The space $c_\lambda^\ast S$ is then naturally a subspace of~$c_\lambda^\ast R/\cJ^i$, which is bijectively mapped to the top of~$c_\lambda^\ast R/\cJ^i$ under the canonical surjection.

We let~$e^\ast_i\in A(R)$ be the projection onto the summand~$M_R^{i}$ for~$1\le i\le n$. This yields a decomposition of unity as in~\eqref{decomp1}, with $e_0^\ast=0$.
We define a subspace $B=\bigoplus_{i,j}e_i^\ast B e_j^\ast$ of~$A(R)$, by
$$e_i^\ast B e_j^\ast \;:=\;\{\alpha\in\Hom_R(M_R^j,M_R^i)\,|\, \alpha(c_{\ge j}^\ast) \in c_{\ge i}^\ast S\subset M^i_R\},\qquad\forall\;1\le i,j\le n.$$
In particular, we find $e_i^\ast B e_j^\ast =0$ if~$i>j$, by considering Loewy lengths. If $i\le j$, then $c_{\ge j}^\ast c_{\ge i}^\ast =c_{\ge j}^\ast $, so $\alpha\in e_i^\ast Be_j^\ast$ even implies $\alpha(c_{\ge j}^\ast) \in c_{\ge j}^\ast S$.

\begin{lemma}\label{LemmaAus}
The subspace $B$ is a graded pre-Borelic subalgebra of~$A(R)$ regarded as an idempotent graded algebra.
\end{lemma}
\begin{proof}
Set $A:=A(R)$. By construction, $B$ is a $\mZ$-graded subspace of~$A$ containing~$1_A$.
Now we prove that~$B$ is in fact an algebra. Take $\alpha\in e^\ast_i B e^\ast_k$ and~$\beta\in e^\ast_k Be^\ast_j$ with~$i\le k\le j$. It follows easily that~$\alpha\beta\in B$ from the fact that~$S$ is a subalgebra of~$R$.

To prove that~$B$ is complete, it suffices to prove equation~\eqref{eqComplete2}, {\it viz.} that for~$i\le j$,
\begin{equation}\label{rewcom}e^\ast_i Ae^\ast_j\;=\; e^\ast_i B e^\ast_j\oplus \sum_{k <i}e^\ast_i Ae^\ast_k Ae^\ast_j.\end{equation}
For this we observe that the morphisms in $e_i^\ast Be_j^\ast$ are spanned by the surjections from indecomposable summands in $M_R^j$ onto their factor modules appearing in $M_R^i$. A natural complement of that space is given by morphisms from $M_R^j$ with image in the radical of~$M_R^i$. As this image will have Loewy length strictly lower than $i$, the morphism factors through $M_R^k$ (a direct sum of all factor modules of Loewy length $k$ of the projective $R$-modules) for some $k<i$. This complement is hence precisely $\sum_{k <i}e^\ast_i Ae^\ast_k Ae^\ast_j$.

Now we show that the complete $\mN$-graded subalgebra~$B$ is pre-Borelic. The algebra
\begin{equation}\label{eqHADR}H=B_0\cong\bigoplus_{i=1}^n c_{\ge i}^\ast S\end{equation} is semisimple and therefore ${}_HB$ is projective.
Now we prove that~$A_B$ is projective. Take first $\alpha\in e_k^\ast Ae_1^\ast$, for some $k\ge 1$. By construction, $\alpha$ corresponds to a monomorphism $D\hookrightarrow M_R^k$, with $D$ a direct summand of the semisimple module $M_R^1$. Let $e$ denote the idempotent in $e_1^\ast Ae_1^\ast\cong S$ such that~$D=M^1_Re$. In particular, we have $\alpha=\alpha e$. Since $\alpha$, restricted to~$D$ is injective, it follows that the canonical epimorphism $eB\tto \alpha B$ given by $x\mapsto \alpha x$ is an isomorphism. It follows that
$$Ae_1^\ast A=Ae_1^\ast B$$
is a projective right $B$-module.

Now assume that we have proved that~$A(\sum_{k<i}e_i^\ast)A$ is projective as a right $B$-module, for~$i>1$. 
We take $\alpha\in Ae_i^\ast$ and associate a right $R$-module
$$N_\alpha=\{v\in M^i_R\,|\,\alpha(v)R\mbox{ has Loewy length strictly less than $i$.}\}$$
The image of $N_\alpha$ under the projection $M^i_R\tto \Top M^i_R$, yields a direct summand~$X_{\alpha}$ of $\Top M^i_R$. Note that~$N_\alpha=N_{\alpha'}$ if and only if $X_\alpha=X_{\alpha'}$, since the radical of $M^i_R$ has Loewy length $i-1$. We have
$X_\alpha=\Top(M^i_R) f$ for some idempotent $f\in Sc^\ast_{\ge i}\subset e_i^\ast Ae_i^\ast$. It follows that~$N_\alpha=N_e$, for the idempotent $e:=e_2^\ast-f$.
Observe that~$A(\sum_{k<i}e_i^\ast)A$ is the ideal of morphisms whose image has Loewy length strictly less than $i$. For~$\beta\in e_i^\ast B$, we thus have $\alpha\beta\in A(\sum_{k<i}e_i^\ast)A$ if and only if $\im\beta\subset N_\alpha=N_e$.
Working inside the right $B$-module $A/(A(\sum_{k<i}e_i^\ast)A)$, it thus follows that~$\alpha B\cong eB$.

Hence we find that~$A(\sum_{k\le i}e_i^\ast)A/A(\sum_{k<i}e_i^\ast)A$ is also projective. This means that also $A(\sum_{k\le i}e_i^\ast)A$ is projective. It follows that $A_B$ is projective by iteration, which concludes the proof.
\end{proof}

\subsection{Main theorem on Auslander-Dlab-Ringel algebras}

\begin{thm}\label{ThmAus}
Consider a field $\mk$. If the algebra~$R$ is Wedderburn (for instance if~$\mk$ is perfect),
the simple modules of the ADR algebra~$A(R)$
are labelled by 
\begin{equation}\label{LAR}\Lambda_{A(R)}=\{(i,\lambda)\; \mbox{with }\;1\le i\le \ell(\lambda) \mbox{ and }\lambda\in \Lambda^0\}.\end{equation}
Moreover, $\left(A(R),\le_{\cQ}\right)$ is quasi-hereditary and admits an exact Borel subalgebra.
\end{thm}
\begin{proof}
Since $H$ in equation~\eqref{eqHADR} is semisimple, the statement follows from Lemma~\ref{LemmaAus}, by using Corollary~\ref{CorgrBrBo}(3) and equation \eqref{eqLABH}.\end{proof}

\begin{rem}
By \cite{DRAus}, the ADR algebra is quasi-hereditary, regardless of wether~$R$ is Wedderburn. However, if~$R$ is not Wedderburn, Theorem \ref{ThmAus} does not yield an exact Borel subalgebra. In this case, the field $\mk$ cannot be algebraically closed and the results in~\cite{KKO} also do not ensure the existence of exact Borel subalgebras in Morita equivalence classes.
\end{rem}

\subsection{All algebras are standardly based}

\begin{thm}\label{AllBased}
Any algebra over an algebraically closed field $\mk$ is standardly based.
\end{thm}
\begin{proof}
Consider an arbitrary algebra~$R$. As the field is perfect, Theorem~\ref{ThmAus} implies that~$A(R)$ is quasi-hereditary. By \cite[Theorem~4.2.7]{JieDu}, $A(R)$ is standardly based. As $R\cong e^\ast A(R)e^\ast$ for some idempotent~$e^\ast$, it follows that~$R$ is also standardly based, by \cite[Proposition~3.5]{Yang}.
\end{proof}

\subsubsection{}\label{Explanation}The standardly based structure can be derived explicitly from the above proof. Take a simple module~$L$ in the socle of the right regular $R$-module. Acting on this with the left $R$-action leads to a two-sided ideal~$W\otimes L$ in $R$ for some left $R$-module~$W$. Factoring out this ideal and continuing the procedure yields the structure.

\begin{rem}
Jie Du informed us that he was aware of this result. In the unpublished manuscript~\cite{preprint}, it is proved that any split finite dimensional algebra over any field is standardly based, using the same approach as in \ref{Explanation}.
\end{rem}


\section{Example from Lie theory}\label{SecBGG}
We use the general theory to construct a very elementary Lie theoretic example of a properly stratified algebra with exact Borel subalgebra. A more advanced example can be found in \cite[Section~7]{Kluc}. In this section we set $\mk=\mC$.

\subsection{Thick category~$\cO$}
We consider the category~$\cO^{k}$, for~$k\in\mZ_{\ge 1}$, studied in e.g. \cite{Soergel}. For a reductive Lie algebra~$\fg$, with triangular decomposition~$\fg=\fn^-\oplus\fh\oplus \fn^+$ and Borel subalgebra~$\fb=\fh\oplus\fn^+$, the category~$\cO^{k}$ is the category of all $\fg$-modules which
\begin{itemize}
\item are finitely generated;
\item have a basis, where each basis element~$v$ satisfies $h^kv\in\Span\{h^jv\,|\, j<k\}$, if~$h\in\fh$;
\item are locally $U(\fn^+)$-finite.
\end{itemize} If $k=1$, we get the ordinary BGG category $\cO$.
For a module~$M$ in $\cO^k$ and~$\nu\in\fh^\ast$, we consider the generalised weight space
$$M_{(\nu)}:=\{v\in M\,|\, (h-\nu(h))^k v=0,\;\mbox{ for all }\,h\in\fh\}, \quad\mbox{so }\;M=\bigoplus_{\nu\in\fh^\ast}M_{(\nu)}.$$
The category~$\cO^k$ decomposes into subcategories $\cO^k_\chi$ based on the central characters $\chi$ of~$U(\fg)$. The simple objects in $\cO^k_\chi$ are the simple highest weight modules $L(\mu)$ with~$\mu$ in the Weyl group orbit corresponding to~$\chi$, see \cite[Section~1.7]{Humphreys}, which we denote by~$\Lambda$. We consider the Bruhat (partial) order $\le$ on $\Lambda$ of~\cite[Section~5.2]{Humphreys}. The module~$\widetilde{M}_{n,k}(\lambda)$ with~$\lambda\in \Lambda$, for~$n\in\mN$, introduced in \cite[Section~4]{SHPO}, is projective and does not depend on $n$, for~$n$ large enough, by \cite[Proposition~12]{SHPO}. We denote this module by~$P_\lambda$. By the construction in \cite[Section~4]{SHPO}, this module is generated by a vector~$v_\lambda\in (P_\lambda)_{(\lambda)}$, and we have an isomorphism
$$\Hom_{\cO^k}(P_\lambda,M)\,\;\tilde{\to}\;\, M_{(\lambda)};\quad \alpha\mapsto \alpha(v_\lambda). $$
As $\oplus_\lambda P_\lambda$ is a projective generator of~$\cO_\chi^k$, see \cite{SHPO, KKM, Soergel}, we have
$$\cO_\chi ^k\cong A\mbox{-mod}\qquad\mbox{with}\qquad A:=\End_{\cO_\chi^k}\left(\bigoplus_{\lambda\in\Lambda}P_\lambda\right)^{\op}.$$
We denote the identity morphism of~$P_\lambda$ by~$e^\ast_\lambda$. The modules $P_\lambda$ are not indecomposable, but they represent precisely the modules with corresponding notation in Lemma~\ref{PPlambda}.
\subsection{Main theorem on thick $\cO$}
The following theorem generalises~\cite[Theorem~D]{Koenig}.
\begin{thm}
The algebra~$(A,\le)$ is properly stratified with exact Borel subalgebra.
\end{thm}
The fact that~$(A,\le)$ is properly stratified was already pointed out in \cite[Corollary~9(a)]{KKM}, due to the Morita equivalence in \cite[Proposition~1]{Soergel}.
\begin{proof}
Let us define subalgebras $B,H,\overline{N}$ of~$A$. For arbitrary~$\lambda,\mu\in\fh^\ast$, we set
\begin{eqnarray*}
e^\ast_\lambda Be^\ast_\mu&:=&\{\alpha\in \Hom_{\cO^k}(P_\lambda,P_\mu)\,|\, \alpha(v_\lambda) \in U(\fb)v_\mu\}\\
e^\ast_\lambda He^\ast_\mu&:=&\{\alpha\in \Hom_{\cO^k}(P_\lambda,P_\mu)\,|\, \alpha(v_\lambda) \in U(\fh)v_\mu\}\\
e^\ast_\lambda \overline{N}e^\ast_\mu&:=&\{\alpha\in \Hom_{\cO^k}(P_\lambda,P_\mu)\,|\, \alpha(v_\lambda) \in U(\fn^-)v_\mu\}.
\end{eqnarray*}
By construction, $e_\lambda^\ast He_\mu^\ast=0$ unless $\mu=\lambda$ and~$He_\lambda^\ast\cong \mC[x]/(x^k)$ is a local algebra.
Since $P_\mu$ is generated by $v_\mu$, it follows from the PBW theorem in \cite[Section 0.5]{Humphreys} that $P_\mu$ is spanned by vectors of the form $u_1u_2v_\mu$, where $u_2\in U(\fb)$ and~$u_1\in U(\fn^-)$. Moreover, since all simple constituents have highest weight  in $\Lambda$, it follows that we can take a basis of such vectors where each~$u_2v_\mu\in (P_\mu)_{(\nu)}$ for some $\nu\in\Lambda$. This implies that
$A=\overline{N}B.$

Now consider $\alpha\in e^\ast_\lambda\overline{N}e^\ast_\mu$, with $\alpha(v_\lambda)=Yv_\mu$, for some $Y\in U(\fn^-)$. For any $\beta\in e^\ast_\mu Be^\ast_\nu$, with $\beta(v_\mu)=X_\beta v_\nu$, we have
$$\alpha\beta(v_\lambda)=\beta(Y v_\mu)=YX_\beta v_\nu.$$ 
Since the modules $P_\mu$ are $U(\fn^-)$-free, we find that $\alpha\beta=0$ if and only if $\beta=0$. It follows that $\alpha B\cong e^\ast_\mu B$ 
and hence that~$A_B$ is projective. Similarly it follows that~${}_HB$ is projective.
From the PBW theorem we find
\begin{equation}\label{ConsPBW}U(\fg)\;=\; U(\fn^-)U(\fb)\,=\, U(\fb)\,\oplus\, \fn^-U(\fg)\,=\, U(\fb)\,\oplus\, U(\fg)_{<} U(\fg),\end{equation}
where $U(\fg)_{<}$ is the subspace of~$U(\fg)$ of all elements which are negative weight vectors for the adjoint $\fh$-action.
Now consider an arbitrary extension~$\le^e$ of~$\le$, which is a total order, meaning we can identify $(\Lambda,\le^e)$ with a subset of~$\mN$. The basis of~$P_\mu$ mentioned above and equation~\eqref{ConsPBW} imply equation~\eqref{eqComplete2}. We thus find that~$B$ is a graded exact Borelic subalgebra. The results then follow from Corollary~\ref{CorgrBrBo}(2).
\end{proof}


\section{Algebras in diagram categories}\label{Sec8}
\label{SecDia}
We will obtain many examples of the types of algebras introduced in Part I, as algebras of morphisms in the partition category. This category is a natural generalisation of the Brauer category of~\cite{Deligne, BrCat} and the partition algebra of~\cite{Jones, PartM}. 

\subsection{Category algebras}\label{GrAlg}
For a category~$\cC$ with finitely many objects and morphisms, the {\em category algebra} $\mk[\cC]$ is given, as a vector space, by all formal sums of morphisms in $\cC$ with coefficients in $\mk$. This is an algebra for the natural product of composition. In particular, for a finite group $G$, which is a category with one object where all morphisms are isomorphisms, we denote the group algebra by~$\mk G$.
If $\cC$ is {\em $\mk$-linear} with finitely many objects and finite dimensional morphism spaces, we define the category algebra as
 $$\mk[\cC]\;:=\;\bigoplus_{x,y\in \Ob\cC}\Hom_{\cC}(x,y)\qquad\mbox{with}\qquad e^\ast_y\mk[\cC]e^\ast_x= \Hom_{\cC}(x,y),$$
 where~$e^\ast_z$ is the identity morphism of~$z\in\Ob\cC$.

\subsection{Categories of diagrams}\label{SecDefCat}
Consider an arbitrary field $\mk$ and a fixed element~$\delta\in\mk$. 

\subsubsection{Partition category}
The partition category~$\cP(\delta)$ is $\mk$-linear with set of objects $\mN$. 
A $\mk$-basis of~$\Hom_{\cP(\delta)}(i,j)$ is given by all partitions of the set $\{1,2,\cdots,i,1',2',\cdots, j'\}$. We will also view a partition as an equivalence relation on the set.

For a partition~$p$ of the set $S$ and a partition~$q$ of the set $T$, with
$$S=\{1,2,\cdots,i,1',2',\cdots, j'\}\quad\mbox{and}\quad T=\{1',2',\cdots,j',1'',2'',\cdots, k''\},$$ we define the partition~$q\ast p$ of the set
$$S\cup T=\{1,2,\cdots,i,1',2',\cdots, j',1'',2'',\cdots, k''\},$$ corresponding to the minimal equivalence relation generated by~$p$ and~$q$. 

We derive two properties from~$q\ast p$. Firstly, it induces a partition~$q\odot p$ of the set
$$\{1,2,\cdots,i,1'',2'',\cdots, k''\},$$ where two elements in the latter set are equivalent if and only if they were so in $S\cup T$. The second property is the number~$d(p,q)$ of equivalence classes in $q\ast p$ which are contained in $S\cap T=\{1',2',\cdots,j'\}$. Now we identify the partitions $p$, $q$ and~$q\odot p$ with basis elements in respectively $\Hom_{\cP(\delta)}(i,j)$, $\Hom_{\cP(\delta)}(j,k)$ and~$\Hom_{\cP(\delta)}(i,k)$. 
The product (composition) $qp=q\circ p$ is defined as $\delta^{d(p,q)}q\odot p$, which extends bilinearly to
 $$\Hom_{\cP(\delta)}(j,k)\,\times\,\Hom_{\cP(\delta)}(i,j)\;\to\;\Hom_{\cP(\delta)}(i,k).$$
 
It is easily checked that for the above definitions, $\cP(\delta)$ is indeed a ($\mk$-linear) category. For any $n\in\mZ_{>1}$, the partition algebra is
$$\cP_n(\delta)=\End_{\cP(\delta)}(n).$$

We will think graphically of the sets (and their partitions) by imagining the elements of~$\{1,2,\cdots,i\}$ to be $i$ dots on a horizontal line, ordered from left to right and the elements of~$\{1',2',\cdots, j'\}$ to be similarly ordered on a horizontal line above the other one.

\subsubsection{Brauer category}
The Brauer category~$\cB(\delta)$, as introduced in~\cite[Definition~2.4]{BrCat}, is a $\mk$-linear subcategory of~$\cP(\delta)$, with the same set of objects $\Ob\cB(\delta)=\mN$. The morphism spaces are spanned by the partitions into subsets containing exactly two elements. Such partitions will be graphically represented as {\em Brauer diagrams}.
A $(k,l)$-Brauer diagram consists of~$k$ points on a horizontal line, $l$ points on a parallel line above the first line and~$(k+l)/2$ lines, each connecting two points. Composing a $(k,l)$- and a $(i,k)$-Brauer diagram, by drawing the first on top of the second one and using the procedure of composing general partitions yields a $(i,l)$-Brauer diagram, up to a power of $\delta$. We have for instance
\begin{equation*}\label{defOA}
\begin{tikzpicture}[scale=0.9,thick,>=angle 90]
\begin{scope}[xshift=8cm]

\draw (0,0) to [out=90,in=-180] +(.6,.6);
\draw (.6,.6) to [out=0,in=90] +(.6,-.6);
\draw (.6,0) to [out=90,in=-180] +(.6,.6);
\draw (1.2,.6) to [out=0,in=90] +(.6,-.6);
\node at (2.2,0.5) {$\circ$};

\draw (2.6,0) to   [out=60,in=-120] +(1.2,1);
\draw (3.2,0) to   [out=110,in=-70] +(-0.6,1);
\draw (3.2,1) to [out=-90,in=-180] +(.6,-.6);
\draw (3.8,.4) to [out=0,in=-90] +(.6,.6);

\node at (5,0.5) {=};

\node at (5.6,0.5) {$\delta$};

\draw (5.9,0) to [out=90,in=-180] +(.3,.3) to [out=0,in=90] +(.3,-.3);

\node at (12.5,0.5) {in $\Hom_{\cB(\delta)}(4,0)\times\Hom_{\cB(\delta)}(2,4)\to \Hom_{\cB(\delta)}(2,0)$.};
\end{scope}
\end{tikzpicture}
\end{equation*}
For any $n\in\mZ_{>1}$, the Brauer algebra is
$$\cB_n(\delta)=\End_{\cB(\delta)}(n).$$

The lines in Brauer diagrams which connect the lower and upper line will be referred to {\em propagating lines}. Lines connecting two points on the lower line are called {\em caps} and lines connecting two points on the upper are {\em cups}.

\subsubsection{Walled Brauer category}
The walled Brauer algebra~$\cB_{r,s}(\delta)$ is a subalgebra of the Brauer algebra~$\cB_{r+s}(\delta)$ satisfying 
$$\cB_{r,s}(\delta)\cong \cB_{s,r}(\delta)\quad\mbox{and} \quad\cB_{n,0}(\delta)\cong \mk\mS_n.$$

There are several options to define a ``walled Brauer category''. The most natural is as the category~$\underline{{\rm Rep}}_0({\rm GL}_\delta)$ of~\cite[Section~3.2]{Comes2}, see also~\cite{Deligne}, which has a tensor category structure. However, that category decomposes into blocks, and each such block is equivalent to a subcategory of the Brauer category~$\cB(\delta)$, which we realise as follows.

For any $p\in \mN$, the subcategory~${}^p\cB(\delta)$ of~$\cB(\delta)$ has set of objects $p+2\mN\subset \mN=\Ob\cB(\delta)$ and the morphisms $\Hom_{{}^p\cB(\delta)}(p+2i,p+2j)$ are spanned by the Brauer diagrams which are ``well-behaved'' with respect to a straight vertical line separating the left $p+i$ and~$p+j$ dots from the right $i$ and~$j$ dots on the two lines. Well-behaved means that propagating lines do not cross the straight line, but cups and caps intersect it precisely once. For~$r\ge s>0$, we have
$$\cB_{r,s}(\delta)=\End_{{}^{r\text{-}s}\cB(\delta)}(r+s).$$

\subsubsection{Jones category}
A {\em $(k,l)$-Jones diagram} is a partition of a set of~$k+l$ dots into pairs (so a Brauer diagram), which can be drawn without intersections when the $l$ dots are on the outer boundary of an annulus and the $k$ dots on the inner boundary of the annulus.
The Jones category~$J(\delta)$ is the $\mk$-linear subcategory of~$\cB(\delta)$, with the same set of objects $\mN$ and where the morphisms are spanned by the Jones diagrams.
For~$n\in\mZ_{>1}$, the Jones algebra is
$$J_n(\delta)=\End_{J(\delta)}(n).$$
For instance, we have $J_2(\delta)\cong\cB_2(\delta)$.

\subsubsection{Temperley-Lieb category}
The Temperley-Lieb category~$\TL(\delta)$ is a subcategory of the Brauer (or Jones) category with the same set of objects $\mN$, but morphisms are spanned by the diagrams without intersections. This is the category of~\cite[Section~2.2]{BFK}, specialised at $q$ with~$-q-q^{-1}=\delta$.
For~$n\in\mZ_{>1}$, the Temperley-Lieb algebra (of type A) is
$$\TL_n(\delta)=\End_{\TL(\delta)}(n).$$
For instance, we have $\TL_2(\delta)\cong \cB_{1,1}(\delta)$.

\subsubsection{The category~$\FI$}
Usually, $\FI$ is introduced as the category of which the objects are all finite sets and morphisms are all injective maps between sets. We take the equivalent full subcategory with~$\mN$ as set of objects, where $n\in\mN$ is identified with some set of cardinality~$n$. This is a subcategory of~$\cP(\delta)$, for arbitrary~$\delta$, but is not $\mk$-linear.

\subsection{Triangular structure and truncation}\label{SecNHN}
We will distinguish three types of partitions. 
\begin{enumerate}
\item Partitions of type~$H$: These are partitions into subsets of exactly two elements, one on each line (one primed and one unprimed).
\item Partitions of type $N^+$: These are partitions where 
\begin{itemize}
\item each element of the upper line is contained in a set with at least one element of the lower line and no other elements of the upper line.
\item for~$k$, resp.  $l$, minimal in the set containing~$i'$, resp. $j'$, we have $i<j\Leftrightarrow k<l$.
\end{itemize}
\item Partitions of type $N^-$: These are partitions where \begin{itemize}
\item each element of the lower line is contained in a set with at least one element of the upper line and no other elements of the lower line.
\item for~$i'$, resp.  $j'$, minimal in the set containing~$k$, resp. $l$, we have $k<l\Leftrightarrow i<j$.
\end{itemize}
\end{enumerate}
By definition, partitions of type~$H$ must appear in $\Hom_{\cP(\delta)}(j,j)$, for some $j\in\mN$, while those of type $N^+$, resp. $N^-$, must appear in $\Hom_{\cP(\delta)}(i,j)$, for~$i<j$, resp. $i>j$.

As a special case we have Brauer diagrams of the three corresponding types.
\begin{enumerate}
\item Brauer diagrams of type~$H$: These diagrams consist solely of propagating lines.
\item Brauer diagrams of type $N^+$: After removing all caps one is left with only non-crossing propagating lines.
\item Brauer diagrams of type $N^-$: After removing all cups one is left with only non-crossing propagating lines.
\end{enumerate}
The subspace of one of the category algebras spanned by all diagrams of type $H$ is also denoted by $H$, with similar convention for~$N^+$ and~$N^-$.
We draw some examples:
\begin{equation*}\label{defOA}
\begin{tikzpicture}[scale=0.9,thick,>=angle 90]
\begin{scope}[xshift=8cm]

\draw (0,1) to [out=-90,in=-180] +(.6,-.6);
\draw (.6,.4) to [out=0,in=-90] +(.6,.6);
\draw (.6,1) to [out=-90,in=-180] +(.6,-.6);
\draw (1.2,.4) to [out=0,in=-90] +(.6,.6);
\node at (2.7,0.5) {$\in N^-$,};

\draw (4.6,0) to   [out=60,in=-120] +(1.2,1);
\draw (5.2,0) to   [out=110,in=-70] +(-0.6,1);
\draw (5.8,0) to   [out=110,in=-70] +(-0.6,1);
\draw (6.4,0) to   [out=90,in=-90] +(0,1);

\node at (7.1,0.5) {$\in H,$};

\draw (9.6,0) to   [out=110,in=-70] +(-0.6,1);
\draw (10.8,0) to   [out=130,in=-50] +(-1.2,1);

\draw (10.2,0) to [out=90,in=-180] +(.6,.4);
\draw (10.8,.4) to [out=0,in=90] +(.6,-.4);

\draw (9,0) to [out=75,in=-180] +(1.5,.8);
\draw (10.5,.8) to [out=0,in=105] +(1.5,-.8);

\node at (12.7,0.5) {$\in N^+$.};

\end{scope}
\end{tikzpicture}
\end{equation*}

For all the categories, except ${}^p\cB(\delta)$, we denote the identity morphism in $\End(i)$ by~$e_i^\ast$.
For~${}^p\cB(\delta)$ we use~$e_i^{\ast}$ for the identity morphism of~$p+2i\in \Ob{}^p\cB(\delta)$.

For each $n\in\mZ_{>1}$, we introduce the full subcategory~$\cP^{\le n}(\delta)$, resp. $\FI^{\le n}$, of~$\cP(\delta)$, resp. $\FI$, with objects $\{0,1,2\cdots, n\}\subset \mN$. 
Their category algebras satisfy \eqref{decomp1}.

We observe that $\cB(\delta)$ decomposes into two subcategories, one with objects $2\mN$ and one with objects $2\mN+1$. Hence, for~$n\in\mZ_{>1}$, we define $\cB^{\le n}(\delta)$, $J^{\le n}(\delta)$ and~$\TL^{\le n}(\delta)$ as the full subcategories of resp. $\cB(\delta)$, $J(\delta)$ and~$\TL(\delta)$, with objects
\begin{equation}
\JJJ(n)\;:=\; \{n,n-2,\cdots,n-2\lfloor \frac{n}{2}\rfloor\}.
\end{equation}
The corresponding category algebras satisfy $1=\sum_{i\in\JJJ(n)}e^\ast_i$.
For~$p\in\mN$ and~$n\in\mZ_{>0}$, the category~${}^p\cB^{\le n}(\delta)$ is the full subcategory of~${}^p\cB(\delta)$ with objects
$$\{p,p+2,\cdots,p+2n\}.$$ The category algebra satisfies equation~\eqref{decomp1}.

\subsection{The category algebras}
We consider the category algebras of the previous section with idempotents $e_i^\ast$. We use the corresponding decomposition~$\cQ$ of~$\Lambda$, and the idempotent quasi-order $\preccurlyeq_{\cQ}$ and idempotent partial order $\le_{\cQ}$ on $\Lambda$ of Definition~\ref{lepleq}.
\begin{thm}\label{ThmDiagram1}
Fix an arbitrary field $\mk$ and~$\delta\in\mk$. 
\begin{enumerate}
\item The algebras $\mk[\cP^{\le n}(\delta)]$, $\mk[\cB^{\le n}(\delta)]$, $\mk[{}^p\cB^{\le n-p}(\delta)]$ and~$\mk[\FI^{\le n}(\delta)]$ are
\begin{itemize}
\item quasi-hereditary for~$\le_{\cQ}$, with exact Borel subalgebra if~$\charr(\mk)\not\in [2,n]$;
\item exactly standardly stratified for~$\preccurlyeq_{\cQ}$, with exact Borel subalgebra;
\item base stratified for decomposition~$\cQ$.
\end{itemize}
\item The algebra~$\mk[J^{\le n}(\delta)]$ is
\begin{itemize}
\item quasi-hereditary for~$\le_{\cQ}$ with exact Borel subalgebra, if~$\charr(\mk)$ does not divide any element of~$\JJJ(n)$;
\item properly stratified for~$\le_{\cQ}$, with exact Borel subalgebra;
\item base stratified for decomposition~$\cQ$ if the polynomials $x^{i}-1$, for~$i\in\JJJ(n)$, split over the field $\mk$.
\end{itemize}
\item The algebra~$\mk[\TL^{\le n}(\delta)]$ is
\begin{itemize}
\item quasi-hereditary for~$\le_{\cQ}$, with exact Borel subalgebra;
\item base stratified for decomposition~$\cQ$.
\end{itemize}
\end{enumerate}
When the above condition for quasi-heredity is not satisfied, the category algebra is not quasi-hereditary, for any partial order on $\Lambda$ (with~$\Lambda$ given in Lemma~\ref{LemExBH}).
\end{thm}

\begin{rem}
The statements on the exact standard stratification in Theorem~\ref{ThmDiagram1}(1) can be refined by replacing~$\preccurlyeq_{\cQ}$ by a partial quasi-order $\preccurlyeq$, such that~$<$ is the same as $\prec_{\cQ}$, but $\sim$ reflects the block decomposition of~$\mk \mS_t$, see \cite[Section~21]{James}.
\end{rem}

If $\mk$ is algebraically closed, the algebras in Theorem~\ref{ThmDiagram1} are all standardly based by Theorem~\ref{ThmBaSt}, with poset $L:=\sqcup_{i}L_i$ and~$L_i$ the posets for the standardly based algebra~$He_i^\ast$, as can be obtained from Examples \ref{ExS1} and~\ref{ExS2} and Lemma~\ref{LemExBH}. 
\begin{thm}\label{ThmDiagram2}
Assume that~$\mk$ is algebraically closed. The cell modules of~$C$ form a standard system for~$(L,\le)$ if and only if the condition below is satisfied.
\vspace{-4mm}
\begin{center}
\begin{tabular}{ | l | l | l |  }
\multicolumn{3}{c }{}\\
\hline
algebra~$C$& condition& Set $L\,=\,\sqcup_{i}L_i$ \\ \hline\hline
$\mk[\cP^{\le n}(\delta)]$ & $\charr(\mk)\not\in\{2,3\}$ or~$\charr(\mk)=3$ and~$n=2$&$\{(i,\mu)\,|\, 0\le i\le n,\; \mu\vdash i\}$  \\ \hline
$\mk[\cB^{\le n}(\delta)]$ & $\charr(\mk)\not\in\{2,3\}$ or~$\charr(\mk)=3$ and~$n=2$ &$\{(i,\mu)\,|\, i\in \mathscr{J}(n),\; \mu\vdash i\}$\\ \hline
$\mk[J^{\le n}(\delta)]$ & $i$ not divisible by~$\charr(\mk)$, for~$i\in\JJJ(n)$ &$\{(i,\omega)\,|\, i\in \mathscr{J}(n),\; \omega\in\mk,\,\, \omega^i=1\}$  \\ \hline
$\mk[\TL^{\le n}(\delta)]$ & $\emptyset$&$\JJJ(n)$ \\ \hline
$\mk[\FI^{\le n}(\delta)]$ & $\charr(\mk)\not\in\{2,3\}$ or~$\charr(\mk)=3$ and~$n=2$& $\{(i,\mu)\,|\, 0\le i\le n,\; \mu\vdash i\}$ \\ \hline
$\mk[{}^p\cB^{\le n}(\delta)]$ & $\charr(\mk)\not\in\{2,3\}$&$\{(i,\mu,\nu)\,|\, 0\le i\le n,$ \\ 
 & or~$\charr(\mk)=3$ and~$n+p\le 2$&$\mu\vdash p+i,\,\nu\vdash i \}$ \\ 
\hline
\end{tabular}
\end{center}
\end{thm}

We start the proofs of the theorems with the following proposition.

\begin{prop}\label{PropExBH}
Let $C$ be one of the category algebras in Theorem \ref{ThmDiagram2}, equipped with the idempotent grading of \eqref{grBr}.
The subspaces $H,N^+$ and~$N^-$ of~$C$, defined in Section~\ref{SecNHN} are subalgebras and~$B:=HN^+$ is a graded pre-Borelic subalgebra of~$C$.
\end{prop}
\begin{proof}
One verifies that the subspaces are subalgebras for~$\cP(\delta)$. The other cases follow by restriction. Further, we can classify partitions into those where each element of the upper line is contained in a set/class with at least one element of the lower line and no other elements of the upper line, and the rest. In each of the cases, the first type of diagrams span $B$, whereas those of the second type span $C_-B$, proving equation~\eqref{eqComplete}. So we find that~$B$ is a complete $\mN$-graded subalgebra of~$C$.

It hence remains to prove that~$C_B$ and~${}_HB$ are projective modules. We have $C=N^-B$ and by associating to each partition the number of classes which contain elements of both rows (for Brauer diagrams this is just the number of propagating lines) we find a decomposition
$$C=\bigoplus_{i}N^- e_i^\ast B.$$
For any partition~$q$ in $N^- e_i^\ast$ we have $qB\cong e_i^\ast B$, proving that~$C_B$ is projective. Similarly it follows that for any partition~$q$ in~$e^\ast_i B$ we have $Hq\cong He_i^\ast$, which concludes the proof.
\end{proof}

For the cyclic group $C_t$ of order $t$, we have $\mk C_t\cong \mk[x]/(x^t-1)$. We denote its labelling set of simple modules by~$\Lambda_C^{\mk,t}$. In case the polynomial~$x^t-1$ splits over $\mk$, for instance when~$\mk$ is algebraically closed, we can take
$\{\omega\in\mk\,|\, \omega^t=1\}$ for this set.
\begin{lemma}\label{LemExBH}
The following table summarises the structure of the subalgebra~$H\subset C$, its labelling set $\Lambda_C$ of isoclasses of simple modules and the criterion for~$H$ to be semisimple. \vspace{-4mm}
\begin{center}
\begin{tabular}{ | l | l | l |  l |}
\multicolumn{4}{c }{}\\
\hline
algebra~$C$& algebra~$H$ & set $\Lambda_C=\Lambda_H=\sqcup_i \Lambda_i$ & semisimplicity criterion for~$H$\\ \hline\hline
$\mk[\cP^{\le n}(\delta)]$ & $\bigoplus_{i=0}^n \mk\mS_i $ &$\{(i,\mu)\,|\, 0\le i\le n,\; \mu\vdash_{\mk} i\}$&$\charr(\mk)\not\in[2,n]$  \\ \hline
$\mk[\cB^{\le n}(\delta)]$ &$\bigoplus_{i\in\JJJ(n)} \mk\mS_i $ & $\{(i,\mu)\,|\, i\in \mathscr{J}(n),\; \mu\vdash_{\mk} i\}$ &$\charr(\mk)\not\in[2,n]$  \\ \hline
$\mk[J^{\le n}(\delta)]$ & $\bigoplus_{i\in\JJJ(n)} \mk C_i $ & $\{(i,\omega)\,|\, i\in \mathscr{J}(n),\; \omega\in \Lambda_C^{\mk,i}\}$ &$\charr(\mk)\nmid i$, for~$i\in\JJJ(n)$  \\ \hline
$\mk[\TL^{\le n}(\delta)]$ &$\bigoplus_{i\in\JJJ(n)} \mk $ &$\JJJ(n)$&$\emptyset$ \\ \hline
$\mk[\FI^{\le n}(\delta)]$ & $\bigoplus_{i=0}^n \mk\mS_i $ &$\{(i,\mu)\,|\, 0\le i\le n,\; \mu\vdash_{\mk} i\}$ &$\charr(\mk)\not\in[2,n]$  \\ \hline
$\mk[{}^p\cB^{\le n}(\delta)]$ & $\bigoplus_{i=0}^n \mk\mS_{p+i}\times\mS_i$  & $\{(i,\mu,\nu)\,|\, 0\le i\le n,$ &$\charr(\mk)\not\in[2,p+n]$ \\ 
 &  & $\mu\vdash_{\mk} p+i,\,\nu\vdash_{\mk} i \}$ &  \\ 
\hline

\end{tabular}
\end{center}
\end{lemma}
\begin{proof}
The structure of the algebra~$H$ follows from its definition in Section \ref{SecDefCat}. We use equation~\eqref{eqLABH}. The simple modules of~$\mk\mS_t$ are labelled by~$\mk$-regular partitions of~$t$, see \cite[Section~11]{James}. 
By Maschke's theorem, for a finite group $G$, the algebra~$\mk G$ is semisimple if and only if the order $|G|$ is not divisible by~$\charr(\mk)$. Hence, $\mk\mS_t$ is semisimple if and only if~$\charr(\mk)>t$ or~$\charr(\mk)=0$.\end{proof}

\begin{proof}[Proof of Theorem~\ref{ThmDiagram1}]
We freely use Proposition~\ref{PropExBH} and Lemma~\ref{LemExBH}.

That the algebras are based-stratified follows from Theorem~\ref{ThmIdGrB} and Examples \ref{ExS1} and~\ref{ExS2}. 
The exact standard stratification follows from Corollary~\ref{CorgrBrBo}(1). The quasi-heredity follows from Corollary~\ref{CorgrBrBo}(3). It is easily checked that the group algebra~$\mk C_i$ is quasi-local, hence $C=\mk[J^{\le n}(\delta)]$ is properly stratified by Corollary~\ref{CorgrBrBo}(2).

Now we prove that in the remaining cases $C$ is not quasi-hereditary. By equation \eqref{eqAiH} and Lemma~\ref{NewLemStr}, $C$ will have infinite global dimension as soon as~$H$ has. For any finite group~$G$, the algebra~$\mk G$ is Frobenius and hence self-injective. In particular, the global dimension of~$\mk G$ is finite if and only if it is zero. In conclusion, when the criteria for semisimplicity of~$H$ in Lemma~\ref{LemExBH} are not satisfied, $C$ has infinite global dimension. In particular $C$ is not quasi-hereditary for any order, by \cite{CPS} or \cite[Corollary~6.6]{APT}.
\end{proof}

The following lemma can be derived from~\cite{Nakano1} with minor additional work.
\begin{lemma}\label{LemHN}
Assume that~$\mk$ is algebraically closed and~$i>1$. 
\begin{itemize}
\item The Specht modules of~$\mk\mS_i$ form a standard system if and only if~$$\begin{cases}\charr(\mk)\not\in \{2,3\}\,\,\,\mbox{ or}\\ \charr(\mk)=3 \,\mbox{ and }\,i=2.\end{cases}$$
\item The cell modules of the standardly based algebra~$\mk C_i$ of Example \ref{ExS1} form a standard system if and only if~$\charr(\mk)$ does not divide $i$.
\end{itemize}
\end{lemma}
\begin{proof}
For~$\charr(\mk)\not\in\{2,3\}$, \cite[Theorem~6.4(b)(ii)]{Nakano1} implies that the Specht modules of~$\mk\mS_i$ form a standard system. When $\charr(\mk)=3$, $\mk\mS_2$ is semisimple and the (simple) Specht modules thus form a standard system.

Now we prove that, in the remaining cases, the Specht modules do not constitute a standard system, for any order. The trivial module~$S_1$ and the sign module~$S_2$ are both Specht modules.
Assume $\charr(\mk)=2$, then~$S_1\cong S_2$, contradicting property (2) in the definition of a standard system in \ref{IntroFaithCov}. 
Now assume $\charr(\mk)=3$ and~$i\ge 3$. We will prove that
$$\Ext^1_{\mk\mS_i}(S_1,S_2)\not=0\not=\Ext^1_{\mk\mS_i}(S_2,S_1),$$
contradicting property (3) in the definition of a standard system. The modules $M$ and~$N$ corresponding to the extensions are given by
$$M=\langle v,w\rangle\qquad \mbox{with }\quad s_j v=v+c_j w\;\mbox{ and }\; s_j w=-w,$$
for~$s_{j}$, $j\in\{1,\ldots,i-1\}$ the generators of~$\mS_i$ and~$c_j\in\mk$ arbitrary, and
$$N=\langle v,w\rangle\qquad\mbox{with }\quad s_j w=-w+d_j v\;\mbox{ and }\; s_j v=v,$$
for~$d_j\in\mk$ arbitrary. It can be verified that in characteristic 3, there are no conditions on the coefficients $c_j,d_j$ coming from imposing the braid relations. However, when~$c_j\not= c_{j'}$ or~$d_j\not= d_{j'}$, for~$1\le j\not=j'\le i-1$, the modules do not decompose and the extensions hence do not split.

For~$\mk C_i\cong\mk[x]/(x^i-1)$, the cell modules are two-by-two non-isomorphic if and only if all $i$-th roots of~$1$ in the algebraically closed field $\mk$ are different. This is equivalent to demanding that~$\mk C_i$ is semisimple, in which case the cell modules are simple and form a standard system for any order. So a necessary and sufficient condition is that~$\charr(\mk)$ does not divide $i$.
\end{proof}

\begin{proof}[Proof of Theorem~\ref{ThmDiagram2}]
By Corollary~\ref{CorStB} and Proposition~\ref{PropExBH}, this follows from Lemmata~\ref{LemExBH} and~\ref{LemHN}.\end{proof}

\begin{rem} \label{RemnotS}
Consider the case~$\charr(\mk)\not\in[2,n]$.
We have proved that~$\mk[\cB^{\le n}(\delta)]$ is quasi-hereditary with exact Borel subalgebra. This is not a {\em strong} exact Borel subalgebra as defined in \cite{Koenig}, as $H$ is not a maximal semisimple subalgebra. In fact, generically, $\mk[\cB^{\le n}(\delta)]$ will be semisimple itself, as follows from \cite{Rui} and Theorem~\ref{ThmMor} below.
\end{rem}

\begin{rem}\label{NotKKO}
The exact Borel subalgebra~$B$ of~$C:=\mk[\cB^{\le n}(\delta)]$ does not satisfy the property
$$\Ext^k_C(C\otimes_B M,C\otimes_B N)\;\cong\;\Ext^k_B(M,N),\qquad \forall \,k\ge 2,$$
for arbitrary~$M,N\in B$-mod, which is satisfied for the exact Borel subalgebra predicted by the general theory of \cite{KKO}, see \cite[Theorems~1.1 and~10.5]{KKO}. Indeed, we observe that~$B$, and hence the extension group on the right-hand side, does not depend on $\delta$. The extension group in the left-hand side depends heavily on $\delta$.
As mentioned in Remark~\ref{RemnotS}, for generic $\delta$ the left-hand side will vanish. It is easily checked that it does not vanish for {\it e.g.} $\delta=0$. This proves that the displayed isomorphism cannot be true in general.
\end{rem}

\begin{rem}
For~$C=\mk[\FI^{\le n}]$, we have $B=H$ as $N^+=0$. So for~$\charr(\mk)\not\in[2,n]$, $\mk[\FI^{\le n}]$ is quasi-hereditary with semisimple exact Borel subalgebra. It can easily be checked that for the reversal of the order $\le_{\cQ}$, we have that~$\mk[\FI^{\le n}]$ is an exact Borel subalgebra of itself.
\end{rem}

\subsection{Morita equivalences}

\begin{thm}\label{ThmMor}
For an arbitrary field $\mk$, we have the following Morita equivalences under the respective conditions given in the table:
\vspace{-4mm}
\begin{center}
\begin{tabular}{ | l | l |  }
\multicolumn{2}{c }{}\\
\hline
Morita equivalence & condition \\ \hline\hline
$\cP_n(\delta)\;\,\,\stackrel{M}{=}\,\;\,\mk[\cP^{\le n}(\delta)]$ & $\delta\not=0$  \\ \hline
$\cB_n(\delta)\;\,\,\stackrel{M}{=}\;\,\,\mk[\cB^{\le n}(\delta)]$ & $\delta\not=0$ or~$n$ odd \\ \hline
$J_n(\delta)\;\,\,\stackrel{M}{=}\;\,\,\mk[J^{\le n}(\delta)]$ & $\delta\not=0$ or~$n$ odd   \\ \hline
$\TL_n(\delta)\,\,\stackrel{M}{=}\,\,\mk[\TL^{\le n}(\delta)]$ & $\delta\not=0$ or~$n$ odd \\ \hline
$\cB_{p+n,n}(\delta)\,\stackrel{M}{=}\,\mk[{}^p\cB^{\le n}(\delta)]$ & $\delta\not=0$ or~$p\not=0$\\ 
\hline
\end{tabular}
\end{center}
\end{thm}

\begin{rem}
Theorem~\ref{ThmMor} and Lemma~\ref{LemExBH} together determine the labelling set $\Lambda_A$ of the simple modules of the diagram algebras $A$, under the condition in the right column of the table. In almost all cases $\Lambda_A$ is known by \cite{CDDM, CellAlg, partition, Xi}.
\end{rem}

We start the proof of Theorem~\ref{ThmMor} by constructing special elements in the category algebras, which will also be essential for constructions in Part III.
\begin{lemma}\label{Lemab}
Consider algebras $C$,~$A=e_n^\ast Ce_n^\ast$ and~$i\in I_A(n)$, as in the table below.
\vspace{-4mm}
\begin{center}
\begin{tabular}{ | l | l | l |  }
\multicolumn{3}{c }{}\\
\hline
Category algebra~$C$ & Diagram algebra~$A$ & $I_A(n)$\\ \hline\hline
$\mk[\cP^{\le n}(\delta)]$& $\cP_n(\delta)$ & $\{0,1,\ldots,n\}$  \\ \hline
$\mk[\cB^{\le n}(\delta)]$& $\cB_n(\delta)$ & $\JJJ(n)$  \\ \hline
$\mk[J^{\le n}(\delta)]$& $J_n(\delta)$ & $\JJJ(n)$  \\ \hline
$\mk[\TL^{\le n}(\delta)]$& $\TL_n(\delta)$ & $\JJJ(n)$  \\ \hline
$\mk[{}^p\cB^{\le n}(\delta)]$& $\cB_{p+n,n}(\delta)$ & $\{0,1,\ldots,n\}$  \\ \hline
\end{tabular}
\end{center}
Then there are elements $a_i\in e_n^\ast Ce_i^\ast$ and~$b_i\in e_i^\ast Ce_n^\ast$ such that~$c^\ast_i:=a_ib_i\in A$ and~$e_i^\ast$ satisfy
$$\begin{cases} e_i^\ast= b_ia_i\quad \mbox{and}\quad (c^\ast_i)^2=c^\ast_i&\mbox{ if either $\delta\not=0$ or~$i+p\not=0$}\\
(c^\ast_i)^2=0=a_ib_i&\mbox{ if~$\delta=0$ and~$i+p=0$.}
\end{cases}$$
Here, we set $p=0$ in all cases, except for~${}^p\cB^{\le n}(\delta)$.
\end{lemma}
\begin{proof}
There are many different choices for~$a_i$ and~$b_i$. The constructions given below for the various cases produce different~$a_i$ and~$b_i$ in the overlapping situations.

Firstly we consider the partition category (so $p=0$) and arbitrary~$1\le i\le n$. We put
$$b_i=\{\{1,1'\},\{2,2'\},\cdots,\{i,i'\},\{i+1,i+2,\cdots,n\}\}\qquad\mbox{and}$$
$$a_i=\{\{1,1'\},\{2,2'\},\cdots,\{i-1,(i-1)'\},\{i,i',(i+1)',(i+2)',\cdots,n'\}\}.$$
We have $b_i a_i= e^\ast_i $, which implies that~$c^\ast_i:=a_ib_i$ is an idempotent. In case~$i=0$ we define $a_0\in \Hom_{\cP(\delta)}(0,n)$ and~$\overline{a}_0\in \Hom_{\cP(\delta)}(n,0)$ as the partitions into one set. When $\delta\not=0$, we set $b_0=\overline{a}_0/\delta$, which leads to~$b_0a_0=e_0^\ast$ and an idempotent~$c^\ast_0=a_0b_0$. When $\delta=0$ we set $b_0=\overline{a}_0$, in which case~$c^\ast_0:=a_0b_0$ squares to zero. The above already deals with the partition category completely, although the diagrams that will be constructed below for the other categories can also be used for the partition category when~$i\in\JJJ(n)$.

Now we consider the cases $\cB_n(\delta)$, $J_n(\delta)$ and~$\TL_n(\delta)$ (so $p=0$).
For arbitrary~$i\in\JJJ(n)$ we introduce three Brauer diagrams, which are also Temperley-Lieb and hence Jones diagrams. We consider two diagrams $\overline{a}_i$ and~$\widehat{a}_i$ with~$(n-i)/2$ caps and~$i$ propagating lines and~$a_i$ with~$(n-i)/2$ cups and~$i$ propagating lines. Note that the definition of~$\widehat{a}_i$ requires $i>0$.

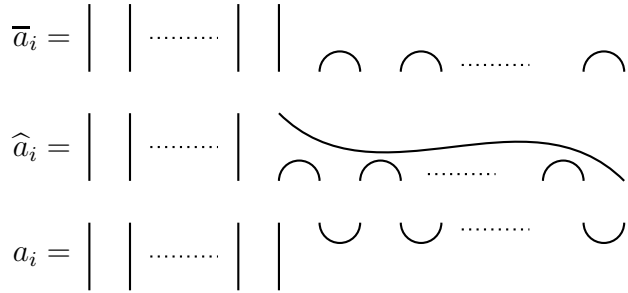
\begin{equation*}\label{defOA}
\begin{tikzpicture}[scale=0.9,thick,>=angle 90]
\begin{scope}[xshift=8cm]
\node at (0,0.5) {$\overline{a}_i=$};
\draw (.7,0) -- +(0,1);
\draw (1.3,0) -- +(0,1);
\draw [dotted] (1.6,.5) -- +(1,0);
\draw (2.9,0) -- +(0,1);
\draw (3.5,0) -- +(0,1);
\draw (4.1,0) to [out=90,in=-180] +(.3,.3) to [out=0,in=90] +(.3,-.3);
\draw (5.3,0) to [out=90,in=-180] +(.3,.3) to [out=0,in=90] +(.3,-.3);
\draw [dotted] (6.2,.1) -- +(1,0);
\draw (8,0) to [out=90,in=-180] +(.3,.3) to [out=0,in=90] +(.3,-.3);
\end{scope}
\end{tikzpicture}
\end{equation*}
\begin{equation*}\label{defBB}
\begin{tikzpicture}[scale=0.9,thick,>=angle 90]
\begin{scope}[xshift=8cm]
\node at (0,0.5) {$\widehat{a}_i=$};
\draw (.7,0) -- +(0,1);
\draw (1.3,0) -- +(0,1);
\draw [dotted] (1.6,.5) -- +(1,0);
\draw (2.9,0) -- +(0,1);
\draw (3.5,0) to [out=90,in=-180] +(.3,.3) to [out=0,in=90] +(.3,-.3);
\draw (4.7,0) to [out=90,in=-180] +(.3,.3) to [out=0,in=90] +(.3,-.3);
\draw [dotted] (5.7,.1) -- +(1,0);
\draw (7.4,0) to [out=90,in=-180] +(.3,.3) to [out=0,in=90] +(.3,-.3);
\draw (8.6,0) to [out=135,in=-45] +(-5.1,1);
\end{scope}
\end{tikzpicture}
\end{equation*}
\begin{equation*}\label{defA}
\begin{tikzpicture}[scale=0.9,thick,>=angle 90]
\begin{scope}[xshift=8cm]
\node at (0,0.5) {$a_i=$};
\draw (.7,0) -- +(0,1);
\draw (1.3,0) -- +(0,1);
\draw [dotted] (1.6,.5) -- +(1,0);
\draw (2.9,0) -- +(0,1);
\draw (3.5,0) -- +(0,1);
\draw (4.1,1) to [out=-90,in=-180] +(.3,-.3) to [out=0,in=-90] +(.3,.3);
\draw (5.3,1) to [out=-90,in=-180] +(.3,-.3) to [out=0,in=-90] +(.3,.3);
\draw [dotted] (6.2,.9) -- +(1,0);
\draw (8,1) to [out=-90,in=-180] +(.3,-.3) to [out=0,in=-90] +(.3,.3);
\end{scope}
\end{tikzpicture}
\end{equation*}

 If $\delta\not=0$, we can set $b_i:=\delta^{(i-n)/2}\overline{a}_i$ and then we have $b_ia_i=e_i^\ast$. 
 If $i\not=0$, we can set $b_i:=\widehat{a}_i$ and then we have $b_ia_i=e_i^\ast$. In either case, we automatically find that~$c^\ast_i:=a_ib_i$ is an idempotent. 
 When both $\delta=0$ and~$i=0$, we set $b:= \overline{a}_0$ and then $c^\ast_0=a_0b_0$ squares to zero.
This completes the cases $\cB_n(\delta)$, $J_n(\delta)$ and~$\TL_n(\delta)$.

The case~$\cB_{p+n,n}(\delta)$ behaves similarly. For instance, if~$p>0$ and~$\delta=0$, we can take
$$
\begin{tikzpicture}[scale=0.9,thick,>=angle 90]
\begin{scope}[xshift=4cm]
\node at (0,1) {$c_0^\ast:=$};
\draw  (1.2,0) -- +(0,2);
\draw [dotted] (1.4,1) -- +(0.6,0);
\draw  (2.2,0) -- +(0,2);
\draw (2.8,0) to [out=30, in=150] +(5.4,0);
\draw (3.4,0) to [out=32, in=148] +(4.2,0);
\draw (3.4,2) to [out=-32, in=-148] +(4.8,0);
\draw (3.9,0) [dotted] -- +(0.7,0);
\draw (3.9,2) [dotted] -- +(1.2,0);
\draw (4.9,0) to [out=80,in=100] +(1.2,0);
\draw (5.5,0) to [out=120, in=-60] +(- 2.7,2);
\draw (5.5,2) to [out=-90,in=180] +(.3,-.3) to [out=0,in=-90] +(.3,.3);
\draw[line width =2] (5.8,0) -- +(0,2);
\draw (6.5,0) [dotted] -- +(0.7,0);
\draw (6.5,2) [dotted] -- +(1.2,0);
\end{scope}
\end{tikzpicture}
$$
where there are $p$ propagating lines and~$n$ cups and caps. 
\end{proof}

\begin{proof}[Proof of Theorem~\ref{ThmMor}]
Any of the category algebras satisfies 
$$1_C=\sum_{i\in I_A(n)}e^\ast_i,$$
with $I_A(n)$ the set in Lemma~\ref{Lemab}.
By Lemma~\ref{Lemab}, the condition in the right column of the table in the theorem implies that $e_i^\ast\in Ce_n^\ast C$, for all $i\in I_A(n)$.
Consequently, we have~$1_C\in Ce_n^\ast C$, and hence $C=Ce_n^\ast C$.
The conclusion thus follows from~\ref{SecMor}.
\end{proof}

\subsection{Cellularity}We prove that the category algebras and the diagram algebras are cellular in the sense of \cite[Definition~1.1]{CellAlg}.
\subsubsection{}
For all $i,j\in\mN$, we define a $\mk$-vector space isomorphism
$$\imath:\Hom_{\cP(\delta)}(i,j)\to \Hom_{\cP(\delta)}(j,i),$$
which is determined by demanding that any partition is mapped to its ``horizontal flip'', which simply identifies the $i$ dots on the lower line of partitions in $\Hom_{\cP(\delta)}(i,j)$ with those on the upper line of partitions in $\Hom_{\cP(\delta)}(j,i)$. This clearly extends to an involutive anti-algebra morphism of~$\mk[\cP^{\le n}(\delta)]$, which furthermore restricts to involutive anti-algebra morphisms of~$\mk[\cB^{\le n}(\delta)]$, $\mk[{}^{p}\cB^{\le n}(\delta)]$ and~$\mk[\TL^{\le n}(\delta)]$. On a diagram with only propagating lines, interpreted as an element of a symmetric group, $\imath$ acts as inversion. For~$\mk[J^{\le n}(\delta)]$ we use a different involution~$\imath$, {\it viz.} $\imath$ is defined on a Jones diagram as its horizontal and vertical flip. In particular, this stabilises any Jones diagram with only propagating lines.

\begin{thm}
Consider an arbitrary field $\mk$. Let $C$ be $\mk[\cP^{\le n}(\delta)]$, $\mk[\cB^{\le n}(\delta)]$, $\mk[J^{\le n}(\delta)]$, $\mk[\TL^{\le n}(\delta)]$ or~$\mk[{}^p\cB^{\le n}(\delta)]$. Then $C$ is cellular for involution~$\imath$ when it satisfies the condition to be base stratified in Theorem~\ref{ThmDiagram1}.
\end{thm}
\begin{proof}
We take an arbitrary extension of the partial order of \ref{DefL} on $L=\sqcup_i L_i$ to a total order, again denoted by $\le$. The ideals $C^{\ge (l,p)}$ for~$(l,p)\in L$ of the standardly based algebra~$C$ then form a chain of ideals, which is a refinement of the chain of ideals $J_i=Cf_i C$. We introduce the space $J_i'$ spanned by all partitions where exactly $i$ subsets contain dots of the upper and lower row, or the intersection of that space with the relevant category algebra. Then we have 
$$J_i = J_i'\oplus J_{i-1}\qquad\mbox{and}\qquad \imath(J_i')=J'_i.$$
By extending the cellular structures on $He_i^\ast$ in Examples \ref{ExS1} and \ref{ExS2} to~$C$, we construct similar decompositions 
$$C^{\ge (l,p)}=\left(C^{\ge (l,p)}\right)'\;\oplus\; C^{\ge (l,p_1)}$$ for the refinement, where $p_1$ is the (unique) maximal~$q\in L_l$ with~$q<p$.
Together with the above decompositions, it then follows easily that \cite[Definition~2.2]{CellQua} is satisfied for the refined chain and~$C$ is hence cellular, see also \cite[Lemma~1.2.4]{JieDu}.
\end{proof}

As $\imath(e_n^\ast)=e_n^\ast$, \cite[Proposition~4.3]{StructureCell} implies the following property.
\begin{cor}\label{CorCell}
The algebras $\cP_n(\delta)$, $\cB_n(\delta)$, $\cB_{r,s}(\delta)$ and~$\TL_n(\delta)$ are cellular with respect to the restriction of~$\imath$. If the polynomials $x^i-1$, for~$i\in\JJJ(n)$, split over $\mk$, the algebra~$J_n(\delta)$ is cellular with respect to the restriction of~$\imath$.
\end{cor}
\begin{prop}\label{PropCMod}
Consider any of the cellular algebras $A$ in Corollary~\ref{CorCell} with the corresponding category algebra~$C$ in Lemma~\ref{Lemab}. The cell modules of~$A$ are labelled by~$L_A:=L_C$ in Theorem~\ref{ThmDiagram2}, and are given by
$$W_A(i,p):=e_n^\ast W_C(i,p)\quad\mbox{with}\quad W_C(i,p)=C\otimes_B W^0_i(p)\quad\mbox{ for }\; i\in I_A\;\mbox{ and }\, p\in L_i.$$
\end{prop}
\begin{proof}
We use Proposition \ref{CellB}. Since 
$$e_n^\ast C\otimes_B W^0_i(p)\not=0,$$
for any $W^0_i(p)$, the proposition follows from the proof of~\cite[Proposition~4.3]{StructureCell}.
\end{proof}


\part{Faithfulness of covers}
\label{Faith}

\section{Covers of the diagram algebras}\label{SecCov}
The two main methods to prove that cell modules of a cellular algebra form a standard system are construction of a certain quasi-hereditary 1-cover ({\it viz.} a cover-Schur algebra), see \cite{Nakano1, Henke}; and cellular/base stratification, see \cite{HHKP} and Section \ref{Sec8}. 

When the diagram algebra~$A$ in Lemma~\ref{Lemab} is Morita equivalent to~$C$, the base stratification of~$C$ solves the problem also for~$A$.
In this section, we use a combination of both approaches above, for when~$A$ is not Morita equivalent to~$C$. We prove that the base stratified algebra~$C$ is almost always a $1$-faithful cover of~$A$, in the sense of Section \ref{IntroFaithCov}. Even though $C$ might not be quasi-hereditary, this determines when the cell modules of~$A$ form a standard system.

\subsection{Connection with the coarse filtration}
In Section \ref{SecDefCat}, we defined the diagram algebras as the endomorphism algebras of the objects of the corresponding category. Remarkably, we can reconstruct the respective categories starting from the diagram algebra, by using the elements~$\{a_i,b_i,c_i^\ast\,|\, i\in I_A(n)\}$ of Lemma~\ref{Lemab}.
\begin{thm}\label{thmcoarseP}
Consider an arbitrary field $\mk$, $\delta\in\mk$ and let $\cA$ be $\cP(\delta)$, $\cB(\delta)$, ${}^p\cB(\delta)$, $J(\delta)$ or~$\TL(\delta)$. For~$n\in\mZ_{>0}$, set $A:=\End_{\cA}(n)$. For any $i,j\in I_A(n)$, we have an isomorphism
$$\Hom_{\cA}(j,i)\;\tilde{\to}\; \Hom_A(Ac_i^\ast,Ac_j^\ast);\quad x\mapsto \alpha_x\;\mbox{ with}\quad \alpha_x(c_i^\ast)=a_i xb_j,\quad\mbox{for~$x\in \Hom_{\cA}(j,i)$}.$$Moreover, the composition of such morphisms on both sides agrees contravariantly. \end{thm}
As a special case we have the following corollary, giving an alternative description of the category algebra~$C$ connected to the diagram algebra~$A$.
\begin{cor}\label{CorCatAlgAlt}
Maintain the notation of Theorem~\ref{thmcoarseP} and consider the category algebra~$C:=\mk[\cA^{\le n}]$. We have an isomorphism of left $A$-modules
$$Ac^\ast_j\;\stackrel{\sim}{\to}\; \Hom_{\cA}(j,n)=e_n^\ast Ce_j^\ast,\quad c_j^\ast\mapsto a_j,\qquad\mbox{for }\; j\in I_A(n),$$
and an isomorphism of algebras
$$C\;\cong\; \End_A(\,\bigoplus_{j\in I_A(n)} Ac^\ast_j\,)^{\op}.$$
\end{cor}
\subsubsection{}Before we prove the theorem, we elaborate on the idempotents $c_i^\ast$.
The partition algebra has a filtration by two-sided ideals, known as the {\em coarse filtration}:
$$0=J_0\subsetneq J_1\subsetneq J_2\subsetneq \cdots\subsetneq J_{n}\subsetneq J_{n+1}=\cP_n(\delta),$$
see {\it e.g.} \cite[Lemma 4.6]{Xi}.
Here, $J_i$ is the ideal spanned by those partitions of 
\begin{equation}\label{SSS}
S=S_1\sqcup S_2=\{1,2,\ldots,n\}\sqcup\{1',2',\ldots, n'\},
\end{equation} 
where at most $i-1$ of the subsets contain elements of both $S_1$ and~$S_2$, for each $1\le i\le n+1$. Note that we have $J_i=Ac^\ast_{i-1}A$ for~$1\le i\le n+1$.
Hence, $J_i$ is an idempotent ideal, if either $i\not=1$ or~$\delta\not=0$. When $\delta=0$, we have $J_1^2=0$.

\begin{rem}
When $\delta\not=0$, the coarse filtration is an exactly standardly stratifying chain, even quasi-heredity when~$\charr(\mk)\not\in[2,n]$, which leads to the corresponding properties of~$\cP_n(\delta)$ in Theorem~\ref{ThmC}. 
\end{rem}

\subsubsection{}By restricting the coarse filtration of~$\cP_n(\delta)$, we obtain coarse filtrations of~$\cB_n(\delta)$, $\cB_{r,s}(\delta)$, $J_n(\delta)$ and~$\TL_n(\delta)$.
These coarse filtrations are based on the number of propagating lines. The corresponding idempotents (or the nilpotent element) generating the ideals are also given by~$\{c_i^\ast\}$ in Lemma~\ref{Lemab}.

Now we start the proof of Theorem~\ref{thmcoarseP} with the following lemma.
\begin{lemma}\label{LemAc00}
Let $A$ be one of the diagram algebras in Theorem~\ref{thmcoarseP}, with~$n$ even in case~$A$ is $\cB_n(\delta)$, $J_n(\delta)$ or~$\TL_n(\delta)$. We have an isomorphism
$$\Hom_A(Ac_0^\ast,A)\;\,\tilde{\to}\;\,c_0^\ast A,\quad\alpha\mapsto \alpha(c_0^\ast).$$
\end{lemma}
\begin{proof}
We have a monomorphism
$$\phi:\Hom_A(Ac_0^\ast,A)\;\,\hookrightarrow\;\, A,\quad\alpha\mapsto \alpha(c_0^\ast).$$
First we claim that~$c_0^\ast A$ is contained in the image of~$\phi$. We have the $A$-bimodule isomorphism 
$$Ac_0^\ast\otimes_{\mk} c_0^\ast A\;\tilde\to \; A c_0^\ast A;\qquad ac_0^\ast \otimes c_0^\ast b\mapsto ac_0^\ast b,$$
which is clear by construction, or follows from the cellular structure. In particular, for any $x\in c_0^\ast A$, the above gives rise to a morphism $\alpha_x$ from~$Ac_0^\ast$ to~$A$, given by $\alpha_x(ac_0^\ast)=ax$. Hence, $\phi(\alpha_x)=x$, which proves the claim.

Now we will prove that the image of~$\phi$ is contained in $c_0^\ast A$. If $c_0^\ast$ is an idempotent, this follows from $\alpha(c_0^\ast)=c_0^\ast\alpha(c_0^\ast)$. Consider therefore $(c_0^\ast)^2=0$ and some element~$a\in A$ in the image of~$\phi$, so $a=\alpha(c_0^\ast)$ for some $\alpha$.
We prove that~$a\in c_0^\ast A$ by considering cases separately. We will use the conventions of \ref{IntroC} for the elements $c_i^\ast$.

Firstly, consider $A=\cP_n(0)$.  Observe that~$a$ must satisfy $c_1^\ast a=a$, since $c_1^\ast c_0^\ast=c_0^\ast$. Hence, $a$ must be a linear combination of partitions of \eqref{SSS} where $S_2=\{1',\ldots, n'\}$ is contained in a subset. So $a=a_1+a_2$, where $a_1$ is a linear combination of partitions where $S_2$ is a subset (hence $a_1\in c_0^\ast A$) and~$a_2$ is a linear combination of partitions which have a subset which strictly contains $S_2$. We define $u=u_1-u_2$, with~$u_1,u_2$ partitions of \eqref{SSS} as:
\begin{equation}\label{equ}u_1:=\{\{1,2,\ldots,n,1'\},\{2',3',\ldots,n'\}\},\quad u_2:=\{\{1,2,\ldots,n,2',3',\ldots,n'\},\{1'\}\}.
\end{equation}
Since $uc_0^\ast=0=u a_1$, we must have $ua_2=0$. However, $u_1a_2$ is a linear combination of partitions where $\{1'\}$ is strictly contained in a subset, whereas $u_2a_2$ is a linear combination of partitions where $\{1'\}$ is a subset. Hence we must have $u_1a_2=0=u_2a_2$. However, the partition $v$ consisting of one set satisfies $vu_1a_2=a_2$. This shows that~$a_2=0$ and hence that~$a\in c_0^\ast A$.

Consider $A$ equal to~$ \cB_n(0)$, $J_n(0)$ or~$\TL_n(0)$ with~$n$ even. As the case~$n=2$ is straightforward, we assume $n\ge 4$. We now have that~$c_2^\ast a=a$. This means $a$ is a linear combination of diagrams with top row of the form 
\begin{equation}
\label{diagfora}
\begin{tikzpicture}[scale=0.9,thick,>=angle 90]

\begin{scope}[xshift=8cm]
\draw (2.8,0.8) -- +(0,0.2);
\draw (3.4,0.8) -- +(0,0.2);
\draw (4,1) to [out=-90,in=-180] +(.3,-.3) to [out=0,in=-90] +(.3,.3);
\draw (5.2,1) to [out=-90,in=-180] +(.3,-.3) to [out=0,in=-90] +(.3,.3);
\draw [dotted] (6.1,.9) -- +(1,0);
\draw (7.5,1) to [out=-90,in=-180] +(.3,-.3) to [out=0,in=-90] +(.3,.3);

\end{scope}
\end{tikzpicture}
\end{equation}
where the first two dots can be arbitrarily connected.
Consider also the diagram
\begin{equation}\label{Defw}
\begin{tikzpicture}[scale=0.9,thick,>=angle 90]

\begin{scope}[xshift=8cm]
\node at (-3,0.5) {$w:=$};
\draw (-1.7,1) to [out=-90,in=-180] +(.3,-.3) to [out=0,in=-90] +(.3,.3);
\draw (-1.7,0) to [out=50, in=-130] +(1.2,1);
\draw (-1.1,0) to [out=90,in=-180] +(.3,.3) to [out=0,in=90] +(.3,-.3);
\draw (.1,0) -- +(0,1);
\draw (.7,0) -- +(0,1);
\draw (1.3,0) -- +(0,1);
\draw [dotted] (1.6,.5) -- +(1,0);
\draw (2.9,0) -- +(0,1);
\draw (3.5,0) -- +(0,1);

\end{scope}
\end{tikzpicture}
\end{equation}
Since $wc_0^\ast=c_0^\ast$ we must also have $wa=a$. This implies that the top rows of the diagrams~\eqref{diagfora}, appearing in $a$, must have only caps. It now follows easily that~$a\in c_0^\ast A$.

Now set $A:=\cB_{r,r}(0)$ and observe that $a=c_1^\ast a$. We focus on the case~$r>1$ and define
\begin{equation}\label{Defv}
\begin{tikzpicture}[scale=0.9,thick,>=angle 90]
\begin{scope}[xshift=4cm]
\node at (0,1) {$v:=$};
\draw (2.8,0) to [out=50, in=130] +(3.3,0);
\draw (2.8,2) to [out=-30, in=-150] +(6,0);
\draw  (3.4,0) -- +(0,2);

\draw [dotted] (3.9,1) -- +(1,0);
\draw  (5.5,0) -- +(0,2);
\draw[line width =2] (5.8,0) -- +(0,2);
\draw  (6.7,0) -- +(0,2);
\draw [dotted] (7.1,1) -- +(0.6,0);
\draw  (8.2,0) -- +(0,2);
\draw  (8.8,0) to [out=120, in=-60] +(-2.7,2);
\end{scope}
\end{tikzpicture}
\end{equation}
which satisfies $vc_0^\ast=c_0^\ast$.
The conclusion follows as for the previous case.
\end{proof}

\begin{proof}[Proof of Theorem~\ref{thmcoarseP}]
First we prove that $\Hom_{\cA}(j,i)\to \Hom_A(Ac_i^\ast,Ac_j^\ast)$ is indeed an isomorphism. We distinguish four different cases.

i) When $c_i^\ast$ and~$c_j^\ast$ are idempotents, the proposed map is obviously well-defined and has inverse $\alpha\mapsto b_i \alpha(c_i^\ast)a_j$.

ii) Now assume that~$i=0$ and~$(c_0^\ast)^2=0$, but $j>0$, so that~$c_j^\ast$ is an idempotent. By Lemma~\ref{LemAc00}, it suffices to prove that
$$\phi:\;\Hom_{\cA}(j,0)\;{\to}\; c_0^\ast Ac_j^\ast;\qquad x\mapsto a_0 xb_j,$$
is an isomorphism.
An inverse to~$\phi$ is constructed by mapping any diagram $d$ in $c_0^\ast Ac_j^\ast$ to~$d'a_j$ where $d'$ is the diagram obtained from~$d$ by omitting the top row (essentially forgetting the diagram $a_0$), proving that~$\phi$ is an isomorphism.

iii) The case where $c_i^\ast$ is an idempotent, but $c_j^\ast=c_0^\ast$ is not, follows similarly, by using $\Hom_{A}(Ac_i^\ast,Ac_j^\ast)\cong c_i^\ast Ac_j^\ast$. 

iv) Finally, assume~$i=j=0$ and~$(c_0^\ast)^2=0$. Using the monomorphism $Ac_0^\ast\hookrightarrow A$ and the isomorphism in Lemma~\ref{LemAc00} shows that the image of the injective composition
$$\Hom_{A}(Ac_0^\ast,Ac_0^\ast)\hookrightarrow \Hom_A(Ac_0^\ast, A)\stackrel{\sim}{\to} c_0^\ast A$$
is the one-dimensional space $c_0^\ast A\cap Ac_0^\ast=\mk c_0^\ast$. So both $\Hom_{A}(Ac_0^\ast,Ac_0^\ast)$ and $\Hom_{\cA}(0,0)$ are one-dimensional and it follows that the morphism in the lemma is an isomorphism.

Now take $x\in \Hom_{\cA}(j,i)$ and~$y\in \Hom_{\cA}(k,j)$, so $xy\in \Hom_{\cA}(k,i)$. We claim that~$$\alpha_y\circ\alpha_x=\alpha_{xy}.$$
We have $\alpha_{y}\circ\alpha_x(c_i^\ast)=\alpha_y(a_i xb_j)$. We can take $x'\in \Hom_{\cA}(n,i)$ such that~$x' a_j=x$, then
$$\alpha_y(a_i xb_j)=a_ix' \alpha_y(c_j^\ast)=a_ix' a_jyb_k=a_ixyb_k,$$
proving the claim. \end{proof}

\subsection{Covers}\label{SecCovers}

\begin{thm}\label{ThmCov}Let $C$ be any of the category algebras in Lemma~\ref{Lemab} and~$A=e^\ast_nCe^\ast_n$ the corresponding diagram algebra.
Then~$C$ is a cover of~$A$.
\end{thm}
\begin{proof}
The condition in~\ref{SecSecCover} follows from Corollary~\ref{CorCatAlgAlt}.
\end{proof}
In most cases, the cover $C$ is Morita equivalent to~$A$, by Theorem~\ref{ThmMor}. Now we focus on the remaining cases. Consider the full subcategory~$\overline{\cF}$ of~$C$-mod of modules which admit a filtration with sections $\overline{\Delta}(\mu)$, $\mu\in\Lambda$. Recall from Section~\ref{SecCover} the exact functor
$$F:=e_n^\ast C\otimes_C-\cong e_n^\ast- \;:C\mbox{-mod}\;\to\; A\mbox{-mod}.$$
\begin{thm}\label{Thm0Cover}
Let $A$ be $\cP_n(0)$ with~$n>2$; $\cB_n(0)$, $J_n(0)$ or~$\TL_n(0)$ with~$n$ even and~$n>4$; or~$A=\cB_{r,r}(0)$ with~$r>2$.
For all $M,N\in \FFF$, the functor~$F$ induces isomorphisms
$$\Hom_{C}(M,N)\;\,\tilde\to\;\, \Hom_A(FM,FN)\quad\mbox{ and }\,\quad \Ext^1_{C}(M,N)\;\,\tilde\to\;\, \Ext^1_A(FM,FN).$$
Hence, the cover $C$ is $1$-faithful.
\end{thm}

\begin{rem}
\label{SchurGen} Consider one of the algebras $C$ in Theorem~\ref{ThmDiagram2} under the condition that its cell modules form a standard system. Due to the base stratification, the Schur algebra of~$C$ can be obtained similarly to \cite{Paget, HHKP, Henke}. The corresponding Schur algebra of~$C$ is also naturally a Schur algebra of~$A$ in case~$C$ and~$A$ are Morita equivalent, but also under the conditions in Theorem~\ref{Thm0Cover}. This will be worked out in more detail elsewhere.
\end{rem}

We make preparations for the proof of the theorem. The left exact functor
$$G:=\Hom_{A}(e_n^\ast C,-):\; A\mbox{-mod}\;\to\; C\mbox{-mod},$$
is right adjoint to~$F$. Hence we have the adjoint (unit) natural transformation~$$\eta:\Id\to G\circ F.$$

\begin{lemma}\label{LemGFiso}
Let $A$ be $\cB_n(0)$, $J_n(0)$ or~$\TL_n(0)$, for~$n$ even, $\cP_n(0)$ or~$\cB_{r,r}(0)$.
\begin{enumerate}
\item For any $M$ in $C${\rm-mod}, $\eta_M$ induces an isomorphism $M\cong G\circ F(M)$ if and only if
$$e_0^\ast M\;\to\; \Hom_A(Ac^\ast_0,e_n^\ast M):\quad\, v\mapsto \beta_v,\;\;\mbox{ where }\,\beta_v(c_0^\ast)=a_0v,$$
is an isomorphism, with~$a_0$ from Lemma~\ref{Lemab}.
\item For any $M$ in $C${\rm-mod}, we have $\cR_1G\circ F(M)=0$ if and only if
$$\Ext^1_A(Ac_0^\ast,e_n^\ast M)=0.$$
\end{enumerate}
\end{lemma}
\begin{proof}
We evaluate $\eta$ on $M$, to get the following morphism
of~$C$-modules:
$$\eta_M:M\to G\circ F(M)\stackrel{\sim}{\to}\Hom_A(e_n^\ast C,e_n^\ast M);\;\quad v\mapsto \alpha_v,\;\;\alpha_v(c)=cv.$$
This restricts to vector space morphisms
$$\eta^i:\;\,e_i^\ast M\to\Hom_A(e_n^\ast Ce_i^\ast,e_n^\ast M)\quad\mbox{for}\quad i\in I_A.$$
We also introduce
$$\rho^i:\;\, \Hom_A(e_n^\ast Ce_i^\ast,e_n^\ast M)\to e_i^\ast M;\;\quad \alpha\mapsto b_i\alpha(a_i),$$
with~$a_i,b_i$ as in Lemma~\ref{Lemab}. Then $\eta^i$ and~$\rho^i$ are mutually inverse if $i\not=0$.

Thus we find that~$\eta_M$ is an isomorphism if and only if~$\eta^0$ is. Part (1) then follows from applying the isomorphism
$\Hom_A(e_n^\ast Ce_0^\ast, e_n^\ast M)\cong \Hom_A(Ac^\ast_0,e_n^\ast M)$ from Corollary~\ref{CorCatAlgAlt}.

To prove part (2), we observe that we have
$$\cR_1G\cong \Ext^1_A(e_n^\ast C,-)\cong \Ext^1_A(\bigoplus_{j\in I_A}Ac_j^\ast,-)\cong \Ext^1_A(Ac_0^\ast,-),$$
by Corollary~\ref{CorCatAlgAlt}, since $Ac_j^\ast$ is projective for~$j\not=0$. \end{proof}

The proof of the lemma also gives the following result.
\begin{cor}\label{Corin}
Let $A$ be as in Lemma \ref{LemGFiso}. For~$i\not=0$, we have
$$\Hom_A(Ac_i^\ast,e_n^\ast M)\;\cong\; e_i^\ast M,\quad\mbox{for any $M\in A${\rm -mod}.}$$
\end{cor}

\subsubsection{}\label{IntroC}
We fix some of the $c_i^\ast$ in Lemma \ref{Lemab}. For~$\cB_n(0)$, $J_n(0)$ or~$\TL_n(0)$ with~$n$ even, take
$$
\begin{tikzpicture}[scale=0.9,thick,>=angle 90]
\begin{scope}[xshift=4cm]
\node at (0.4,0.5) {$c_0^\ast=$};
\draw (1.9,0) to [out=90,in=-180] +(.3,.3) to [out=0,in=90] +(.3,-.3);
\draw (3.1,0) to [out=90,in=-180] +(.3,.3) to [out=0,in=90] +(.3,-.3);
\draw (4.3,0) to [out=90,in=-180] +(.3,.3) to [out=0,in=90] +(.3,-.3);
\draw [dotted] (5.2,.1) -- +(1,0);
\draw (6.6,0) to [out=90,in=-180] +(.3,.3) to [out=0,in=90] +(.3,-.3);

\draw (1.9,1) to [out=-90,in=-180] +(.3,-.3) to [out=0,in=-90] +(.3,.3);
\draw (3.1,1) to [out=-90,in=-180] +(.3,-.3) to [out=0,in=-90] +(.3,.3);
\draw (4.3,1) to [out=-90,in=-180] +(.3,-.3) to [out=0,in=-90] +(.3,.3);
\draw [dotted] (5.1,.9) -- +(1,0);
\draw (6.6,1) to [out=-90,in=-180] +(.3,-.3) to [out=0,in=-90] +(.3,.3);

\node at (9.4,0.5) {$c_2^\ast=$};
\draw (10.9,0) -- +(0,1);
\draw (11.5,0) to [out=90,in=-180] +(.3,.3) to [out=0,in=90] +(.3,-.3);

\draw (12.7,0) to [out=90,in=-180] +(.3,.3) to [out=0,in=90] +(.3,-.3);
\draw [dotted] (13.6,.1) -- +(1,0);
\draw (15,0) to [out=90,in=-180] +(.3,.3) to [out=0,in=90] +(.3,-.3);
\draw (16.2,0) to [out=140,in=-40] +(-4.7,1);
\draw (12.1,1) to [out=-90,in=-180] +(.3,-.3) to [out=0,in=-90] +(.3,.3);
\draw (13.3,1) to [out=-90,in=-180] +(.3,-.3) to [out=0,in=-90] +(.3,.3);
\draw [dotted] (14.1,.9) -- +(1,0);
\draw (15.6,1) to [out=-90,in=-180] +(.3,-.3) to [out=0,in=-90] +(.3,.3);
\end{scope}
\end{tikzpicture}
$$
$$
\begin{tikzpicture}[scale=0.9,thick,>=angle 90]
\begin{scope}[xshift=4cm]
\node at (9.4,0.5) {$c_4^\ast=$};
\draw (10.9,0) -- +(0,1);
\draw (11.5,0) -- +(0,1);
\draw (12.1,0) -- +(0,1);
\draw (12.7,0) to [out=90,in=-180] +(.3,.3) to [out=0,in=90] +(.3,-.3);
\draw [dotted] (13.6,.1) -- +(1,0);
\draw (15,0) to [out=90,in=-180] +(.3,.3) to [out=0,in=90] +(.3,-.3);
\draw (16.2,0) to [out=133,in=-47] +(-3.5,1);

\draw (13.3,1) to [out=-90,in=-180] +(.3,-.3) to [out=0,in=-90] +(.3,.3);
\draw [dotted] (14.1,.9) -- +(1,0);
\draw (15.6,1) to [out=-90,in=-180] +(.3,-.3) to [out=0,in=-90] +(.3,.3);
\node at (16.6,0) {.};
\end{scope}
\end{tikzpicture}
$$
For~$A=\cB_{r,r}(0)$, take
$$
\begin{tikzpicture}[scale=0.9,thick,>=angle 90]
\begin{scope}[xshift=4cm]
\node at (1.8,1) {$c_0^\ast:=$};
\draw (2.8,0) to [out=30, in=150] +(6,0);
\draw (2.8,2) to [out=-30, in=-150] +(6,0);
\draw (3.4,0) to [out=32, in=148] +(4.8,0);
\draw (3.4,2) to [out=-32, in=-148] +(4.8,0);
\draw (3.9,0) [dotted] -- +(1.2,0);
\draw (3.9,2) [dotted] -- +(1.2,0);
\draw (5.5,0) to [out=90,in=-180] +(.3,.3) to [out=0,in=90] +(.3,-.3);
\draw (5.5,2) to [out=-90,in=180] +(.3,-.3) to [out=0,in=-90] +(.3,.3);
\draw[line width =2] (5.8,0) -- +(0,2);
\draw (6.5,0) [dotted] -- +(1.2,0);
\draw (6.5,2) [dotted] -- +(1.2,0);

\node at (10.8,1) {$c_1^\ast:=$};
\draw (11.8,0) to [out=30, in=150] +(5.4,0);
\draw (12.4,0) to [out=32, in=148] +(4.2,0);
\draw (12.4,2) to [out=-32, in=-148] +(4.8,0);
\draw (12.9,0) [dotted] -- +(0.7,0);
\draw (12.9,2) [dotted] -- +(1.2,0);
\draw (13.9,0) to [out=80,in=100] +(1.2,0);
\draw (14.5,0) to [out=120, in=-60] +(- 2.7,2);
\draw (14.5,2) to [out=-90,in=180] +(.3,-.3) to [out=0,in=-90] +(.3,.3);
\draw[line width =2] (14.8,0) -- +(0,2);
\draw (15.5,0) [dotted] -- +(0.7,0);
\draw (15.5,2) [dotted] -- +(1.2,0);
\draw  (17.8,0) -- +(0,2);

\end{scope}
\end{tikzpicture}
$$
$$
\begin{tikzpicture}[scale=0.9,thick,>=angle 90]
\begin{scope}[xshift=4cm]
\node at (1.8,1) {$c_2^\ast:=$};
\draw  (2.8,0) -- +(0,2);
\draw (3.4,0) to [out=32, in=148] +(4.2,0);
\draw (4,2) to [out=-32, in=-148] +(3.6,0);
\draw (3.9,0) [dotted] -- +(0.7,0);
\draw (4.4,2) [dotted] -- +(0.7,0);
\draw (4.9,0) to [out=80,in=100] +(1.2,0);
\draw (5.5,0) to [out=120, in=-60] +(- 2.1,2);
\draw (5.5,2) to [out=-90,in=180] +(.3,-.3) to [out=0,in=-90] +(.3,.3);
\draw[line width =2] (5.8,0) -- +(0,2);
\draw (6.5,0) [dotted] -- +(0.7,0);
\draw (6.5,2) [dotted] -- +(0.7,0);
\draw  (8.2,0) -- +(0,2);
\draw  (8.8,0) -- +(0,2);
\node at (9.1,0) {.};
\end{scope}
\end{tikzpicture}
$$
For~$\cP_n(0)$, take 
$$c_0^\ast=\{\{1,2,\ldots,n\},\{1',2',\ldots,n'\}\},\quad c_1^\ast=\{\{1,1',2',\ldots,n'\},\{2,3,\ldots,n\}\},\mbox{ and}$$ 
$$c_2^\ast=\{\{1,1'\},\{2,2',3',\ldots,n'\},\{3,4,\ldots,n\}\}.$$

We also introduce $\gamma=\gamma_A$, where $\gamma_A=1$ for~$A=\cB_{r,r}(0)$ or~$\cP_n(0)$ and~$\gamma_A=2$ otherwise.

\begin{lemma}\label{Ac0A}
Consider $A$ equal to~$\cB_n(0)$, $J_n(0)$ or~$\TL_n(0)$ for~$n$ even and~$n\ge 4$, $\cP_n(0)$ with~$n\ge 2$, or~$\cB_{r,r}(0)$ with~$r\ge 2$. We have
$$\Hom_A(Ac_0^\ast, A/Ac_0^\ast A)=0.$$
\end{lemma}
\begin{proof}
This is a stronger version of Lemma~\ref{LemAc00}. It can be proved using the same arguments. Consider $\alpha:Ac_0^\ast\to A/Ac_0^\ast A$ and~$a\in A$ such that 
$a+Ac_0^\ast A=\alpha(c_0^\ast)$. 

For~$\cB_n(0)$, $J_n(0)$ or~$\TL_n(0)$, we have $c_2^\ast c_0^\ast=c_0^\ast$, so we can assume that~$a\in c_2^\ast A$. As $w$ in equation~\eqref{Defw} is an idempotent satisfying~$wc_0^\ast=c_0^\ast$, we can further assume that~$wa=a$. These two conditions on $a$ show that it must be a linear combination of diagrams with~$n/2$ caps, so $a\in Ac_0^\ast A$ and hence $\alpha=0$.

The proof for~$\cB_{r,r}(0)$ is identical, by using the idempotent~$v$ in equation~\eqref{Defv}. Also the proof for~$\cP_n(0)$ works along the same lines, by using~$uc_0^\ast=0$ with~$u=u_1-u_2$ in \eqref{equ}.
\end{proof}

\begin{lemma}\label{esc0}
Let $A$ be as in Lemma \ref{Ac0A}. There exists an exact sequence
$$Ac^\ast_{2\gamma}\to Ac_{\gamma}^\ast\to Ac_0^\ast\to 0.$$
\end{lemma}
\begin{proof}
First, let $A$ be $\cB_n(0)$, $J_n(0)$ or~$\TL_n(0)$.
The map $Ac_2^\ast\tto Ac_0^\ast$ is defined as $a\mapsto ac_0^\ast$ for any $a\in Ac_2^\ast$, where surjectivity follows from $c_2^\ast c_0^\ast=c_0^\ast$. Now we determine the kernel $K$ of this epimorphism. From the structure of~$c_2^\ast$ and~$c_0^\ast$ it follows that~$K$ is the spanned by all diagrams without propagating lines and by all elements of the form $d_1-d_2$, where $d_1,d_2\in Ac_2^\ast$ are diagrams satisfying the following conditions. For~$k\in \{1,2\}$, the diagram $d_k$ has two propagating lines, connecting~$1$ to~$i_k$ and~$n$ to~$j_k$, giving four different dots $\{i_1,i_2,j_1,j_2\}$.
There is a cap in $d_k$ which connects dots $i_l$ and~$j_l$, with~$\{k,l\}=\{1,2\}$. Finally, removing these caps and propagating lines in $d_1$ and~$d_2$ yield identical~$(n-2,n-4)$-Brauer diagrams.

An example of such a $d_1-d_2$, with~$i_1=1'$, $j_1=2'$, $i_2=3'$ and~$j_2=4'$, is given by
\begin{equation*}\label{Defx}
\begin{tikzpicture}[scale=0.9,thick,>=angle 90]
\begin{scope}[xshift=8cm]
\node at (1.4,0.5) {$x:=$};
\draw (2.9,0) -- +(0,1);
\draw (3.5,0) to [out=90,in=-180] +(.3,.3) to [out=0,in=90] +(.3,-.3);
\draw (4.7,0) to [out=90,in=-180] +(.3,.3) to [out=0,in=90] +(.3,-.3);
\draw [dotted] (5.6,.1) -- +(1,0);
\draw (6.9,0) to [out=90,in=-180] +(.3,.3) to [out=0,in=90] +(.3,-.3);
\draw (8.2,0) to [out=147,in=-33] +(-4.7,1);
\draw (4.1,1) to [out=-90,in=-180] +(.3,-.3) to [out=0,in=-90] +(.3,.3);
\draw (5.2,1) to [out=-90,in=-180] +(.3,-.3) to [out=0,in=-90] +(.3,.3);
\draw [dotted] (6.1,.9) -- +(1,0);
\draw (7.5,1) to [out=-90,in=-180] +(.3,-.3) to [out=0,in=-90] +(.3,.3);
\node at (9.5,0.5) {$-$};
\draw (10.9,0)  to [out=70,in=-110] +(1.2,1);
\draw (11.5,0) to [out=90,in=-180] +(.3,.3) to [out=0,in=90] +(.3,-.3);
\draw (12.7,0) to [out=90,in=-180] +(.3,.3) to [out=0,in=90] +(.3,-.3);
\draw [dotted] (13.6,.1) -- +(1,0);
\draw (14.9,0) to [out=90,in=-180] +(.3,.3) to [out=0,in=90] +(.3,-.3);
\draw (16.2,0) to [out=136,in=-44] +(-3.5,1);
\draw (10.9,1) to [out=-90,in=-180] +(.3,-.3) to [out=0,in=-90] +(.3,.3);
\draw (13.2,1) to [out=-90,in=-180] +(.3,-.3) to [out=0,in=-90] +(.3,.3);
\draw [dotted] (14.1,.9) -- +(1,0);
\draw (15.5,1) to [out=-90,in=-180] +(.3,-.3) to [out=0,in=-90] +(.3,.3);
\end{scope}
\end{tikzpicture}
\end{equation*}

Now we claim that~$K=Ax$. That the span of all diagrams without propagating lines is in $Ax$ follows easily from multiplying~$x$ with the diagram having a cup and cap connecting the first two dots and otherwise only vertical propagating lines. For~$d_1-d_2\in K$ as above, we consider the three algebras separately. 

For~$A=\cB_n(0)$, we can consider a diagram $a\in\mS_n$, where $1$ is connected to~$i_1$, $2$ to~$j_1$, $3$ to~$i_2$ and~$4$ to~$j_2$. It then follows easily that~$a$ can be completed such that~$ax=d_1-d_2$.

For~$A=\TL_n(0)$, take an arbitrary~$d_1-d_2$ as above and consider the $n/2-2$ cups which appear in both diagrams $d_1$ and~$d_2$. It follows easily that this information determines $d_1-d_2$ uniquely, up to sign.
Now we consider the unique diagram $a\in A$ which contains those $n/2-2$ cups, the $n/2-2$ caps which appear in $c_4^\ast$ and four propagating lines. Then we find $ax=\pm (d_1-d_2)$.
Finally, the case~$A=J_n(0)$ follows similarly.

As $c_4^\ast x=x$, we have a surjection~$Ac_4^\ast\tto K$, with~$K=Ax$, proving the exact sequence.

For the two remaining algebras one proves, similarly to the above, that the kernel of~$Ac_1^\ast \tto Ac_0^\ast$ is generated by $x=c_2^\ast x$, given by
$$
\begin{tikzpicture}[scale=0.9,thick,>=angle 90]
\begin{scope}[xshift=4cm]
\node at (1.5,1) {$x:=$};

\draw (2.8,0) to [out=30, in=150] +(5.4,0);
\draw (3.4,0) to [out=32, in=148] +(4.2,0);
\draw (3.4,2) to [out=-32, in=-148] +(4.8,0);
\draw (3.9,0) [dotted] -- +(0.7,0);
\draw (3.9,2) [dotted] -- +(1.2,0);
\draw (4.9,0) to [out=80,in=100] +(1.2,0);
\draw (5.5,0) to [out=120, in=-60] +(- 2.7,2);
\draw (5.5,2) to [out=-90,in=180] +(.3,-.3) to [out=0,in=-90] +(.3,.3);
\draw[line width =2] (5.8,0) -- +(0,2);
\draw (6.5,0) [dotted] -- +(0.7,0);
\draw (6.5,2) [dotted] -- +(1.2,0);
\draw  (8.8,0) -- +(0,2);

\node at (10.5,1) {$-$};

\draw (11.8,0) to [out=30, in=150] +(5.4,0);
\draw (12.4,0) to [out=32, in=148] +(4.2,0);
\draw (11.8,2) to [out=-30, in=-150] +(6,0);
\draw (13,2) to [out=-35, in=-145] +(3.6,0);
\draw (12.9,0) [dotted] -- +(0.7,0);
\draw (13.6,2) [dotted] -- +(0.7,0);
\draw (13.9,0) to [out=80,in=100] +(1.2,0);
\draw (14.5,0) to [out=120, in=-60] +(- 2.1,2);
\draw (14.5,2) to [out=-90,in=180] +(.3,-.3) to [out=0,in=-90] +(.3,.3);
\draw[line width =2] (14.8,0) -- +(0,2);
\draw (15.5,0) [dotted] -- +(0.7,0);
\draw (15.2,2) [dotted] -- +(0.7,0);
\draw  (17.8,0) to [out=100,in=-80] +(-0.6,2);

\end{scope}
\end{tikzpicture}
$$
for~$A=\cB_{r,r}(0)$, and by
$$x:=\;\;\{\{1,1'\},\{2,3,\ldots,n\},\{2',3',\ldots,n'\}\}\;-\;\{\{1'\},\{2,3,\ldots,n\},\{1,2',3',\ldots,n'\}\},$$
for~$A=\cP_n(0)$.
\end{proof}

\begin{lemma}\label{LemDi}
For~$A$ as in Lemma~\ref{LemGFiso}, let $D_i:=e_n^\ast C\otimes_B He_i^\ast$. Then
$$D_i\,\cong\, Ac_i^\ast/Ac_{i-\gamma}^\ast Ac_i^\ast,\;\,\mbox{ for }\quad\;i\in I_A.$$
\end{lemma}
\begin{proof}
Set $f'=\sum_{j<i}e_j^\ast$ and~$f=f'+e_i^\ast$. Lemma \ref{LemABAM} and equation~\eqref{eqAiH} then imply that
$$C\otimes_B He_i^\ast\;\cong\; \left(C/Cf'C\right)f.$$
By Corollary~\ref{CorCatAlgAlt}, the $A$-module $e_n^\ast C\otimes_B He_i^\ast$ is a quotient of $Ac_i^\ast$. Furthermore, the submodule $e_n^\ast C\sum_{j<i}e_i^\ast Ce_i^\ast$ corresponds precisely to the submodule in $Ac_i^\ast$ spanned by diagram which have strictly fewer than $i$ propagating lines (subsets which contain dots on both lines). This is precisely $Ac_{i-\gamma}^\ast Ac_i^\ast$. \end{proof}

\begin{prop}\label{PropHom}
For~$A$ in Lemma \ref{Ac0A}, the unit $\eta:\Id\to G\circ F$ induces isomorphisms
$$G\circ F(\overline{\Delta}(i,\nu))\;\cong\;\overline{\Delta}(i,\nu),\qquad\mbox{for all $i\in I_A$ and~$\nu\in\Lambda_i$.}$$
\end{prop}
\begin{proof}
First we prove the case $i=0$. We have $\overline{\Delta}(0,\emptyset)\cong Ce_0^\ast$ and $F \overline{\Delta}(0,\emptyset)\cong Ac_0^\ast$ by Lemma~\ref{LemDi}. Lemma~\ref{LemGFiso}(1) applied to $M=Ce_0^\ast$ shows that $\eta_M$ is indeed an isomorphism, by Theorem~\ref{thmcoarseP} for~$i=j=0$.

For~$i>0$, we have $e_0^\ast \overline{\Delta}(i,\nu)=0$ for all $\nu\in \Lambda_i$, so by Lemma~\ref{LemGFiso}(1), we only need to prove
\begin{equation}\label{Hom000}\Hom_A(Ac_0^\ast, e_n^\ast \overline{\Delta}(i,\nu))=0.\end{equation}
Recall that~$He_{\gamma}^\ast$ is a group algebra and hence self-injective. In particular, for any simple module~$L^0(\gamma,\nu)$ of~$He_\gamma^\ast$ we have $L^0(\gamma,\nu)\hookrightarrow He_\gamma^\ast$ and consequently $e_n^\ast\overline{\Delta}(\gamma,\nu)\hookrightarrow D_\gamma$.
By Lemma~\ref{LemDi}, $D_\gamma$ is a direct summand of~$A/Ac_0^\ast A$. Hence Lemma~\ref{Ac0A} implies
$$\Hom_A(Ac_0^\ast, D_\gamma)=0.$$
 Thus the above equation implies equation~\eqref{Hom000}, concluding the case~$i=\gamma$. Finally, for~$i>\gamma$, we use Lemma~\ref{esc0}, which implies that the left-hand side of equation \eqref{Hom000} is a subspace of
$\Hom_A(Ac^\ast_\gamma,e_n^\ast \overline{\Delta}(i,\nu))$, which is equal to~$e_\gamma^\ast \overline{\Delta}(i,\nu)$
by Corollary~\ref{Corin}. Now $e_\gamma^\ast \overline{\Delta}(i,\nu)=0$ since $i>\gamma$, and equation \eqref{Hom000} is again satisfied.
\end{proof}

\subsubsection{}\label{IntroY}From now on we will assume that for~$A$ equal to~$\cB_n(0)$, $J_n(0)$ or~$\TL_n(0)$ we have $n>4$. This allows to introduce diagrams $y_1,y_2\in A$ as
$$
\begin{tikzpicture}[scale=1,thick,>=angle 90]
\begin{scope}[xshift=4cm]

\node at (7,0.5) {$y_1:=$};

\draw (7.9,0) -- +(0,1);
\draw (8.5,0) -- +(0,1);
\draw (9.1,0) to [out=60,in=-110] +(1.2,1);
\draw (9.7,0) to [out=90,in=-180] +(.3,.3) to [out=0,in=90] +(.3,-.3);
\draw [dotted] (10.6,.1) -- +(1,0);
\draw (12,0) to [out=90,in=-180] +(.3,.3) to [out=0,in=90] +(.3,-.3);
\draw (13.2,0) to [out=120,in=-60] +(-2.3,1);

\draw (9.1,1) to [out=-90,in=-180] +(.3,-.3) to [out=0,in=-90] +(.3,.3);
\draw [dotted] (12.14,.9) -- +(0.4,0);
\draw (12.6,1) to [out=-90,in=-180] +(.3,-.3) to [out=0,in=-90] +(.3,.3);
\draw (11.5,1) to [out=-90,in=-180] +(.3,-.3) to [out=0,in=-90] +(.3,.3);

\node at (15,0.5) {$y_2:=$};

\draw (15.9,0) to [out=60,in=-110] +(1.2,1);
\draw (16.5,0) to [out=60,in=-110] +(1.2,1);
\draw (17.1,0) to [out=60,in=-110] +(1.2,1);
\draw (17.7,0) to [out=90,in=-180] +(.3,.3) to [out=0,in=90] +(.3,-.3);
\draw [dotted] (18.6,.1) -- +(1,0);
\draw (20,0) to [out=90,in=-180] +(.3,.3) to [out=0,in=90] +(.3,-.3);
\draw (21.2,0) to [out=120,in=-60] +(-2.3,1);

\draw (15.9,1) to [out=-90,in=-180] +(.3,-.3) to [out=0,in=-90] +(.3,.3);
\draw [dotted] (20.14,.9) -- +(0.4,0);
\draw (20.6,1) to [out=-90,in=-180] +(.3,-.3) to [out=0,in=-90] +(.3,.3);
\draw (19.5,1) to [out=-90,in=-180] +(.3,-.3) to [out=0,in=-90] +(.3,.3);

\end{scope}
\end{tikzpicture}
$$
Similarly, for~$A=\cB_{r,r}(0)$ we will assume $r>2$, which allows to introduce
$$
\begin{tikzpicture}[scale=0.9,thick,>=angle 90]
\begin{scope}[xshift=4cm]
\node at (1.6,1) {$y_1:=$};
\draw  (2.8,0) -- +(0,2);
\draw (3.4,0) to [out=32, in=148] +(4.2,0);
\draw (3.4,2) to [out=-30, in=-150] +(4.8,0);
\draw (4.6,2) to [out=-40, in=-140] +(2.4,0);
\draw (3.9,0) [dotted] -- +(0.7,0);
\draw (4.9,2) [dotted] -- +(0.4,0);
\draw (4.9,0) to [out=80,in=100] +(1.2,0);
\draw (5.5,0) to [out=120, in=-60] +(- 1.5,2);
\draw (5.5,2) to [out=-90,in=180] +(.3,-.3) to [out=0,in=-90] +(.3,.3);
\draw[line width =2] (5.8,0) -- +(0,2);
\draw (6.5,0) [dotted] -- +(0.7,0);
\draw (6.25,2) [dotted] -- +(0.4,0);
\draw  (8.2,0) to [out=100,in=-80] +(-0.6,2);
\draw  (8.8,0) -- +(0,2);

\node at (10,1) {$y_2:=$};

\draw  (10.8,0) to [out=80, in=-100] +(0.6,2);
\draw (11.4,0) to [out=32, in=148] +(4.2,0);
\draw (10.8,2) to [out=-30, in=-150] +(6,0);
\draw (12.6,2) to [out=-40, in=-140] +(2.4,0);
\draw (11.9,0) [dotted] -- +(0.7,0);
\draw (12.9,2) [dotted] -- +(0.4,0);
\draw (12.9,0) to [out=80,in=100] +(1.2,0);
\draw (13.5,0) to [out=120, in=-60] +(- 1.5,2);
\draw (13.5,2) to [out=-90,in=180] +(.3,-.3) to [out=0,in=-90] +(.3,.3);
\draw[line width =2] (13.8,0) -- +(0,2);
\draw (14.5,0) [dotted] -- +(0.7,0);
\draw (14.25,2) [dotted] -- +(0.4,0);
\draw  (16.2,0) to [out=100,in=-80] +(-0.6,2);
\draw  (16.8,0) to [out=100,in=-80] +(-0.6,2);
\end{scope}
\end{tikzpicture}
$$
For all four algebras a direct computation shows that
\begin{equation}\label{eqyx}
x\;=\; y_1x-y_2x,
\end{equation}
with~$x$ as introduced in the proof of Lemma \ref{esc0}.

For~$A=\cP_n(0)$, we will assume that~$n>2$, and we introduce $y\in A$,
$$y:=\quad \{\{1,1'\},\{2,2'\},\{3,4,\ldots,n\},\{3',4',\ldots,n'\}\}\,-$$
$$\{\{1,1'\},\{2'\},\{3,4,\ldots,n\},\{2,3',4',\ldots,n'\}\}+\{\{1'\},\{1,2'\},\{3,4,\ldots,n\},\{2,3',4',\ldots,n'\}\}.$$
It follows immediately that~$yx=0$.

\begin{lemma}\label{LemHom01}
Let $A$ be as in Theorem~\ref{Thm0Cover}. Then
$$\Ext^1_{A}(Ac_0^\ast,Ac_0^\ast)=0.$$
\end{lemma}
\begin{proof}
First we will prove $\Hom_{A}(Ax,Ac_0^\ast)=0.$ Consider $\phi:Ax\to Ac_0^\ast$ and~$a:=\phi(c_0^\ast)$.

First, let $A$ be $\cB_n(0)$, $J_n(0)$ or~$\TL_n(0)$, the case~$\cB_{r,r}(0)$ is proved similarly. As $c_4^\ast x=x$, we have $a\in c_4^\ast A c_0^\ast$. By equation \eqref{eqyx}, we must have
$$a=y_1 a- y_2a.$$
This means that~$a$ is a linear combination of diagrams which have the $n/2$ caps of~$c_0^\ast$, the $n/2-2$ cups of~$c_4^\ast$ and another cup connecting either $1',2'$, or~$3',4'$. The only such diagram is $c_0^\ast$. However, $y_1c_0^\ast=y_2c_0^\ast$ and hence $a=0$.

For~$A=\cP_n(0)$, we have $a\in c_2^\ast Ac_0^\ast$, with~$\dim c_2^\ast Ac_0^\ast=2$. With $y$ from \ref{IntroY}, we must have $ya=0$. It follows by direct computation that~$ya=0$ for~$a\in  c_2^\ast Ac_0^\ast$ implies $a=0$.

Hence, in every case we have indeed, $\Hom_{A}(Ax,Ac_0^\ast)=0$. Recall the short exact sequence 
\begin{equation}\label{sesK}0\to Ax\to Ac_\gamma^\ast \to Ac_0^\ast\to 0\end{equation}
from the proof of Lemma~\ref{esc0}. This implies an exact sequence
\begin{equation}\label{ExtHomeq}\Hom_A(Ax,M)\to \Ext^1_A(Ac_0^\ast,M)\to 0,\end{equation}
for any $A$-module~$M$, as $Ac_\gamma^\ast$ is projective. As we established that~$\Hom_A(Ax,Ac_0^\ast)=0$, the statement follows.
\end{proof}

\begin{lemma}\label{LemHom02}
Let $A$ be as in Theorem \ref{Thm0Cover}. Then 
$$\Ext^1_A(Ac_0^\ast, Ac_{2\gamma}^\ast/Ac_{\gamma}^\ast Ac_{2\gamma}^\ast)=0.$$
\end{lemma}
\begin{proof}
We will prove $\Hom_{A}(Ax,Ac_{2\gamma}^\ast/Ac_{\gamma}^\ast Ac_{2\gamma}^\ast)=0,$ then the statement follows from~\eqref{ExtHomeq}. Consider $\phi:Ax\to Ac_{2\gamma}^\ast/Ac_{\gamma}^\ast Ac_{2\gamma}^\ast$ and~$a:=\phi(c_0^\ast)$.

First, let $A$ be $\cB_n(0)$, $J_n(0)$ or~$\TL_n(0)$, the case~$\cB_{r,r}(0)$ is proved similarly. As $c_4^\ast x=x$, we have $a\in c_4^\ast A/(Ac_2^\ast) c_4^\ast$. By equation \eqref{eqyx}, we must have
$$a=y_1 a- y_2a.$$
This means that~$a$ is represented by a linear combination of diagrams, where each contains $4$ propagating lines, the $n/2-2$ cups and caps of~$c_4^\ast$, but also either the cup connecting~$1',2'$ or~$3',4'$. This is an inconsistency, so $a=0$.

For~$A=\cP_n(0)$, the dimension of~$c_2^\ast (A/Ac_0^\ast A)c_2^\ast$ is 2. It follows quickly that no non-zero element~$a$ satisfies $ya=0$, so $a=0$.
\end{proof}

\begin{lemma}\label{ext02}
Let $A$ be as in Theorem~\ref{Thm0Cover}. Then
$$\Ext^1_A(Ac_0^\ast,Ac_\gamma^\ast/Ac_0^\ast Ac_\gamma^\ast)=0.$$
\end{lemma}
\begin{proof}
Set $M:=Ac_\gamma^\ast/Ac_0^\ast Ac_\gamma^\ast$. 
The short exact sequence \eqref{sesK} and Lemma \ref{Ac0A} imply
a short exact sequence
$$0\to \Hom_A(Ac_\gamma^\ast,M)\to\Hom_A(Ax,M)\to \Ext^1_A(Ac_0^\ast,M)\to 0.$$
We will prove that for each algebra, 
$$\dim \Hom_A(Ax,M)\;\le\; d:=\dim c_\gamma^\ast(A/Ac_0^\ast A)c_\gamma^\ast =\dim H e_\gamma^\ast,$$
proving that the extension group must vanish.

First, let $A$ be $\cB_n(0)$, $J_n(0)$ or~$\TL_n(0)$, the case~$\cB_{r,r}(0)$ is proved similarly. If $a$ is the image of~$x$ under a morphism $Ax\to M$, then, by equation \eqref{eqyx}, we have
$$ a= c_4^\ast a\quad\mbox{and}\quad a=y_1a-y_2a.$$
Hence $a$ must be represented by a linear combinations of diagrams, containing the $n/2-1$ caps of~$c_2^\ast$, the $n/2-2$ cups of~$c_4^\ast$ and a cup which either connects $1',2'$ or~$3',4'$. The dimension of the space of such $a$ is $4$ for~$\cB_n(0)$ and~$J_n(0)$ and~$2$ for~$\TL_n(0)$. In each case, imposing the actual condition that~$a=y_1a-y_2a$, leaves half of the dimensions. For each case, this yields precisely $d$.

For~$A=\cP_n(0)$, we have $d=1$ and the dimension of~$c_2^\ast (A/Ac_0^\ast A)c_1^\ast$ is $3$. The subspace of elements that are annihilated by left multiplication with~$y$ also has dimension~$1$.\end{proof}

\begin{prop}\label{PropExt}
Maintain the notation and assumptions of Theorem \ref{Thm0Cover}.
We have
$$\cR_1G\circ F(\overline{\Delta}(i,\nu))=0,\qquad\mbox{for all $i\in I_A$ and~$\nu\in\Lambda_i$.}$$
\end{prop}
\begin{proof}
By Lemma~\ref{LemGFiso}(2), it suffices to prove that
\begin{equation}\label{eqExtvan}\Ext^1_A(Ac_0^\ast,e_n^\ast\overline{\Delta}(i,\nu))=0.\end{equation}
By Lemma \ref{esc0}, this space is a subquotient of 
$$\Hom_A(Ac_{2\gamma}^\ast,e_n^\ast\overline{\Delta}(i,\nu)),$$
which is zero when~$i>2\gamma$, by Corollary \ref{Corin}. Hence we focus on $i \in\{0,\gamma,2\gamma\}$.

By Lemma~\ref{LemDi}, we have $e_n^\ast\overline{\Delta}(0,\emptyset)\cong Ac_0^\ast$, so \eqref{eqExtvan} is satisfied for~$i=0$ by Lemma~\ref{LemHom01}.

As $He_i^\ast$ is self-injective, for~$D_i$ in Lemma~\ref{LemDi}, we have a short exact sequence
$$0\to e_n^\ast \overline{\Delta}(i,\nu)\to D_i\to Q_\nu\to 0,$$
where $Q_\nu$ has a filtration with sections $e_n^\ast\overline{\Delta}(i,\nu')$ with~$\nu'\in\Lambda_i$. This gives an exact sequence
$$\Hom_A(Ac_0^\ast,Q_\nu)\to \Ext^1_A(Ac_0^\ast,e_n^\ast \overline{\Delta}(i,\nu))\to \Ext^1_A(Ac_0^\ast,D_i).$$
If $i\not=0$, the left-hand space is zero by \eqref{Hom000}. For~$i\in\{\gamma,2\gamma\}$, the right-hand side is zero by Lemmata \ref{LemHom02} and \ref{ext02}. Hence the middle term is zero and \eqref{eqExtvan} is satisfied. \end{proof}

\begin{proof}[Proof of Theorem~\ref{Thm0Cover}]
By Proposition \ref{PropExt}, we have 
\begin{equation*}\label{eqRG}
\cR_1G(FM)=\Ext^1_{A}(e_n^\ast C,FM)=0,\qquad\mbox{for any $M\in\overline{\cF}$.}
\end{equation*} Consider a short exact sequence $M_1\hookrightarrow M\tto M_2$, with~$M_i$ (and hence also $M$) in $\overline{\cF}$. Using the above vanishing of cohomology, we find a commutative diagram with exact rows
\begin{equation}\label{commdia}\xymatrix{
0\ar[r] & M_1 \ar[r]\ar[d]^{\eta_{M_1}} & M\ar[r]\ar[d]^{\eta_M} &
 M_2\ar[r]\ar[d]^{\eta_{M_2}} & 0\\
0\ar[r] & GFM_1 \ar[r] & GFM\ar[r] & GFM_2\ar[r] &0.
}\end{equation}
This implies that, if~$\eta_{M_1}$ and~$\eta_{M_2}$ are isomorphisms, so is $\eta_M$.
Proposition \ref{PropHom} can then be used to prove, by induction on the length of the filtration that we have
\begin{equation*}\label{eqeta}\eta_M:\;\;M\;\,\tilde\to\;\, G\circ F(M),
\end{equation*}
for any $M\in\overline{\cF}$. Hence, we have an isomorphism
$$\Hom_C(M,N)\cong \Hom_C(M,G\circ F( N))\cong \Hom_A(FM,FN),$$
induced by~$F$.
Now we consider arbitrary~$M_1,M_2\in\overline{\cF}$. The morphism
$$F:\;\Ext^1_C(M_2,M_1)\to  \Ext^1_A(FM_2,FM_1)$$
has left inverse induced by $G$, by \eqref{commdia}. As furthermore $F\circ G\cong \Id$, the above morphism is an isomorphism. This concludes the proof.
\end{proof}

\subsection{Cell modules and standard systems}

\begin{thm}\label{ThmCellStan}
Assume that the field $\mk$ is algebraically closed. Let $A$ be $\cB_n(0)$, $J_n(0)$ or~$\TL_n(0)$ with~$n$ even, $\cP_n(0)$ or~$\cB_{r,r}(0)$. The cell modules of the cellular algebra~$A$ form a standard system if and only if the following condition is satisfied.
\vspace{-4mm}
\begin{center}
\begin{tabular}{ | l | l |  }
\multicolumn{2}{c }{}\\
\hline
algebra~$A$& condition \\ \hline\hline
$\cP_n(0)$ & $n>2$ and~$\charr(\mk)\not\in\{2,3\}$  \\ \hline
$\cB_n(0)$ &  $n\not\in \{2,4\}$ and~$\charr(\mk)\not\in\{2,3\}$ \\ \hline
$J_n(0)$ & $n\not\in \{2,4\}$ and~$\charr(\mk)\not\in[2,n/2]$ \\ \hline

  $\TL_n(0)$ & $n\not\in \{2,4\}$  \\ \hline
$\cB_{r,r}(0)$ & $r>2$ and~$\charr(\mk)\not\in\{2,3\}$\\ 
\hline
\end{tabular}
\end{center}
\end{thm}
First we prove the following lemma.

\begin{lemma}
\label{NegPn}
The cell modules of 
\begin{enumerate}
\item $\cB_2(0)$, $J_2(0)$, $\TL_2(0)$ and~$\cB_{1,1}(0)$ do not form a standard system, for any partial order;
\item $\cB_4(0)$, $J_4(0)$, $\TL_4(0)$, $\cP_2(0)$ and~$\cB_{2,2}(0)$ do not form a standard system.
\end{enumerate}
\end{lemma}
\begin{proof}
First we prove part (1), we have $$\TL_2(0)\cong\cB_{1,1}(0)\cong \mk[x]/(x^2).$$ The only cellular structure on this algebra gives two cell modules which are isomorphic. If $\charr(\mk)\not=2$, we have 
$$\cB_2(0)\cong J_2(0)\cong \mk[x]/(x^2)\oplus \mk,$$
so the result follows as above. If $\charr(\mk)=2$, the algebra~$\cB_2(0)\cong J_2(0)$ has two isomorphic cell modules, induced from the sign and trivial~$\mk\mS_2$-module, by Proposition \ref{PropCMod}.

Now we prove part (2). We have $c_{2\gamma}^\ast=1$ and, as $c_\gamma^\ast$ is an idempotent, equation~\eqref{sesK} implies $$ \Hom_A(Ax,A/Ac_\gamma^\ast A)\cong \Ext^1(Ac_0^\ast, A/Ac_\gamma^\ast A).$$
As $Ac_0^\ast\cong W(0,\emptyset)$ and~$A/Ac_\gamma^\ast A$ has a filtration with sections of the form $W(2\gamma,\nu)$ with~$\nu\in L_{2\gamma}$, the right-hand side must be zero in order to have a standard system for~$(L,\le)$. However, we claim that there exists a non-zero morphism 
$$\phi: Ax\to A/Ac_\gamma^\ast A.$$ For~$A=\TL_4(0)$, we have $A/Ac_2^\ast A\cong \mk$ and we can set $\phi(x)= 1$. For~$A=\cP_2(0)$ we have $A/Ac_1^\ast A\cong\mk\mS_2$ and we can set $\phi(x)=1-s$, for~$s$ the generator of $\mS_2$. The other cases are left as an exercise.
 \end{proof}

\begin{proof}[Proof of Theorem~\ref{ThmCellStan}]
By Proposition \ref{PropCMod}, under the conditions in Theorem~\ref{Thm0Cover}, the cell modules of~$A$ form a standard system if and only if the cell modules of~$C $ form a standard system. The necessary and sufficient condition for the latter is given in Theorem~\ref{ThmDiagram2}.

For the remaining cases, {\it i.e.} when Theorem~\ref{Thm0Cover} is not applicable, the cell modules do not form a standard system by Lemma~\ref{NegPn}. 
\end{proof}


\appendix

\section{Stratified algebras}\label{TheApp}

\subsection{Homological stratification}
We have the following alternative characterisations of standardly and/or exactly stratifying ideals.
\begin{lemma}\label{LemStu}
Consider an idempotent ideal~$J=AeA$ in $A$.
\begin{enumerate}
\item The ideal~$J$ is standardly stratifying (${}_AJ$ is projective) if and only if
\begin{itemize}
\item multiplication induces an $A$-bimodule isomorphism $Ae\otimes_{eAe}eA\;\tilde\to\;J$, and
\item the left $eAe$-module~$eA$ is projective.
\end{itemize}
\item The ideal~$J$ is exactly stratifying ($J_A$ is projective)  if and only if
\begin{itemize}
\item multiplication induces an $A$-bimodule isomorphism $Ae\otimes_{eAe}eA\;\tilde\to\;J$, and
\item the right~$eAe$-module~$Ae$ is projective.
\end{itemize}
\end{enumerate}

\end{lemma}
\begin{proof}
We prove part (1), as part (2) then follows from considering~$A^{\op}$.
If $Ae\otimes_{eAe}eA\;\tilde\to\;J$, then the left $A$-module~$J$ is induced from the left $eAe$-module~$eA$. If the latter is projective then ${}_AJ$ is also projective, so the ideal~$J$ is standardly stratifying.
If $J$ is standardly stratifying, the conclusion follows from the proof of~\cite[Remark~2.1.2(b)]{CPSbook}.
\end{proof}

An {\em exact stratification of an algebra} $A$ is a chain \eqref{stratchain} of ideals
such that~$J_{i}/J_{i-1}$ is exactly stratifying in $A/J_{i-1}$, see \ref{SecStratId2}(0), for~$1\le i\le  m$. For a chain \eqref{stratchain} we consider the algebras~$A^{(i)}$ as in Remark \ref{DefAi} and we identify $A^{(i)}$-modules with~$f_{i+1}Af_{i+1}$-modules having trivial~$f_{i+1}J_{i}f_{i+1}$-action. 
The following principle is well-known.
\begin{lemma}\label{NewLemStr}
Consider an algebra~$A$ with an exact stratification \eqref{stratchain}. For any $A^{(l)}$-module~$M$ and~$A^{(j)}$-module~$N$, with~$0\le l\le j\le m-1$, let $\widetilde{M}:=Af_{l+1}\otimes _{f_{l+1}Af_{l+1}}M$ and~$\widetilde{N}:=Af_{j+1}\otimes_{f_{j+1}Af_{j+1}}N$.
Then
$$\Ext^k_{A}(\widetilde{M},\widetilde{N})\;\cong\; \delta_{jl}\,\Ext^k_{A^{(l)}}(M,N),\qquad\mbox{for any }\; k\in\mN.$$
\end{lemma}
\begin{proof}
First we claim that~$J_{l}\widetilde{M}=0=J_{l}\widetilde{N}$. It suffices to show that~$f_{l}$ acts as zero. Now $f_{l}$ is included in $f_{j+1}J_{l}f_{j+1}$, since $l\le j$. Action of~$f_{j+1}J_{l}f_{j+1}$ on $\widetilde{N}$ gives
$$f_{j+1}J_{l}f_{j+1}\otimes _{f_{j+1}Af_{j+1}} N=0.$$
This observation and \cite[equation (2.1.2.1)]{CPSbook} imply
$$\Ext^k_A(\widetilde{M}, \widetilde{N})\;\cong\; \Ext^k_{A/J_{l}}(\widetilde{M}, \widetilde{N}).$$
Consider the functors 
$$\Upsilon_i=(A/J_i)f_{i+1}\otimes_{A^{(i)}}-\;:A^{(i)}\mbox{{\rm-mod}}\to A/J_i\mbox{{\rm-mod}}.$$
By Lemma~\ref{LemStu}(2) these are exact. 
As $A/J_{l}$-modules we have $\widetilde{M}\cong\Upsilon_lM$,  and as $A/J_{j}$-modules 
$\widetilde{N}\cong\Upsilon_jN.$
Since the exact functor~$\Upsilon_l$ is left adjoint to the exact functor~$f_{l+1}-$, we have
$$\Ext^k_{A/J_{l}}(\widetilde{M}, \widetilde{N})\;\cong\; \Ext^k_{A^{(l)}}\left(M,f_{l+1}\Upsilon_j  N\right).$$
Since we have
$$f_{l+1}(A/J_{l}) f_{j+1}\;=\;\begin{cases} 0 &\mbox{for }\,  l<j, \\ A^{(l)} &\mbox{for }\, l=j, \end{cases}$$
we find $f_l\Upsilon_l\cong\Id$ and $f_l\Upsilon_j=0$ if $l<j$, which concludes the proof. \end{proof}

\subsection{Equivalence of ring and module theoretic definitions}
\label{AppEq}
In case the quasi-order $\preceq$ is total, Definitions \ref{DefA1} and~\ref{DefA2} agree. We give an overview of where this is proved, and use the numbering corresponding to Definitions \ref{DefA1} and~\ref{DefA2}.
\begin{enumerate}
\item This is~\cite[Theorem~2.2.3]{CPSbook}.
\item By \cite[Proposition~7]{Frisk}, Definition~\ref{DefA2}(2) is equivalent to demanding that Definition~\ref{DefA2}(1) holds both for~$A$ and~$A^{\op}$. By definition the same relation holds between Definition~\ref{DefA1}(1) and~(2). So this case follows from the above case~(1).
\item Both in Definitions~\ref{DefA1} and~\ref{DefA2} we find that going from (1) to (3) only corresponds to restricting from quasi-orders to orders, see Remark \ref{RemSSS}. This case thus also follows from (1).
\item As above, this case follows from case (2) by going from quasi-orders to orders. Alternatively one can apply results in \cite[Theorem~5]{Dlab}.
\item This is \cite[Theorem~3.6]{CPS}.
\end{enumerate}


\subsection{Comparison with Kleshchev's terminology}\label{AppKl}
For any class $\cB$ of Noetherian, positively graded, connected algebras, Kleshchev introduced in \cite[Section~6]{AffineK} the notions of~$\cB$-standardly stratified, $\cB$-properly stratified and~$\cB$-quasi-hereditary algebras. Even though the aim of~\cite{AffineK} is to study infinite dimensional and graded algebras, the definitions also cover all the cases in Definition~\ref{DefA1}. Therefore, we assume that every grading is reduced to the zero component and introduce the classes $\cS\subset \cL\subset \cD$, where $\cD$ contains the finite dimensional unital algebras, $\cL$ the quasi-local algebras and~$\cS$ the semisimple algebras.
Then we have the following identification between our notions and those in \cite{AffineK}:
\vspace{-4mm}

\begin{center}
\begin{tabular}{ | l | l| l | l |l |}
\multicolumn{4}{c }{}\\
\hline
&& $\cB=\cD$& $\cB=\cL$ & $\cB=\cS$ \\ \hline\hline
$\cB$-standardly && standardly  &  strongly standardly&quasi-\\
 stratified && stratified   &  stratified &hereditary\\ \hline
 
$\cB$-properly  & &exactly standardly  &  properly&quasi-\\ 
stratified &&stratified & stratified & hereditary\\ 
\hline
$\cB$-quasi && properly  &  properly&quasi-\\ 
-hereditary &&stratified  &  stratified&hereditary\\ \hline
\end{tabular}
\end{center}

This follows by comparing \cite[Definitions 6.1 and 6.2]{AffineK} with~\ref{SecStratId2}, using Lemma~\ref{LemStu}.

\subsection*{Acknowledgement}
The research was supported by Australian Research Council Discover-Project Grant DP140103239 and an FWO postdoctoral grant.

The authors thank Steffen K\"onig, Julian K\"ulshammer and Volodymyr Mazorchuk for very useful comments on the first version of the manuscript.

\begin{flushleft}
	K. Coulembier\qquad \url{kevin.coulembier@sydney.edu.au}
	
	School of Mathematics and Statistics, University of Sydney, NSW 2006, Australia
	
	\medskip
	
	R.~B.~Zhang\qquad \url{ruibin.zhang@sydney.edu.au}
	
	School of Mathematics and Statistics, University of Sydney, NSW 2006, Australia 
\end{flushleft}

\end{document}